\documentclass[12point]{amsart}

\title[Extensions]{Extensions of C*-algebas by a  small ideal}
\author{Huaxin Lin
and Ping Wong Ng}

\newtheorem{thm}{Theorem}[section]  
\newtheorem{lem}[thm]{Lemma}  
\newtheorem{prop}[thm]{Proposition} 
\newtheorem{cor}[thm]{Corollary}
\newtheorem{rmk}[thm]{Remark}
\newtheorem{df}[thm]{Definition}
\usepackage[usenames,dvipsnames]{color}

\newcommand{\A}{{A}}
\newcommand{\B}{{B}}
\newcommand{\C}{\mathcal{C}}
\newcommand{\D}{{D}}
\newcommand{\E}{\mathcal{E}}
\newcommand{\F}{\mathcal{F}}
\newcommand{\G}{\mathcal{G}}

\newcommand{\ep}{\epsilon}
\newcommand{\dt}{\delta}
\newcommand{\hm}{homomorphism}
\newcommand{\U}{\mathcal{U}}
\newcommand{\N}{\mathbb N}

\newcommand{\T}{\mathcal{T}}

\newcommand{\I}{\mathcal{I}}

\newcommand{\K}{\mathcal{K}}

\newcommand{\W}{\mathcal{W}}

\newcommand{\Z}{\mathcal{Z}}

\newcommand{\Aff}{\text{Aff}}

\newcommand{\M}{\mathbb{M}}
\newcommand{\Tt}{\mathbb{T}}
\newcommand{\af}{\alpha}

\newcommand{\zo}{{\mathcal Z}_0}

\newcommand{\R}{\mathbb{R}}

\newcommand{\ZI}{\mathbb{Z}}

\newcommand{\Ext}{\textbf{Ext}}

\newcommand{\diag}{{\rm diag}}

\newcommand{\wilog}{without loss of generality}
\newcommand{\Wlog}{Without loss of generality}

\newcommand{\red}{\textcolor{red}}
\newcommand{\blue}{\color{blue}}

\newcommand{\Green}{\color{Green}}

\newcommand{\Brown}{\color{Brown}}

\newcommand{\beq}{\begin{eqnarray}}
\newcommand{\eneq}{\end{eqnarray}}
\newcommand{\tforal}{\,\,\,\text{for\,\,\,all}\,\,\,}

\newcommand{\la}{\langle}
\newcommand{\ra}{\rangle}
\newcommand{\andeqn}{\,\,\,{\rm and}\,\,\,}
\newcommand{\rforal}{\,\,\,{\rm for\,\,\,all}\,\,\,}
\newcommand{\CA}{$C^*$-algebra}
\newcommand{\SCA}{$C^*$-subalgebra}

\numberwithin{equation}{section}

\usepackage{amsmath, amssymb}

\begin{document}

\begin{abstract}
We classify all essential extensions of the form
$$0 \rightarrow \W \rightarrow \D \rightarrow A \rightarrow 0$$
where $\W$ is the  unique separable simple C*-algebra
with a unique tracial state,
with finite nuclear dimension and with $K_i(\W)=\{0\}$ ($i=0,1$)
which satisfies the Universal Coefficient theorem (UCT),
and $A$ is a separable
amenable $\W$-embeddable C*-algebra which satisfies the UCT. 
We actually prove more general results.

We also classify a class of amenable \CA s which have only one proper closed ideal $\W.$

\end{abstract}

\maketitle

\section{Introduction}

Motivated by the goal of classifying all essentially normal operators using
Fredholm indices, Brown--Douglas--Fillmore (BDF) 
classified all extensions of the form
$$0 \rightarrow \K \rightarrow \D \rightarrow C(X) \rightarrow 0$$
where $X$ is a compact subset of the plane, and, later for all compact metric space $X$
(\cite{BDFOriginal}, \cite{BrownUCT1}, \cite{BrownUCT2}; see
also \cite{BDFAnnals}). 

 The \CA\, $\K$ is perhaps the  simplest non-unital simple \CA. 
In recent developments of the classification of separable simple amenable \CA s, however, some other seemingly 
nice non-unital simple \CA s  arise. One  {{piquant}}
example is
{{$\W,$}} which was first studied by Razak (\cite{Razak}),
and is a non-unital  separable simple \CA\ with a unique tracial state and $K_i(\W)=\{0\},$ $i=0,1.$
It is in fact stably projectionless. 
 It is proved in \cite{ElliottGongLinNiu} that $\W$ is the only 
separable stably projectionless simple \CA\,  with finite nuclear dimension satisfying the UCT which has 
said properties.  It is also algebraically simple.    
Moreover,  
as we will later elaborate,  
$\W$ has another very nice feature shared with ${\mathcal K},$ namely that\
the corona algebra $\C(\W)=M(\W)/\W$ is a purely
infinite simple \CA. 
A natural question is 
whether  one can classify  essential extensions of the following form:
\beq\label{Ext-c-1}
0\to \W \to E 
{\rightarrow} C(X)\to 0.
\eneq

Since $K_i(\W) = 0$ for $i =0, 1$, one immediately realizes that $KK^1(C(X), \W)
= 0$.  However, as we will see soon, there are many nontrivial essential 
extensions of $\W$ by $C(X)$ and a variety of unitary equivalence classes of
these essential extensions.  In other words, the classification of these
essential extensions will not follow from the usual stable KK theory.  

Other questions also naturally emerge. For example, how many extensions 
have the form
\beq
0\to \W\to E\to \W\to 0?
\eneq
Or more generally, can one classify all the essential extensions of the form
\beq\label{Ext-a-1}
0\to \W\to E
{\rightarrow} A\to 0
\eneq 
for some  general class of separable amenable \CA s $A$? 

{{As mentioned above,}}
the classification
will not follow from the usual stable $KK$-theory.
As one may expect, some restrictions on $A$ will be inevitably added.
If  $A$ is a separable amenable \CA, then, by \cite{KirchbergPhillipsEmbed}, 
$A$ can always be embedded into 
$O_2,$ the separable purely infinite simple \CA\, in the UCT class which 
has trivial $K_i$-group ($i=0,1$).  Since $M(\W)/\W$ is simple purely infinite, 
$O_2$ can be embedded into 
$M(\W)/\W.$ This immediately implies that, for the aforementioned \CA s $A,$ 
essential extensions by $\W$ always exist. 
In order to have some nice description of a class of extensions, like the ones in \eqref{Ext-a-1},
one may at least want to have some trivial essential extensions, i.e., 
those  essential extensions $E$, in \eqref{Ext-a-1}, which split.
However, unlike the classical case, this is in general hopeless. 
Note that if \eqref{Ext-a-1} is a trivial extension then it induces
a *-embedding of $A$ into $M(\W).$ 
But $M(\W)$ has a faithful tracial state, which is the
extension by the unique tracial state of $\W.$
This implies that $A$ has a faithful tracial state.   So we will 
assume that $A$ has a faithful tracial state.   Moreover, 
one may also want to have some diagonal trivial extensions of the form \eqref{Ext-a-1}.
The conventional way to do this is to allow $A$ to be embeddable into $\W.$ 
We will then present a classification of these extensions (see Theorem \ref{MTnoncom}).

Recent successes in the theory of classification of simple \CA s also make
it  impossible to resist 
the attempt to   
classify at least some non-simple \CA s. 
{{It is  an ambitious and challenging task.  At this stage, 
our experiments will be limited into the situation that $K$-theory 
is still manageable and  we will 
avoid the cases that tracial information becomes  non-traceable.}}
One of the goals 
of this research is to classify  some amenable \CA s which has only one ideal $\W.$ 
So these \CA s also have the form \eqref{Ext-a-1}.  Since we assume that $\W$ is the only ideal, 
$A$ will be a separable simple amenable \CA. 
As discussed above, 
we will assume that $A$ is embeddable into $\W,$  {{and}}
$A$ is a stably projectionless simple \CA.
{{Let us point out that for  any separable amenable \CA\, $A$  which 
has a faithful tracial state and satisfies the UCT, $A\otimes {\mathcal Z}_0$ is $\W$ embeddable,
where ${\mathcal Z}_0$ is the unique separable simple \CA\, with a unique tracial state  which 
satisfies the UCT 
such that $K_0({\mathcal Z}_0)=\ZI,$ $K_0({\mathcal Z}_0)_+=\{0\},$ $K_1({\mathcal Z}_0)=\{0\}$ and has finite nuclear dimension
(so $K_*(A)=K_*(A\otimes {\mathcal Z}_0)$ and $T(A)=T(A\otimes {\mathcal Z}_0)$).}

Denote by $\E$ the class of 
\CA s which are essential extensions of the form \eqref{Ext-a-1} such that $A$ is any separable 
simple stably projectionless \CA s with $K_0(A)={\rm ker}\rho_A,$ and, as customary,  $A$ has finite
nuclear dimension and satisfies the UCT.  Note that, in the definition of
the class $\E$, we do not fix the quotient algebra $A.$
We will show that, when $E_1$ and $E_2$ are two such \CA s, then $E_1\cong E_2,$
if and only if they have isomorphic Elliott invariants (see Theorem \ref{MTextbyW} below).

For the remainder of this introduction, we elaborate on some aspects that
were earlier alluded to. 
Perhaps one reason for the success of the BDF theory was that their
multiplier algebra $\mathbb{B}(l_2)$ and corona algebra $\mathbb{B}(l_2)/\K$
have particularly nice and simple structure.  
Among other things, $\mathbb{B}(l_2)$ has real rank zero {{(it is in fact a 
a von Neumann algebra),} and strict comparison, and $\mathbb{B}(l_2)/\K$ is simple
purely infinite.  For example, the the BDF--Voiculescu result, which roughly
says that all essential extensions are absorbing (\cite{Voiculescu};  \cite{ArvesonDuke}), 
would not be true if
the Calkin algebra $\mathbb{B}(l_2)/\K$ were not simple.
{{We may further note that, even in the case that the ideal is stable, as long as the corona algebra
is not simple, Kasparov's $KK^1$ cannot be used to classify these essential extensions 
up to unitary equivalence.}}

Recall that a nonunital $\sigma$-unital simple C*-algebra $\B$
is said to have \emph{continuous scale} if
$\B$ has a sequential approximate unit $\{ e_n \}$ such
that
\begin{enumerate}
\item[(a)] $e_{n+1} e_n = e_n$ for all $n$, and
\item[(b)] for every $a \in \B_+ - \{ 0 \}$,
there exists an $N \geq 1$ such that
for all $m > n \geq N$,
$$e_m - e_n {{\lesssim}} a$$
where $e_m - e_n {{\lesssim}} a$ means that there
exists a sequence $\{ x_k \}$ in $\B$ for which
$x_k a x_k^* \rightarrow e_m - e_n$.
\end{enumerate}
(See \cite{LinContScaleI}.)

  In \cite{LinSimpleCorona} (see also \cite{LinContScaleI}), it was shown that
a simple nonunital nonelementary $\sigma$-unital
 C*-algebra $\B$ has continuous
scale if and only if the corona algebra
$\C(\B)$ is simple, {{and,}}  if and only if $\C(\B)$ is
simple purely infinite.  Simple continuous
scale C*-algebras are basic building blocks for
generalizing extension theory (see, for example,
\cite{LinExtII}, \cite{LinExtRR0III}) and have been much studied.
As alluded earlier, aside from their basic role in the theory, the extension theory of these
algebras are in themselves quite interesting.  For example, unlike
the case of the classical theory of absorbing extensions, there are
no infinite repeats, and one needs to develop
a type of nonstable absorption theory, where, among other things,
the class of a trivial
extension need not be the zero class.  More refined considerations 
are required to take into account the new K theory that arises.
Some results in this direction were first derived many years
ago (see, for example, \cite{LinExtII} and \cite{LinExtRR0III}).

As mentioned above, in the present paper, we classify a class of extensions of
the Razak algebra $\W$ which is a C*-algebra with continuous scale, $K_*(\W) = 0$ and unique trace.
{{Unlike the previous cases, our
canonical ideal $\W$ has no non-zero projections -- in fact, the opposite of being real
rank zero, $\W$ is stably projectionless.}}  We note that {{the property of real rank zero has been presented implicitly}} since the beginnings
of the subject (even though the term ``real rank zero" was invented after
the BDF papers \cite{BDFOriginal}, \cite{BDFAnnals}).  For example, the original
BDF proof of the uniqueness of the neutral element (for the case of compact subset
of the plane) was essentially the Weyl--von Neumann--Berg theorem,
and it is well known
among experts that under mild conditions on a C*-algebra $\B$, $M(\B)$ has 
a Weyl--von Neumann theorem for self-adjoint operators if and only if 
$M(\B)$ has real rank zero (\cite{ZhangWVN}).  
And this phenomenon reoccurs throughout the original
and subsequent papers. We believe that our present result is the first classification
of a class of extensions of a simple projectionless C*-algebra (in fact, the first
{{case of a simple algebra which has real rank greater than zero}}).  

This paper is part of a continuing series of papers 
following in the path of the aforementioned program.  
(See, for example, \cite{LinFullext}, 
\cite{NgNonstableAbsorb}, \cite{NgPICorExt}, \cite{NgRR0PICorona},
\cite{NgRobinExtFunctor}, \cite{KNZPICor}.)

To further illustrate the results of this research and difference from usual stable results in the C*-algebra extension theory,
let us present one of our main results  at the end the introduction. Notations and terminolodges in the statement
will be explained later in the text. 

\begin{thm}[see \ref{MTnoncom}]
Let $A$ be a separable amenable \CA\, which is $\W$ embeddable and satisfies the UCT.

(1) {{Let}}  $\tau_1, \tau_2: A\to \C(\W)$ {{be}}  two essential extensions.
{{Then}} $\tau_1\sim^u \tau_2$ if and only if $KK(\tau_1)=KK(\tau_2).$

(2)  The map
\beq
\Lambda: \Ext^u(A,\W)\to KK(A, \C(\W))\cong {\rm Hom}(K_0(A), \R)
\eneq
defined by $\Lambda([\tau])=KK(\tau)$  is a group isomorphism.

(3) An essential extension $\tau$ is trivial and diagonal if and only if $KK(\tau)=0,$
and all trivial and diagonal extensions are unitarily equivalent.
In fact, the essential trivial diagonal extensions induce the neutral
element of $\Ext^u(A, \W)$.

(4) An essential extension $\tau$ is trivial if and only if
there exist $t \in T_f(A)$ {{(see Definition \ref{Dtrace})}} and $r \in [0,1]$ such that 
$$
\tau_{*0}(x)=r \cdot r_A(t)(x)\rforal x\in K_0(A).
$$

(5) Moreover,  let $\T$ be the set of unitary equivalence classes
 of essential trivial extensions of $A$ by $\W.$
Then
\beq\nonumber
\Lambda(\T)=\{r\cdot h: r\in [0,1], h\in {\rm Hom}(K_0(A), \R)_{{T_f(A)}}\}\,\,\, {{{\rm(see\,\,Definition\,\, \ref{Drho})}.}}
\eneq

(6) All quasidiagonal essential 
extensions are  trivial and are unitarily equivalent.

(7) {{In the case   ${\rm ker}\rho_{f,A}=K_0(A),$}}   
all trivial extensions are unitarily equivalent.
Moreover, in this case, an essential extension 
$\tau$ is trivial if and only if $KK(\tau)=\{0\}.$

(8)  {{In the case ${\rm ker}\rho_{f,A}\not=K_0(A),$ }}
there are essential trivial extensions of  $A$ by $\W$ which
are not quasidiagonal, and not all essential trivial extensions 
are unitarily equivalent (see (5) above).

\end{thm}

{{{\bf Acknowledgements:}
The first named author is partially supported by a NSF grant (DMS-1665183). 
Both authors would like to acknowledge the support during their visits 
to the Research Center of Operator Algebras at East China Normal University 
which is partially supported by Shanghai Key Laboratory of PMMP, Science and Technology Commission of Shanghai Municipality (STCSM), grant \#13dz2260400 and a NNSF grant (11531003).}}

\section{Notation}

\begin{df}\label{Dmul}
{\rm For a \CA\, $\B$, $M(\B)$ denotes the multiplier algebra of
$\B$, and  $\C(\B) := M(\B)/\B$ is the corresponding corona algebra.
$$0 \rightarrow \B \rightarrow \D \rightarrow  {{C}} \rightarrow 0$$
(of $C$ by $B$)\footnote{In the literature, the terminology is
sometimes reversed and this is sometimes called an ``extension of
$B$ by $C$".}
we will work with the corresponding \emph{Busby invariant}
 which is a {{monomorphism}}
$\phi : {{C}} \rightarrow \C(\B)$.   
An extension is unital if the corresponding Busby invariant is a unital
map.  We will mainly be working with nonunital extensions. 

Let  $\phi, \psi : C \rightarrow \C(\B)$ be  two essential extensions.
We say $\phi$ and $\psi$ are \emph{(weakly) equivalent} and 
write $\phi\sim \psi$ 
if there is a partial isometry $v\in \C(\B)$ such that 
$v^*v\phi(c)=\phi(c)v^*v=\phi(c)$ and $vv^*\psi(c)=\psi(c)vv^*=\psi(c)$
for all $c\in B$ and 
\beq
v^*\phi(c)v=\psi(c)\rforal c\in C.
\eneq

{{$\Ext(\A, \B)$ denotes the collection of all 
(weak) equivalence classes
of essential extensions $\phi :  \A \rightarrow \C(\B).$}}

{{Let $\pi: M(B)\to M(B)/B=\C(B)$ be the quotient map. Throughout these notes
$\pi$  always denotes this quotient map, unless otherwise stated.}}

We say that $\phi$ and $\psi$ are \emph{unitarily equivalent} (and
write $\phi \sim^u \psi$) if
there exists a unitary $u \in M(\B)$ such that
$$\phi(c) = \pi(u) \psi(c) \pi(u)^*$$
for all $c \in C$.  

{{$\Ext^u(\A, \B)$ denotes the collection of all 
unitary equivalence classes
of 
essential extensions $\phi :  \A \rightarrow \C(\B)$.}}}
\end{df}

\begin{df}\label{Dtrace}

{\rm Let $A$ be a \CA.  Denote by $T(A)$ the tracial state space of $A$ {{(which could be an empty set),}}
given the weak* topology.
Denote by $T_f(A)$ the set of all faithful
tracial states.  It is a convex subset of $T(A).$ 
Let ${\tilde{T}}(A)$ be the cone of densely defined,
positive (norm) lower semi-continuous traces on $A$, equipped with the topology
of point-wise convergence on elements of the Pedersen ideal  ${\rm Ped}(A)$ of $A.$
Let $B$ be another \CA\, with $T(B)\not=\emptyset$ and let $\phi: A\to B$ be a \hm.
We will use then  $\phi_T: T(B)\to T(A)$ for the induced continuous affine map.

{{Let $I$ be a  (closed two-sided) ideal of $A$ and $\tau\in T(I).$ 
It is well known that $\tau$ can be uniquely extended to a  tracial state of $A$
($\tau(a)=\lim_\af \tau(ae_\af)$ for all $a\in A,$ where $\{e_\af\}$ is an approximate identity for $I.$) In what follows we will continue to use $\tau$ for the extension.
{{Also,}} when $A$ is not unital and $\tau\in T(A),$ we will use 
$\tau$ for the extension on ${\widetilde{A}}$ as well as on $M(A),$ 
the multiplier algebra of $A.$}}
}
\end{df}

\begin{df}\label{DTtilde}

{\rm Let $r\ge 1$ be an integer and $\tau\in {\tilde T}(A).$
We will continue to use $\tau$ to denote  the trace $\tau\otimes {\rm Tr}$ on
$A \otimes M_r$, where ${\rm Tr}$ is the standard
non-normalized trace on $M_r.$
Let  $S\subset {\tilde T}(A)$  be a convex subset.  Denote by ${\rm Aff}(S)$ the space of 
all continuous real affine functions on $S.$
Define (see \cite{RI})
\beq
{\rm Aff}_+(S)&=&\{f: C(S, \R)_+: f \,\, 
{{\rm affine}}, f(\tau)>0\,\,{\rm for}\,\,\tau\not=0\}\cup \{0\},\\
{\rm LAff}_+(S)&=&\{f:S\to [0,\infty]: \exists \{f_n\}, f_n\nearrow f,\,\,
 f_n\in {\rm Aff}_+(S)\}\andeqn\\
 {\rm LAff}^{\sim}(S) &=&\{f_1-f_2: f_1\in {\rm LAff}_+(S)\andeqn f_2\in
 {\rm Aff}_+(S)\}.
 \eneq
 Note $0\in {\rm LAff}_+(S).$
 For most part of this paper, $S={\tilde T}(A)$ {{or}} $S=T(A)$
in the above definition will be used.}
\end{df}

Recall that $T(A)$ is compact, and hence a compact convex set,
when $A$ is unital or $A$ is simple separable finite and has
continuous scale.
Also, when $A$ is simple, $T_f(A) = T(A)$.  

\begin{df}\label{Df-ep}
{\rm For 
{{$\delta > 0,$
we let 
$f_{\delta},  
: [0,\infty) \rightarrow [0,1]$ be the unique
continuous maps satisfying
\[
f_{\delta}(t) {{:=}}
\begin{cases}
0 & t \in [0,\dt/2] \\
1 & t \in [\delta, \infty)\\
\makebox{linear on  } & [\dt/2, \delta].
\end{cases}
\]
}}}

\end{df}

\begin{df}\label{Ddimf}
{\rm Recall  that every $\tau \in \tilde{T}(A)$ extends uniquely to a strictly lower 
semicontinuous trace
on $M(A)_+$, which we also denote by $\tau$.

For $\tau \in \tilde{T}(A)$ 
for $a \in A_+$ (or, $a\in M_m(A)_+$ for some integer $m\ge 1$),
$$
d_\tau(a):= \lim_{n\to\infty} \tau(f_{1/n}(a)).
$$
Note that $f_{1/n}(a)$ is in the Pedersen ideal of $A.$ It follows that $d_\tau(a)$ 
is  {{a}} lower semicontinuous on $\tilde{T}(A).$
(Recall, in the case $a\in M_m(A)_+,$ we continue to use  $\tau$ for $\tau\otimes Tr,$
where $Tr$ is the standard non-normalized trace on $M_m$). 
}
\end{df}


\begin{df}\label{Drho}
{\rm Let $A$ be a \CA.   
{{Let ${\rm Hom}(K_0(A), \R)$ be the set of \hm s from $K_0(A)$ to $\R.$
Denote by ${\rm Hom}(K_0(A), \R)_+$ the set of all 
\hm s $f: K_0(A)\to \R$ such that $f(x)\ge 0$ for all $x\in K_0(A)_+.$ 
Denote 
$$
{\rm ker}\rho_A=\{x\in  K_0(A): f(x)=0\rforal f\in {\rm Hom}(K_0(A),\R)_+\}.
$$ 
It is possible that ${\rm Hom}(K_0(A), \R)_+=\{0\}.$ In that case, ${\rm ker}\rho_A=K_0(A).$
 There is a \hm\, $r_A: T(A)\to {\rm Hom}(K_0(A), \R)_+$ induced 
 by $r_A(\tau)([p])=\tau(p)$ for all projections $p\in M_m({\widetilde{A}}).$ 
 The image of $r_A$ is denoted by
 ${\rm Hom}(K_0(A), \R)_{T(A)}$ (or just  ${\rm Hom}(K_0(A), \R)_T$).
 Note that, for any $\tau\in T(A),$ $\tau([1_{\widetilde{A}}])=1.$
 If $A$ is unital and exact, then, by Corollary 3.4 of \cite{BR},
 \beq
 {\rm Hom}(K_0(A), \R)_+=\{ r\cdot s: r\in \R_+, s\in {\rm Hom}(K_0(A),\R)_{T(A)}\}.
 \eneq
Let $Y$ be a locally compact metric space {{and}} $A=C_0(Y).$  Then
$$
{\rm Hom}(K_0(A), \R)_+=\{r\cdot s|_{K_0(C)}: r\in \R_+: s\in {\rm Hom}(K_0({\widetilde{A}}), \R)_{T(A)}\}.
$$
}}

{{Let  $A$ be a separable exact simple \CA. Choose a nonzero element $e\in {\rm Ped}(A)_+.$
Let $A_e={\rm Her}(e)=\overline{eAe}.$ 
Then $r_{A_e}(T(A_e))={\rm Hom}(K_0(A_e), \R)_+.$ 
By \cite{Brstable},  the embedding $\iota: A_e\to A$ induces an isomorphism $\iota_*: K_0(A_e)\cong K_0(A).$
}}
{{Then 
$$
{\rm Hom}(K_0(A),\R)_+=\{r\cdot s\circ {\iota_*}^{-1}:r\in\R_+\andeqn s\in {\rm Hom}(K_0(A_e), \R)_{T(A_e)}\}.
$$
In particular, 
$$
{\rm ker}\rho_A=\{x\in K_0(A): r_{A_e}(\tau)(\iota_*(x))=0\rforal \tau\in T(A_e)\}.
$$
}}


Denote by ${\rm Hom}(K_0(A), \R)_{T_f}=\{r\cdot r_A(\tau): r\in  [0,1],\,\,\, \tau\in T_f(A)\}.$ 
{{Define 
$$
{\rm ker}\rho_{f,A}=\{x\in K_0(A): \lambda(x)=0\rforal \lambda\in {\rm Hom}(K_0(A), \R)_{T_f}\}.
$$}}
It should be noted that ${\rm ker}\rho_A\subset {\rm ker}\rho_{f,A}\subset K_0(A).$
{{Recall if}} $A$ is simple, then $T_f(A)=T(A)$ {{(see \ref{Rhomtf} for more comments on 
${\rm Hom}(K_0(A), \R)_{T_f}.$)}}

{{Suppose that $A$ is a simple \CA\, which has continuous scale,  
every 2-quasi-trace  of $A$ is a trace   and ${\tilde T}(A)\not=\{0\}.$
 Then $T(A)$ is compact and ${\tilde T}(A)$ is a cone with the base 
$T(A).$ There is an order preserving \hm\, $\rho_A: K_0(A)\to {\rm Aff}(A)$
such that $\rho_A([p])(\tau)=\tau(p)$ for all projections $p\in M_\infty({\widetilde{A}}).$
For unital stably finite \CA s, $\rho_A$ can also similarly  defined (see Theorem 3.3 of \cite{BR}.) }}
}

\end{df}

\begin{df}\label{Dcuntz}
{\rm For a \CA\, $\D$ and for $a, b \in \D_+$,
$a \lesssim b$ means that there exists a sequence $\{ x_n \}$ in $\D$
such that $x_n b x_n \rightarrow a$.  
We write $a\sim b$ if $a\lesssim b$ and $b\lesssim a.$
{{To avoid possible confusion,  if both {{$p$ and $q$}} are projections, 
we write $p \approx  q$}}
to mean that $a$ and $b$ are Murray--von Neumann equivalent.  
For  $a \in \D_+$, we let
${\rm Her}_{\D}(a) := \overline{a \D a}$, the hereditary C*-subalgebra of
$\D$ generated
by $a$.  Sometimes, for simplicity, we write ${\rm Her}(a)$ in place of
${\rm Her}_{\D}(a)$.
Similarly, for a C*-subalgebra ${{C}} \subseteq \D$, we let
${\rm Her}_{\D}(C)$ or ${\rm Her}(C)$ denote $\overline{C \D C}$, the hereditary
C*-subalgebra of $\D$ generated by $C$.}

\end{df}

\begin{df}\label{Dcpc+fepmult}
{\rm Let {{$A$ and}}  $C$ be \CA s.  Throughout this paper,
we will write that a map $\phi : \A \rightarrow C$ is \emph{c.p.c.}
if it is linear and completely positive contractive.
Let $\F \subset \A$ be a finite subset and let $\delta > 0$.
A c.p.c. map $\psi : \A \rightarrow \C$ is said to be
\emph{$\F$-$\delta$-multiplicative} if
$\| \psi(fg) - \psi(f)\psi(g) \| < \delta$ for all $f, g \in \F$.}

\end{df}

\begin{df}
{\rm Recall that a nonunital \CA\, $\B$ has \emph{almost stable rank one}
if for any integer $m \geq 1$,  {{and}} for any hereditary C*-subalgebra
$\D \subseteq M_m(\B)$,  $\D \subseteq \overline{GL(\widetilde{D})}$.}
\end{df}

\begin{df}
Let $A$ and $B$ be two \CA s and let $\phi: A\to B$ be a \hm.
We may write $KK(\phi)$ for the element in $KK(A,B)$ induced by $\phi,$
and $KL(\phi)$ for the element in $KL(A,B)$ induced  by $\phi,$ respectively.
\end{df}

Finally, we will be a bit loose in our terminology and use the term
``extension" to refer both to an extension $0 \rightarrow B \rightarrow E
\rightarrow A \rightarrow 0$ as well as the extension algebra $E$ in
the exact sequence.

\section{Nonstable absorption}


\begin{df}\label{DBDFsum}
Let $\A$ be a separable \CA, and let $\B$ be a non-unital but $\sigma$-unital
\CA\, with continuous scale.
Let $\phi, \psi : \A \rightarrow \C(\B)$ be two 
essential extensions.  
The \emph{BDF sum} of $\phi$ and $\psi$ is defined to be 
$$(\phi \dot{+} \psi)(.) := S \phi(.) S^* + T \psi(.) T^*,$$ 
where $S, T \in \C(\B)$ are any two isometries such that
$$S S^* + T T^* \leq 1.$$

The BDF sum $\phi \dot{+} \psi$ is well-defined (i.e., independent of choice of
$S$ and $T$) up to {{weak}}  equivalence.  If, in addition, 
$\phi$ or $\psi$ is nonunital then the BDF sum $\phi \dot{+} \psi$
is well-defined up to unitary equivalence  
(e.g., see \cite{NgNonstableAbsorb},
\cite{NgRobinExtFunctor}). 
\end{df}

\begin{df}
Let $\A$ be a separable \CA, and let $\B$ be a   nonunital but $\sigma$-unital 
\CA\,
$\Ext(\A, \B)$ denotes the collection of all 
(weak) equivalence classes
of essential extensions $\phi :  \A \rightarrow \C(\B).$ 

We make $\Ext(\A, \B)$ into an abelian 
semigroup with the sum induced by the BDF sum,
i.e., for all $[\phi], [\psi] \in \Ext(\A, \B)$,
$$[\phi] + [\psi] := {{ [\phi \dot{+} \psi].}}$$

We often also call the sum on $\Ext(\A, \B)$ the \emph{BDF sum}. 
{ {Similarly, when $A$ is nonunital,  with the BDF sum, $\Ext^u(A,B)$ is also a 
semigroup.} }

\end{df}

The next result is well-known, but we provide it for the convenience
of the reader.

\begin{lem}\label{lem:largecomplement}
Let $A$ be a separable \CA, and let $B$ be a $\sigma$-unital \CA.
Let $\phi : A \rightarrow \C(B)$ be a nonunital essential extension.
Then there exists a nonzero element $c \in \C(B)_+$ such that
$$c \perp {{{\rm ran}(\phi)}}.$$
\end{lem}

\begin{proof}

If $A$ is a unital \CA, then we can simply take
$$c := 1_{\C(B)} - \phi(1_A).$$

Suppose that $A$ is nonunital. Then
$D :=\overline{\phi(A) \C(\B) \phi(A)}$ is a $\sigma$-unital proper 
hereditary C*-subalgebra of $\C(\B)$.  Then, by 
Pedersen's double annihilator theorem {{(see Theorem 7.7 of \cite{PedersenGPOTS}),}} 
$D^{\perp}$ is nonzero, and hence, we can take $c \in D^{\perp}_+ \setminus \{ 0 \}$.
\end{proof}

\begin{prop}\label{W=S}
Let $A$ be a separable  \CA\, and $B$ be a $\sigma$-unital \CA\, such 
that $M(B)/B=\C(B)$ is purely infinite and simple.
Let $\phi_1, \phi_2: A\to \C(B)$ be  two nonunital essential extensions.

Then $\phi_1\sim \phi_2$ if and only if $\phi_1\sim^u \phi_2.$

 Moreover, if $\phi_1$ and $\phi_2$ are approximately unitarily equivalent,
then there exists a sequence of unitaries $U_n\in M(B)$ such
that 
\beq	
\lim_{n\to\infty} \pi(U_n)^*\phi_1(a)\pi(U_n)=\phi_2(a)\rforal a\in A.
\eneq
\end{prop}

\begin{proof}
This is  {{Proposition 2.1 of
\cite{NgRobinExtFunctor}.}}
\end{proof}

\begin{prop}\label{PBDFsum}
{{Let $B$ be a $\sigma$-unital simple \CA\, with continuous scale.
Let $\phi_1, \phi_2, \psi : A\to \C(B)$ be three  non-unital essential  extensions. 
Suppose that there is a unitary $U\in M_2(M(B))$ such that
\beq
\pi(U)^*(\phi_1(a)\oplus \phi_2(a))\pi(U)=\psi(a)\rforal a\in A.
\eneq
Then there is a unitary $V\in M(B)$ such that
\beq
\pi(V)^*(\phi_1(a)\dot{+}\psi_2(a))\pi(V)=\psi(a)\rforal a\in A,
\eneq
where $\phi_1\oplus \phi_2: A\to M_2(\C(B))$ is the orthogonal direct sum of $\phi_1$ and $\phi_2,$
and $\phi_1\dot{+}\phi_2$ is the BDF sum.}}

\end{prop}

\begin{proof}
Write  the BDF sum as 
\beq
\phi_1\dot{+}\phi_2=S\phi_1(.)S^*+T\phi_2(.)T^*
\eneq
where $S, T\in \C(B)$ are isometries as in \ref{DBDFsum}.
Set $p_s=SS^*$ and $p_t=TT^*.$ Then $p_s\perp p_t.$

As in \ref{DBDFsum}, there are unitaries $v_1, v_2\in \C(B)$ such that
\beq
v_1S\phi_1(a)S^*v_1^*=\phi_1(a)\andeqn v_2T\phi_2(a)T^*v_2=\phi_2(a)\rforal a\in A.
\eneq
Put $E_1=\diag(1,0)$ and $E_2=\diag(0,1).$ 
There is a partial isometry $v_3\in M_2(\C(B))$ such that $v_3^*v_3=E_1$ and 
$v_3v_3^*=E_2.$
Define $w=v_1p_s+v_3v_2p_t.$ Then 
\beq
ww^*&=&v_1p_sv_1^*+v_3v_2p_tv_2^*v_3^*\le 1_{M_2(\C(B))}\andeqn\\
w^*w&=&p_sv_1^*v_1p_s+p_tv_2^* v_3^*v_3v_2p_t=p_s+p_t.
\eneq
Moreover
\beq
&& w(S\phi_1(a)S^*+T\phi_2(a)T^*)w^*=\phi_1(a)\oplus \phi_2(a)\rforal a\in A.
\eneq
Therefore 
\beq
1_{\C(B)}\pi(U)w(S\phi_1(a)S^*+T\phi_2(a)T^*)w^*\pi(U)^*1_{\C(B)}=\psi(a)\rforal a\in A.
\eneq

Since $\psi$ is not unital, $\psi(A)^\perp\not=\{0\}.$ As $\C(B)$ is purely infinite and simple, it has 
real rank zero (see \cite{ZhPIRR0}). Let $e_0\in \psi(A)^\perp$ be a non-zero projection and $p=1_{\C(B)}-e_0.$
Let $q=w\pi(U)p\pi(U)^*w^*.$ 
Note 
\beq
q=w\pi(U)p\pi(U)^*w^* \not=w\pi(U)1_{\C(B)}\pi(U)^*w^*\le p_s+p_t\le 1_{\C(B)}.
\eneq
In other words, $1_{\C(B)}-q\not=0.$   Note $p$ and $q$ are equivalent projections in $\C(B).$ 
This implies that $1_{\C(B)}-q$ and $e_0=1_{\C(B)}-p$ are equivalent in $\C(B).$
Thus there is a partial isometry $v_0\in \C(B)$ such that 
$v_0^*v_0=1_{\C(B)}-p$ and $v_0v_0^*=1_{\C(B)}-q.$
Set $v_1=p\pi(U)wq+v_0.$ 
Then
\beq
v_1^*v_1&=&(v_0^*+p\pi(U)^*w^*q)(v_0+qw\pi(U)p)\\
&=&v_0^*v_0+p\pi(U)w^*qw\pi(U)p
=1_{\C(B)}-p+p=1_{\C(B)}\andeqn\\
v_1v_1^*&=&(v_0+qw\pi(U)p)(v_0^*+p\pi(U)^*w^*q)\\
&=&v_0v_0^*+qw\pi(U)p\pi(U)^*w^*q
=1_{\C(B)}-q+q=1_{\C(B)}.
\eneq
So $v_1$ is a unitary. Moreover, 
\beq
v_1(S\phi_1(a)S^*+T\phi_2(a)T^*)v_1^*=\phi(a)\rforal a\in A.
\eneq
Since both $\phi$ and $S\phi_1(a)S^*+T\phi_2(a)T^*$ are not unital, by \ref{W=S}, there is a unitary $V\in M(B)$
such that 
\beq
\pi(V)(S\phi_1(a)S^*+T\phi_2(a)T^*)\pi(V)^*=\phi(a)\rforal a\in A.
\eneq
This completes the proof.
\end{proof}

\begin{rmk}\label{Twosums}
{{By Proposition \ref{PBDFsum}, from now on, we will not  distinguish the usual orthogonal sums of two 
nonunital essential extensions and BDF sums of the same two nonunital essential extensions.
Proposition \ref{PBDFsum} is of course known. Let us point out the following fact:
Suppose $H_1,H_2: A\to M(B)$ are two maps such that $\pi\circ H_1$ and $\pi\circ H_2$ are 
nonunital essential extensions. Then, in general, one may not find unitaries $U, V\in M(B)$
such that ${\rm Ad}\, U\circ H_1\perp {\rm Ad}\, V\circ H_2$ even in the case that both $H_1$ and $H_2$ are diagonal maps
and $\pi\circ H_1\perp \pi\circ H_2.$}} 

\end{rmk}

\begin{thm}
Let $\A$
be a separable nuclear \CA, and let $\B$ be a $\sigma$-unital
simple \CA\, with continuous scale.
Then ${{\Ext(\A, \B)}}$ is an abelian group.
{{Moreover, if $A$  is nonunital, then ${\Ext}^u(A,B)$ is also an abelian group.}}
\label{thm:ExtIsAGroup}
\end{thm}

\begin{proof}
This is  {{Theorem 2.10 of \cite{NgRobinExtFunctor}).}}
See also  Theorem 3.5 of \cite{NgNonstableAbsorb}.
{{The second part follows from \ref{W=S}.}}
\end{proof}

\begin{thm}\label{Tgroup}
Let $A$ be a separable nuclear \CA\, and let $B$ be a $\sigma$-unital simple \CA\, with 
continuous scale.
Suppose that $\phi, \psi: A\to C(B)$ are two monomorphisms with
$\phi$ nonunital.

Then 
\beq
\phi\sim^u \psi\oplus \psi_0
\eneq
for some nonunital monomorphism $\psi_0: A\to C(B).$
\end{thm}

\begin{proof}
This is  Proposition 2.7 of \cite{NgRobinExtFunctor}.  
\end{proof}
\section{Quasidiagonality}

\begin{df}\label{Dqdunit}
Let $\B$ be a nonunital but $\sigma$-unital \CA.  A sequence $\{ b_n \}$
of norm one elements in $\B_+$ is said to be a \emph{system of quasidiagonal units}
if the following statements are true:
\begin{enumerate}
\item $b_m \perp b_n = 0$ for all $m \neq n$.
\item If $\{ x_n \}$ is a bounded sequence in $\B$ such that
$x_n \in \overline{b_n \B b_n}$ for all $n$, then
the sum $\sum x_n$ converges in the strict topology on $M(\B)$.
\end{enumerate}
\end{df}

Note that every $\sigma$-unital nonunital \CA\, has a system
of quasidiagonal units (e.g., see Lemma 2.2 of  \cite{NgRobinQuasidiagonal}).

The first result is an exercise in the strict topology.

\begin{lem}\label{Lpermutation}
Let  $B$ be a separable  nonunital \CA\, which has almost stable rank one
and let $C$ be a separable \CA.
Suppose that $\{b_n\}$ {{is}} a system of quasidiagonal units and 
$\phi_n: C\to \overline{b_nBb_n}$ is a sequence of c.p.c~maps.
For any permutation $\lambda: \N\to \N,$ 
$\sum_{n=1}^\infty \phi_{\lambda(n)}(c)$ converges strictly for all $c\in C$ and 
there exists a unitary $U\in M(B)$ such that
\beq
U^*\left(\sum_{n=1}^\infty \phi_{\lambda(n)}(c)\right)U=\sum_{n=1}^\infty \phi_n(c)\rforal c\in C.
\eneq

\end{lem}


\begin{df}

Let $\B$ be a $\sigma$-unital nonunital \CA.

An 
element $x \in M(\B)$ is said to be
\emph{diagonal},
if there exists a system $\{ b_n \}$ of quasidiagonal units in $\B$, and
there exists a bounded sequence $\{ x_n \}$ for which
$x_n \in \overline{b_n \B b_n}$ for all $n$, such that
$$x = \sum x_n.$$

An {{element}} $x \in M(\B)$ is said to be \emph{(generalized) quasidiagonal}
if $x$ is a sum of a 
diagonal 
element
with an element of $\B$.

A collection $\mathcal{S} \subseteq M(\B)$ is \emph{(generalized) quasidiagonal}
if there exists a single system of quasidiagonal units with respect to which all
the 
{{elements}} in $\mathcal{S}$ can be simultaneously (generalized) 
quasidiagonalized.
\end{df}

\begin{df}
Let $\A$ be a separable \CA\, and $\B$ a $\sigma$-unital nonunital \CA.

An extension $\phi : \A \rightarrow \C(\B)$ is said to be \emph{(generalized) 
quasidiagonal} if $\pi^{-1}(\phi(\A))$ is a (generalized) quasidiagonal 
collection of operators.
\label{df:QDExt}
\end{df}

For the rest of this paper, unless it is clearly false, when we write
``quasidiagonal", we mean generalized quasidiagonal.  


It is easy to prove the following analogue of a classical
quasidiagonality result: 

\begin{prop}
Let $\A$ be a separable \CA, and $\B$ a nonunital but $\sigma$-unital
\CA.

Suppose that $\phi : \A \rightarrow \C(\B)$ is a 
{{quasidiagonal}}
extension such that $\phi$ can be lifted to a c.p.c. map ${{\Phi': A \rightarrow M(\B)}}$
(and $\pi\circ \Phi'=\phi$).

Then there exist a system $\{ b_n \}$ of quasidiagonal units, and,  for {{each}} $n$,
a c.p.c. map $\phi_n : \A \rightarrow \overline{b_n \B b_n}$ such that
$
\phi=\pi\circ \Phi,
$
where 
$\Phi : \A \rightarrow M(\B)$
is the c.p.c. map defined by 
$$\Phi := \sum \phi_n.$$

Moreover, $\{ \phi_n \}$ is asymptotically multiplicative, i..e., for all
$a, b \in \A$,
$$\| \phi_n(ab) - \phi_n(a) \phi_n(b) \| \rightarrow 0\,\,\,  {\rm as}\,\ n\to\infty.$$

In the above setting, we often write $\Phi = \bigoplus_{n=1}^{\infty} \phi_n$.
\label{prop:QDExt} 
\end{prop}

\begin{proof}
This is Proposition 2.5 of \cite{NgRobinQuasidiagonal}.
\end{proof}

\begin{thm}\label{Tappt=qd}
Let $A$ be a separable \CA\, and $B$ be a $\sigma$-unital nonunital
simple \CA.
Then 
the pointwise norm limit of quasidiagonal extensions of $B$ by $A$
 is quasidiagonal.

In particular,
let $h: A\to \C(B)$ be {{an essential}} quasidiagonal {{extension}}
and 
let $\sigma: A\to \C(B)$ be an essential extension.
If there exists a sequence of unitaries $U_n\in M(B)$ such 
that
\beq
\lim_{n\to\infty}\pi(U_n)^* h(a)\pi(U_n)=\sigma(a)\rforal a\in A.
\eneq
Then $\sigma$ is an quasidiagonal extension.

If, in addition, $B$ has continuous scale and 
the extension $h$ is nonunital, then the unitaries
can be taken to simply be in $\C(\B)$. 
\end{thm}

\begin{proof}
This is Theorem 3.7 of \cite{NgRobinQuasidiagonal} together with
the present paper's Proposition \ref{W=S}.

\end{proof}

{Quasidiagonality was first defined by Halmos \cite{Halmos10} in 1970.
There is a long history of K-theoretical characterizations of quasidiagonality,
going back to BDF's observation that an essentially normal operator is
quasidiagonal if and only if it induces the zero element in 
$Ext$ (\cite{BDFOriginal}).  BDF were essentially the first to recognize that
quasidiagonal extensions might be approached by K-theory, and another one of
their fundamental results was that (in their setting) 
limits of trivial extensions correspond to quasidiagonal 
extensions (\cite{BDFAnnals}).  
L. G. Brown pursued this further in \cite{BrownUCT2}.  Further developments
in the study of quasidiagonality can be found in 
\cite{SalinasQuasi}, \cite{VoiculescuQuasi1},  and \cite{VoiculescuQuasi2}.
Schochet proved that stable quasidiagonal extensions are the same as limits
of stable trivial extensions and can be characterized by $Pext(K_*(A), K_*(B))$
if $A$ is assumed to be nuclear, quasidiagonal relative to $B$ and satisfying
the
University Coefficient Theorem (\cite{SchochetQuasi}).  
More recent 
developments in the general nonstable case, with
additional regularity assumptions on $B$, and a historical summary,
can be found in \cite{LinExtQuasidiagonal}.  
We will be implicitly using ideas with its origins in the above paper.
Starting with the next result, we will be presenting various K theoretic
conditions and characterizations for quasidiagonality in our setting.
A good example of this is Proposition \ref{QuasidiagonalKTh}.}

\begin{thm}
Let $\A$ be a separable nuclear \CA\,  {{which satisfies the UCT}} and 
let $\B$ be a nonunital separable simple \CA\, with continuous scale.
Suppose that there exists a nonunital essential quasidiagonal extension
$\sigma : \A \rightarrow \C(\B)$ such that $KL(\sigma) = 0$.
If $\phi : A \rightarrow \C(B)$ is a nonunital essential
extension such that $KL(\phi) = 0$
then $\phi$ is quasidiagonal.
\label{thm:SufficientConditionQD}
\end{thm} 

\begin{proof}
Let $$A^+ := \begin{cases}
A^{\sim} & \makebox{  if  } A \makebox{  is nonunital }\\ 
A \oplus \mathbb{C} & \makebox{  otherwise.  }
\end{cases}$$

Let 
$${\phi}^+, {\sigma}^+ : A^{+} \rightarrow \C(\B)$$
be the unique unital monomorphisms that 
extend $\phi$ and $\sigma$ respective.
Then
$$KL({\phi}^+) = KL({\sigma}^+).$$

Hence, by \cite{LinExtQuasidiagonal} Theorem 3.7, ${\phi}^+$
and ${\sigma}^+$ are approximately unitarily equivalent. 
Consequently, $\phi$ and $\sigma$ are approximately unitarily equivalent.
It follows from 
Theorem \ref{Tappt=qd} that $\phi$ is quasidiagonal.  
\end{proof}

\section{Stable uniqueness}

  Moving towards a nonunital stable uniqueness result, we next
provide some definitions and results from \cite{ElliottGongLinNiu} (see also
\cite{GongLinNonunitalAUE}).

\begin{df} (\cite{ElliottGongLinNiu} Definition 3.13.  See also
\cite{GongLinNonunitalAUE} Definition 7.7)

Let $r_0, r_1 : \mathbb{N} \rightarrow \mathbb{Z}_+$ be maps,
$T : \mathbb{N} \times \mathbb{N} \rightarrow \mathbb{Z}_+$ be a map,
and $s, R \geq 1$ be integers.  We say that a \CA\, $\D$ belongs to
the class $\textbf{C}_{(r_0, r_1, T, s, R)}$ if all of the
the following statements hold:
\begin{enumerate}
\item[(a)] For any integer $n \geq 1$ and any pair of projections
$p, q \in \M_n(\widetilde{\D})$ with 
$$[p] = [q]\,\,\,{{{\rm in}}}\,\,\,K_0(\widetilde{\D}),$$
$$p \oplus 1_{\M_{r_0(n)}(\widetilde{\D})}
\sim^u q \oplus 1_{\M_{r_0(n)}(\widetilde{\D})}.$$
\item[(a')] For any integer $n \geq 1$ and any pair of projections
$p, q \in \M_n(\widetilde{\D})$, if $$[p] - [q] \geq 0,$$ then there exists
a projection $p' \in \M_{n + r_0(n)}(\widetilde{\D})$ such that
$$p' \leq p \oplus 1_{\M_{r_0(n)}(\widetilde{\D})}{{\andeqn}} 
p' \sim q \oplus 1_{\M_{r_0(n)}(\widetilde{\D})}.$$
\item[(b)] For any $n, k \geq 1$ and any $x \in K_0(\widetilde{\D})$ such that
$$-n[1_{\widetilde{\D}}] \leq k x \leq n [1_{\widetilde{\D}}],$$
$$-T(n,k) [1_{\widetilde{\D}}] \leq x \leq T(n,k) [1_{\widetilde{\D}}].$$
\item[(c)]  The canonical map
$$U(\M_s(\widetilde{\D}))/U_0(\M_s(\widetilde{\D})) \rightarrow K_1(\D)$$
is surjective.
\item[(d)] For any integer $n \geq 1$, if
$u \in U( \M_n(\widetilde{\D}))$ and $[u] = 0 \makebox{  in  } K_1(\D),$ then
$$u \oplus 1_{\M_{r_1(n)}(\widetilde{\D})} \in
U_0(\M_{n + r_1(n)}(\widetilde{\D})).$$
\item[(e)] For all $m \geq 1$, the exponential rank
$$cer(\M_m(\widetilde{\D})) \leq R.$$
\end{enumerate}
\end{df}

\begin{prop}\label{Wprop}
Let $B$ be a separable simple stably finite  \CA \makebox{ }  with continuous scale,
finite nuclear dimension, UCT, unique tracial state and  
$K_0(B) = ker \rho_B$.
Let $T: \N\times \N\to \ZI_+$ defined by
$T(n,k)=n.$
Then $$B\in {\bf C}_{0,0, 1, T, 7}.$$
\end{prop}

\begin{proof}
It follows from Theorem 4.3  of \cite{GLII} that
${\rm cer}(M_n({\widetilde{B}}))\le 6+\ep.$

Since $B$ has stable rank one and
$K_0({\widetilde{B}})$ is weakly unperforated,
it is easy to check that
$\B\in {\bf C}_{0, 0, 1, T, 7}.$
\end{proof}

\begin{df}  Let $\A$ be a separable \CA, $\B$ be a nonunital \CA,
and let $\sigma : \A \rightarrow \B$ be a positive map.
Let $F : \A_+ \setminus \{ 0 \} \rightarrow \mathbb{N} \times \mathbb{R}$, and let
$\E \subset \A_+\setminus \{0\}$ be a finite set.

We say that $\sigma$ is \emph{$F$-$\E$ full} if for any $\epsilon > 0$,
for any $b \in \B_+$ with
$\| b \| \leq 1$ and for any $a \in \E$,
 there are $x_1, x_2, ..., x_m \in \B$ with
$m \leq N(a)$ and $\| x_j \| \leq M(a)$,
where $F(a) = (N(a), M(a))$, and such that
\begin{equation} \| \sum_{j=1}^m x_j^* \sigma(a) x_j - b \| \leq
\epsilon.
\label{equ:FEFull} \end{equation}

We say that $\sigma$ is \emph{exactly $F$-$\E$ full} if (\ref{equ:FEFull}) holds
with $\epsilon = 0$.
If $\sigma$ is $F$-$\E$-full for every finite subset $\E$ of $A,$ then we say $\sigma$ is $F$-full.
\end{df}

\begin{df}\label{Dlceil+}
{\rm{ We introduce some notation that will be used in the next result and
later parts of the paper.
For a linear map $\phi : A \rightarrow B$ between C*-algebras,
we often let $\phi$ also denote the induced map
$\phi \otimes {\rm id}_{M_m} : M_m(A) \rightarrow M_m(B)$ for all $m$.
{{If $A$ and $B$ are not unital, to simplify notation, 
we understand that $\phi(x)$ is $\phi^\sim(x)$ for $x\in {\widetilde{A}},$
where $\phi^\sim: {\widetilde{A}}\to {\widetilde{B}}$ is the unitization
of $\phi.$}} 

{{Let $A$ be a unital 
\CA\, and let  $x\in A.$ Suppose
that
$\|xx^*-1\|<1$ and $\|x^*x-1\|<1.$ Then $x|x|^{-1}$ is a unitary.
Let us use $\lceil x \rceil$
to denote $x|x|^{-1}.$ 
Now let $A$ be any  separable amenable \CA.  Let ${\mathcal
P}\subset \underline{K}(A)$ be a finite subset.   Then there exist a finite subset
${\mathcal F}$ and $\ep>0$ such that, for any \CA\, $B$ and 
any ${\mathcal F}$-$\dt$-multiplicative c.p.c~map $L: A\to B,$
$L$ induces a \hm\, $[L]$ on $G({\mathcal P}),$ where $G({\mathcal
P})$ is the subgroup generated by ${\mathcal P},$ to $\underline{K}(B).$
Moreover (by choosing sufficiently small $\ep$ and large ${\mathcal F}$), 
if $L': A\to B$ is another ${\mathcal F}$-$\ep$-multiplicative c.p.c.~map such
that $\|L(x)-L'(x)\|<\ep$ for all $x\in {\mathcal F},$ then $[L']|_{G(\mathcal P)}=[L]|_{G(\mathcal P)}.$
Such a triple $(\ep, {\mathcal F},{\mathcal P})$ is sometime called a $KL$-triple.
In what follows, when we write $[L]|_{\mathcal P},$ we assume 
that $L$ is at least ${\mathcal F}$-$\ep$-multiplicative  so that
$[L]|_{G({\mathcal P})}$ is well defined (see 1.2  of \cite{Lnamj98}, 3.3 of \cite{DE},  and  2.11 of \cite{Lnclasn}, or, 
2.12 of \cite{GLN}).}}}}
\end{df}

\begin{thm}\label{thm:LinStableUniqueness}
Let $\A$ be a nonunital separable amenable \CA\, which satisfies the UCT,
let $r_0, r_1 : \mathbb{N} \rightarrow \mathbb{Z}_+$ and $T : \mathbb{N} \times
\mathbb{N} \rightarrow \mathbb{N}$ be three maps, $s, R \geq 1$ be integers,
and let $F : \A_+ \setminus \{ 0 \} \rightarrow \mathbb{N} \times (0, \infty)$
and $L : \bigcup_{m=1}^\infty U(M_m ( \widetilde{\A})) \rightarrow [0, \infty)$ be two
additional maps.

For any $\epsilon > 0$ and any finite subset $\F \subset \A$, there exists
$\delta > 0$, a finite subset $\G \subset \A$, a finite subset
$\mathcal{P} \subset \underline{K}(\A)$, a finite subset
$\mathcal{U} \subset \bigcup_{m=1}^\infty U(M_m(\widetilde{A}))$, a finite subset
$\E \subset \A_+ \setminus \{ 0 \}$, and an integer $K \geq 1$ satisfying the
following:

For any \CA\, $\B \in \textbf{C}_{(r_0, r_1, T, s, R)}$,
for any two $\G$-$\delta$-multiplicative c.p.c. maps
$\phi, \psi : \A \rightarrow \B$,
and for any $F$--$\E$ full $\G$-$\delta$-multiplicative map
$\sigma : \A \rightarrow M_l(\B)$ such that
$${\rm cel}(\lceil \phi(u) \rceil \lceil \psi(u)^* \rceil) \leq L(u)$$
for $u \in \mathcal{U}$, and
$$[\phi]|_{\mathcal{P}} = [\psi]|_{\mathcal{P}},$$
there exists a unitary $U \in M_{1 + Kl}(\widetilde{\B})$
such that
$$\| Ad(U) \circ (\phi \oplus S)(a) - (\psi \oplus S)(a) \| < \epsilon$$
for all $a \in \F$, where
$$S(a) := diag(\sigma(a), \sigma(a), ..., \sigma(a))$$
(the ``$\sigma(a)$" is repeated $K$ times in the diagonal).

Furthermore, if $\B$ has almost stable rank one, then one can choose
$U \in \widetilde{M_{1 + Kl}(\B)}$.

\end{thm}

\begin{proof}
This is \cite{GongLinNonunitalAUE}
Theorem 7.9.
\end{proof}

\begin{rmk}\label{Rsuniq}
{\rm{Note that the finite subset ${\mathcal U}$ in the statement of \ref{thm:LinStableUniqueness} may be assumed 
to a subset of $U(M_m({\widetilde{A}}))$ for some integer $m\ge 1.$ 
Let $\lceil {\mathcal U}\rceil$ be the image of ${\mathcal U}$ in $U(M_m(\widetilde{A}))/U_0(M_m(\widetilde{A})$
and let $G(\lceil {\mathcal U}\rceil)$ be the (finitely generated) subgroup generated by $\lceil {\mathcal U}\rceil.$}}
\end{rmk}

We provide some notation which will be used in Theorem \ref{TK1stuniq} as well as the proof of
Lemma \ref{Uniq-2}, and other places. 

\begin{df}\label{Dcoset}
Let $A$ be a nonunital \CA.
\begin{enumerate}
\item
{\rm Let  
$$\Pi_1: U(M_\infty({\widetilde{A}}))\to U(M_\infty({\widetilde{A}}))/U_0(M_\infty(\widetilde{A}))=K_1(A),$$
$$\Pi_{cu}: U(M_\infty(\widetilde{A}))\to U(M_\infty(\widetilde{A}))/CU(M_\infty(\widetilde{A})),$$ 
and
$$\Pi_{1,cu}: U(M_\infty(\widetilde{A}))/CU(M_\infty(\widetilde{A}))\to K_1(A)$$ 
be the usual quotient maps.} 
\item {\rm For each $A,$ fix a \hm\, 
$$J_A: K_1(A)\to U(M_\infty(\widetilde{A}))/CU(M_\infty(\widetilde{A}))$$
{{so that the following short exact sequence splits
\beq\label{DJspliting}
0 \longrightarrow U_0(M_\infty(\widetilde{A}))/CU(M_\infty(\widetilde{A}))
\longrightarrow U(M_\infty(\widetilde{A}))/CU(M_\infty(\widetilde{A}))\stackrel{\Pi_{1,cu}}{\rightleftarrows}_{J_A} K_1(A)\to 0
\eneq}}
(see  Cor. 3.3 of \cite{Thomsen}).
In other words, $\Pi_{1,cu}\circ J_A(x)=x$ for all $x\in K_1(A).$ {{We will also use $J$ instead of $J_A$ for brief.
Moreover, in what follows, once $A$ is given, we assume that $J$ is fixed.}} 

{\rm 

Fix a 
map $\Pi_{cu}^-: U(M_\infty(\widetilde{A}))/CU(M_\infty(\widetilde{A}))\to U(M_\infty(\widetilde{A}))$
such that $\Pi_{cu}(\Pi_{cu}^-(z))=z$ for all $z\in U(M_\infty(\widetilde{A})).$
Then, for each $u\in U(M_\infty(\widetilde{A})),$   we write 
$$
u=(\Pi_{cu}^-(\Pi_{cu}(u))u_{cu}
$$
where 
$$
u_{cu}=(\Pi_{cu}^-(\Pi_{cu}(u))^*u\in CU(M_\infty(\widetilde{A})).
$$
}

Note that 
$\Pi_{cu}^-$ is just a map between sets.}   {{Once $A$ is given, $\Pi_{cu}^-$ is fixed.}}

\item {\rm Let $$J^\sim: K_1(A)\to U(M_\infty(\widetilde{A}))$$
be given by 
$$J^{\sim} := \Pi_{cu}^-\circ J.$$
\item 
{\rm For $u\in U(M_\infty(\widetilde{A})),$ once $\Pi_{cu}^-$ is fixed, one may uniquely  write
\beq
u= \Pi_{cu}^-(J\circ \Pi_1(u))u_{0,cu},
\eneq
where 
\beq\label{Du0cu}
u_{0, cu}=\Pi_{cu}^-(J\circ \Pi_1(u))^*u\in U_0(M_\infty(\widetilde{A})).
\eneq}

Denote by  $J^\sim_0: U(M_\infty(\widetilde{A}))\to   U_0(M_\infty(\widetilde{A}))$
defined by $J^\sim_0(u)=J^\sim(\Pi_1(u))^*u\,(=u_{0,cu} \,\,{\rm as\,\,\, \eqref{Du0cu}}).$
For a fixed \CA\, $A,$ let us fix one splitting map $J$ and a map $\Pi_{cu}^-$ above  which then determine $J^\sim$ and $J^\sim_0.$

\item
{{Suppose that $B$ is another \CA\, and $h: A\to B$ be a \hm. 
Denote by $h^\dag: U(M_\infty(\widetilde{A}))/CU(M_\infty(\widetilde{A}))\to U(M_\infty(\widetilde{B}))/CU(M_\infty(\widetilde{B}))$
 the induced \hm. Denote by $h^\ddag: K_1(A)\to U(M_\infty(\widetilde{B}))/CU(M_\infty(\widetilde{B}))$ 
the \hm\, definded by $h^\dag\circ J$ as $J$ is fixed.  Note, in the case that $B$ has stable rank one (see Cor. 3.4 
and its remark 
of \cite{Thomsen}),
 $U(M_\infty(\widetilde{B}))/CU(M_\infty(\widetilde{B}))=U(\widetilde{B})/CU(\widetilde{B}).$ 
 In this case, $h^\dag$ is a \hm\, from $U({\widetilde{A}})/CU({\widetilde{A}})$ to 
 $U({\widetilde{B}}/CU({\widetilde{B}})$ and $h^\dag$ maps $K_1(A)$ to $U({\widetilde{B}}/CU({\widetilde{B}}).$
 }}
 
 \item
 Denote by $\Delta: U({\widetilde{B}}/CU(\widetilde{B})
 \to \Aff(T(\widetilde{B}))/\rho_B(K_0({\widetilde{B}}))$ the determinant map which 
 is an isometric isomorphism (see  section 3 of \cite{Thomsen} and Proposition 3.23 of \cite{GLN}).}
 
\end{enumerate}

\end{df}

\begin{thm}\label{TK1stuniq}
Let $\A$ be a nonunital separable amenable \CA\, which satisfies the UCT,
let $r_0, r_1 : \mathbb{N} \rightarrow \mathbb{Z}_+$ and $T : \mathbb{N} \times
\mathbb{N} \rightarrow \mathbb{N}$ be three maps, $s, R \geq 1$ be integers,
and let $F : \A_+ \setminus \{ 0 \} \rightarrow \mathbb{N} \times (0, \infty)$
and $L : J^\sim(K_1(A))
\rightarrow [0, \infty)$ 
be two
additional maps.  

For any $\epsilon > 0$ and any finite subset $\F \subset \A$, there exists
$\delta > 0$, a finite subset $\G \subset \A$, a finite subset
$\mathcal{P} \subset \underline{K}(\A)$, a finite subset
$\mathcal{U} \subset 
J^\sim(K_1(A))$, a finite subset
$\E \subset \A_+ \setminus \{ 0 \}$, and an integer $K \geq 1$ satisfying the
following:

For any \CA\, $\B \in \textbf{C}_{(r_0, r_1, T, s, R)},$
for any two $\G$-$\delta$-multiplicative c.p.c. maps
$\phi, \psi : \A \rightarrow \B$,
and for any $F$--$\E$ full $\G$-$\delta$-multiplicative map
$\sigma : \A \rightarrow M_l(\B)$ such that
$$cel(\lceil \phi(u) \rceil \lceil \psi(u)^* \rceil) \leq L(u)\rforal u\in 
\mathcal{U},\andeqn$$ 
$$[\phi]|_{\mathcal{P}} = [\psi]|_{\mathcal{P}},$$
there exists a unitary $U \in \M_{1 + Kl}(\widetilde{\B})$
such that
$$\| Ad(U) \circ (\phi \oplus S)(a) - (\psi \oplus S)(a) \| < \epsilon$$
for all $a \in \F$, where
$$S(a) := diag(\sigma(a), \sigma(a), ..., \sigma(a))$$
(the ``$\sigma(a)$" is repeated $K$ times in the diagonal).

Furthermore, if $\B$ has almost stable rank one, then one can choose
$U \in \widetilde{\M_{1 + Kl}(\B)}$.

\end{thm}

\begin{proof}
It suffices to show that there is a map $L_1: U(M_\infty({\widetilde{A}}))\to {[0, \infty)}$
which depends only on $A$ and $L$ such that, for any finite subset $\U\in 
M_\infty(\widetilde{A}),$
if 
$$
{\rm cel}(\lceil \phi(u)\rceil \lceil \psi(u)^*\rceil )\le L(u)
\rforal u\in J^\sim\circ \Pi_1(\U),
$$ 
for any $\phi$ and $\psi$ come later in the statement,
one always has 
\beq\label{TstuniqK1-1}
{\rm cel}(\lceil \phi(u)\rceil \lceil \psi(u)^*\rceil )
\le  L_1(u)\rforal u\in \U,
\eneq
provided that $\phi$ and $\psi$ are $\G_1$-$\dt_1$-multiplicative, where 
$\dt$ is sufficiently small and $\G$ is sufficiently large which depends 
only on $\U$ (in the statement of ) and $L,$ as we then apply Theorem \ref{thm:LinStableUniqueness} 
for this $L_1$ (choosing  $\dt<\dt_1$ large $\G\supset \G_1$ and $K$  and so on). 


Let us provide the details for the issue. {{As $A$  is given, we fix a splitting \hm\, $J$ and a map $\Pi_{cu}^-$ as in \ref{Dcoset}
(so $J^\sim$ and $J_0^\sim$ are also fixed).}} 
Define $L_0: U_0(M_\infty({\widetilde{A}}))\to {[0, \infty)}$ as follows.
For each $u\in U_0(M_\infty(\widetilde{A})),$ there is a smallest $m(u)\ge 1$
such that $u\in U_0(M_{m(u)}(\widetilde{A})).$
Define $L_0(u)={\rm cel}(u)$ (in 
$M_m(\widetilde{A}).$ 

Once $\U\in U(M_m(\widetilde{A}))$ is fixed.
Define 
\beq
\U_1=\{ J^\sim(\Pi_{1,cu}(u)), J^\sim_0(u):  u\in \U\}.
\eneq
{{\Wlog, we may assume that $J_0^\sim(u)\in U_0(M_m(\widetilde{A})).$}}

For each $u\in \U,$ there are $h_1(u), h_2(u),...,h_{k(u)}(u)\in M_m({\widetilde{A}})_{s.a.}$
such that
\beq
\exp(ih_1(u))\exp(i h_2(u))\cdots \exp(i h_{k(u)}(u))=J^\sim_0(u).
\eneq
We choose a small $\dt_1>0$  and  large $\G_1$  such that, for all $u\in \U,$
\beq
{\,\,\,\,\,\,\,\,\,\,\,\,\,\,\,\,\,}{\|\phi(J^\sim_0(u))-\exp(i\phi(h_1(u)))\exp(i \phi(h_2(u)))\cdots \exp(i \phi(h_{k(u)}(u)))\|<1/16\pi}
\eneq
for any $\G_1$-$\dt_1$-multiplicative c.p.c.~map $\phi$ from ${\widetilde{A}}.$
In particular,
\beq
{\rm cel}(\lceil \phi(J^\sim_0(u))\rceil)\le {\rm cel}(J^\sim_0(u))+1/4\rforal u\in \U.
\eneq
We may also assume that
\beq
\lceil \phi(u)\rceil\approx_{1/64\pi} \lceil \phi(J^\sim(\Pi_1(u)))\rceil \lceil \phi(J^\sim_0(u))\rceil\rforal u\in \U.
\eneq

Define $L_1(u)=L(J^\sim(\Pi_1(u)))+ 2L_0(J^\sim_0(u))+1$ for all $u\in U(M_\infty({\widetilde{A}})).$

Note that, as had been demonstrated, if $\dt_1$ is small enough and $\G_1$ is large enough
independent of $\phi$ or $\psi$ (and also independent of $B$ in the class of ${\bf C}_{r_0, r_1, t, s, R}$),
when both $\phi$ and $\psi$ are $\G_1$-$\dt_1$-multiplicative,
for all $u\in \U,$
\beq
&&{\rm cel}(\lceil \phi(u)\rceil \lceil \psi(u)\rceil^*)\\
&&\le  1/16+{\rm cel} \phi(J^\sim(\Pi_1(u)))\rceil \lceil \phi(J^\sim_0(u))\lceil \psi(J^\sim_0(u))\rceil
\lceil \psi(J^\sim (\Pi_1(u)))\rceil\\
&&\le 2(L_0(J^\sim_0(u))+1/4)+L(J^\sim(\Pi_1(u)))\le L_1(u).
\eneq
In other words, \eqref{TstuniqK1-1} holds.  The theorem then follows from Theorem \ref{thm:LinStableUniqueness}.

\end{proof}

\section{Existence and exponential length}


\begin{lem}\label{lem:Feb2620203PM}
Let $A$ be a separable algebraically simple \CA\, with finite nuclear dimension which satisfies the UCT
and has a unique tracial state $\tau_A.$
Suppose that $A$ is  nonunital and stably projectionless. 
Then we have the following:
\begin{enumerate}
\item $A$ is $\Z$-stable and has stable rank one.
\item $K_0(A) = {\rm ker} \rho_A$.
\end{enumerate}
\end{lem}

\begin{proof}
By \cite{aTz}, $A$ is ${\mathcal Z}$-stable.  
It follows from \cite{Rz} that $A$ has almost stable rank one. Since $A$ has only one tracial state, it follows from Corollary 
A7 of \cite{ElliottGongLinNiu}  that $K_0(A)={\rm ker}\rho_A.$ 
By Theorem 15.5 of \cite{GLII}, 
$A$ is classifiable and is in class ${\mathcal D}$, which
is defined in 3.9 of \cite{GLII} (see also 8.1 of  \cite{eglnp}).
Therefore, by {{Theorem 11.5 of \cite{eglnp},}}
$A$ has stable rank one.  
\end{proof}

\begin{df}\label{Drazak}
{\rm Let $\W$ be the \emph{Razak algebra}, which is a nonunital, simple, separable,
nuclear,  
continuous {{scale,}} stably projectionless \CA\,  with unique tracial
state $\tau_W$ and $K_*(\W) = 0$. 
$\W$ also has stable rank one and is $\Z$-stable
(see  
\cite{Razak}, \cite{Tsang}, \cite{Jacelon}).
{{It is proved in  \cite{ElliottGongLinNiu} that $\W$ is the only nonunital simple separable \CA\, with finite nuclear dimension,
$K_i(\W)=\{0\}$ ($i=0,1$) and with a unique tracial state
which satisfies the UCT and has continuous scale.}}
 {{From  this, one can also conclude that 
$\W$ is *-isomorphic to any of its
nonzero hereditary C*-subalgebras.}} } 
\end{df}

\begin{rmk}
In fact, the proof of Lemma \ref{lem:Feb2620203PM} shows that the \CA\,
$A$ is in the classifiable class $\mathcal{D}$ defined in \cite{GLII} 3.9.
We also note that it is not hard to check directly that $\W$ has properties
(1) and (2) of Lemma \ref{lem:Feb2620203PM}. (E.g., see \cite{Jacelon}.)
\label{rmk:Feb2620203PM}
\end{rmk}

\begin{lem}[Theorem 1.1 of \cite{RI}]\label{LWembB}
Let $B$ be an  separable infinite dimensional simple 
$\Z$-stable \CA\, with stable rank one, {{and every 2-quasi-trace of $\overline{bBb}$ is a trace
for any $b\in {\rm Ped}(B)_+\setminus \{0\}.$}}
Then there is an embedding $\phi_{w,b}: \W\to B.$

If $B$ also has continuous scale and is
 stably projectionless, we may require that $\phi_{w,b}$ maps strictly positive elements
to strictly positive elements.

Moreover,
if $\phi_1, \phi_2: \W\to B$ are two monomorphisms such that
$d_\tau(\phi_1(a))=d_\tau(\phi_2(a))$ holds for all $\tau\in T(B)$ and for one
non-zero $a\in W_+\setminus\{0\},$ then there exists
a sequence of unitaries $u_n\in {\widetilde{B}}$ such
that
$$
\lim_{n\to\infty} {\rm Ad}\, u_n\circ \phi_1(c)=\phi_2(c)\rforal c\in \W.
$$
\end{lem}

\begin{proof}
{{Since $K_0(\W)=\{0\}$ and has a unique tracial state,}} from Proposition 6.2.3 of \cite{RI}, {{one computes}} that $Cu^\sim(\W)= 
 (-\infty, \infty].$
It is easy to see that, for the first part of the Lemma, we
may assume that $B$ has continuous scale (e.g., Proposition 5.4
of \cite{eglnp}).   
{{Since $B$ is simple, has stable rank one and simple,  and 
every 2-quasi-trace of $B$ is a trace, $T(B)\not=\emptyset.$}}
Since $B$ has continuous scale, it is  well-known that 
$T(B)$ is compact (see, for example,  Theorem 5.3 of \cite{eglnp}).

It follows from Theorem 7.3 of \cite{eglnp} 
(see also Theorem 6.~2.3 of \cite{RI}) that when $B$ is stably projectionless,  
$Cu^\sim (B)=K_0(B)\sqcup {\rm LAff}^\sim(T(B))$ {{(recall that $T(B)$ is compact as $B$ has continuous scale).}}
This also holds for  the case that $B$ is not stably projectionless (see the proof of Proposition 6.1.1 of \cite{RI};  
in fact,
here we can replace the trace space by $T(pM_m(B)p)$ for some nonzero
projection $p\in M_m(B)$, and for some $m\ge 1$).

{{Fix a strictly positive element $e_W\in \W$  with $\|e_W\|=1$ and a strictly positive element $e_B\in B$ with $\|e_B\|=1.$ 
Then $e_W$ is represented by $1\in (-\infty, \infty]$ in ${\rm Cu}^\sim(\W).$
Note also 
$d_\tau(e_B)=1$ for all $\tau\in T(B).$}} 
Choose $0 < a < 1.$
Define a map $j: (-\infty, \infty] \to {\rm LAff}^\sim (T(B))\subseteq Cu^{\sim}(B)$ 
by $j(r)=ar,$ where  we view $ar$ as a constant function on $T(B).$
Thus, $j$ is a morphism from $Cu^{\sim}(\W)$ to $Cu^{\sim}(B).$  
Hence, by Theorem 1.0.1 of \cite{RI}, there is a *-\hm\, $\phi_{w, b}: \W\to B,$
such that $$Cu^\sim(\phi_{w,b})=j.$$

In the case that $B$ is stably projectionless, $e_B$ is not a projection.
So if, in the previous paragraph, we choose $a = 1$, {{then $\phi_{w,b}(e_W)$
is Cuntz equivalent to $e_B.$ Since $B$ has stably rank one, ${\rm Her}(\phi_{w,b}(e_W))$
is isomorphic to ${\rm Her}(e_B)=B.$ So we may also assume that $\phi_{w,b}$}} 
maps strictly positive elements to strictly positive elements.

The second part of the Lemma follows immediately from the fact 
that $Cu^{\sim}(\W) = (-\infty, \infty]$ and 
Theorem 1.1 of \cite{RI}.

\end{proof}


 \begin{df}\label{Dgamma}
{{\rm  Let $B$ be a separable (nonunital) simple \CA\, with stable rank one and with continuous scale
 such that $K_0(B)={\rm ker}\rho_B.$
 Then $U(M_\infty({\widetilde{B}}))/CU(M_\infty({\widetilde{B}}))=U({\widetilde{B}})/CU({\widetilde{B}}).$
 Fix $J: K_1(B)\to U({\widetilde{B}})/CU({\widetilde{B}}).$ Then, by \ref{Dcoset},  one may write
 \beq
 U({\widetilde{B}})/CU({\widetilde{B}})=(\Aff(T({\widetilde{B}}))/\ZI)\oplus J(K_1({\widetilde{B}})).
 \eneq}}
{\rm{{{Recall $\Aff(T(\widetilde{B}))\cong \Aff(T(B))\oplus \R.$}}}}

{{{\rm Suppose that $D$ is a hereditary \SCA\, of $B$ which also has continuous scale.
Suppose that $e_D\in D$ is a strictly positive element of $D.$ There is
an affine homeomorphism $\gamma_D: T(B)\to T(D)$ defined by
$\gamma_D(\tau)(d)=(1/d_\tau(e_D))(\tau)(d)$ for all $d\in D\subset B$ and $\tau\in \partial{(T(B))},$
where $\partial{(T(B))}$ is the set of the extremal points of $T(B)$
(see the proof of Lemma 6.4 of \cite{LnHah}).
Denote by $\gamma^D: \Aff(T(D))\to \Aff(T(B))$ the induced {{linear}}
map defined by
${ \gamma^D(f)(\tau)}=f(\gamma_D(\tau))$ for all $f\in \Aff(T(D))$ and $\tau\in T(B).$}}}

{{{ Let} ${\overline{\gamma^D}}: \Aff(T({\widetilde{D}}))/\ZI\to \Aff(T({\widetilde{B}}))/\ZI$ { be the map} induced by $\gamma^D.$}}

 Let $j_D: { D} \to B$ be the { inclusion map.}  Note that, if $u\in {\widetilde{D}}$ is a unitary, then
 we may write $u=e^{2\pi i\theta}\cdot 1_{\widetilde{D}}+u_d,$ where $\theta\in (-1, 1]$ and $u_d\in D.$
 Note that $j_D(u)=e^{2\pi i\theta}\cdot 1_{\widetilde{B}}+u_d.$
 Also $(j_D)_{*1}: K_1(D)\to K_1(B)$ is an isomorphism.   Moreover, by Proposition 4.5 of \cite{GLII},
 $j_D^\dag: U(\widetilde{D})/CU(\widetilde{D})\to {{ U(\widetilde{B})/CU(\widetilde{B})}}$
 is an isomorphism.

\end{df}

\begin{lem}\label{LBembed}
Let $B$ be a {{nonunital}} separable  {{finite}} simple  \CA\, with finite nuclear dimension and
with continuous scale. {{Suppose that every quasitrace of $B$ is a trace.}}

(1) For each $t\in (0,1)$ there are elements $a_t, a_{1-t}\in B_+$ such that
$a_ta_{1-t}=0,$
$d_\tau(a_t)=t$
and $d_\tau(a_{1-t})=1-t$ for all
$\tau\in T(B)$ and $a_t+a_{1-t}$ is a strictly positive element of $B;$

(2) Suppose, in addition, that $B$ is finite and
$\Z$-stable, $K_0(B)={\rm ker}\rho_B$, and $B$ satisfies the UCT.

Then, for any $t \in (0,1)$, for any $a_t\in B_+\setminus \{0\}$
with $d_\tau(a_t)=t$ for all $\tau \in T(B)$, there is
an isomorphism
$\phi_t: B\to B_t:={\rm Her}(a_t)$
such that $KL(\phi_t)=KL({\rm id}_B),$
$(\phi_t)_T={{\gamma_{B_t}^{-1}}}: T(B_t)\to T(B)$
({{so
$\gamma_{B_t}^{-1}(\tau)(\phi_t(b))=\tau(b)$ (see  Definition \ref{Dgamma} for $\gamma_{B_t}$) for all $b\in B$
and $\tau\in T(B)),$ and } }
$$
{{(\phi_t)^\dag|_{J(K_1(B))}=(j_{B_t}^\dag)^{-1} |
_{{J(K_1(B))}}
\andeqn 
(\phi_t)^\dag|_{{\rm Aff}(T({\widetilde{B}}))/\ZI}}}=
\overline{\gamma^{B_t}}.  
$$

{{Moreover,  for any $u\in U_0({\widetilde{B}}),$
\beq\label{LBembed-0}
{{{\rm dist}(uCU(\widetilde{B}), j_{B_t}^\dag(\phi_t^\dag(u)))\le (1-t){\rm dist}(u,1_{\widetilde{B}}}}).
\eneq}}
\end{lem}

\begin{proof}
It follows from \cite{aTz} that $B$ is $\Z$ stable.
{{By Theorem 6.8  of \cite{ElliottRobertSantiago} (see also \cite{BrownToms} Theorem 2.5),
$Cu(B)=V(B)\sqcup {{{\rm LAff}_+(T(B))}}.$
For both part (1) and (2), one may assume that $B$ is finite and $T(B)\not=\emptyset.$}}

{{If $B$ is not stably projectionless, then there is a projection $e\in M_m(B)$ for some $m\ge 1.$
It follows Cor. 3.1 of \cite{TomsW07}  that $eM_m(B)e$ is a unital  simple ${\mathcal Z}$-stable
\CA.  By Theorem 6.7 of \cite{RorZRank}, $eM_m(B)e$ has stable rank one. It follows that $B$ has stable rank one.}}

{{Fix a strictly positive element $e_B$ of $B.$ Note that $d_\tau(e_B)=1$
for all
$\tau\in T(B).$  For any $t\in (0,1),$ choose elements $a_{1-t}', a_t'\in B_+$
which are not projections (as $Cu(B)=V(B)\sqcup {\rm LAff}_+(T(B))$) such
that $d_\tau(a_{1-t}')=1-t$  and $d_\tau(a_t')=t$ for all $\tau\in T(B).$
Let $b=a_{1-t}'\oplus a_t'\in M_2(B)_+.$ Then $d_\tau(b)=1$ for all $\tau\in T(B).$
{{Therefore $d_\tau({{b}})=d_\tau(e_B)$ and both $b$ and $e_B$ are not projections.}}
If $B$ is not stably projections,  applying  Theorem 3 of \cite{CEI} (see also Theorem 3.3 of {{\cite{BrownToms}}}), as
and $eM_m(B)e$ is unital and has stable rank one, if $B$ is stably projectionless,
applying (the last part of) Theorem 1.2 of \cite{Rz},
one obtains a}}
isomorphism
$h: \overline{bM_2(B)b}\to B.$ Let $a_t=h(a_t')$ and $a_{1-t}=h(a_{1-t}').$
Then $a_ta_{1-t}=0$ and $a_t+a_{1-t}$ is a strictly positive element of $B.$

For part (2), note that since
$B$ is finite and $\Z$ stable, it is stably finite.
Note also that since $K_0(B)={\rm ker}\rho_B,$ $B$ is stably projectionless.
By Theorem 15.6 of \cite{GLII}, $B\in {\mathcal D}_0$  and $B\in B_T$ as defined there.  Then, part (2) follows from
Theorem 12.8  of \cite{GLII}.  In fact, note that $K_i({\rm Her}(a_t))=K_i(B),$ $i=0,1.$
Define $\kappa_0={\rm id}_{K_0(B)},$
$\kappa_1={\rm id}_{K_1(B)}$
Note that $B_t:={\rm Her}(a_t)$ also has continuous scale {{as $d_\tau(a_t)$ is continuous on $T(B)$
(see Proposition 5.4 of \cite{eglnp}).}}
Let ${{\kappa_T:=(\gamma_{B_t})^{-1}:}} T(B_t)\to T(B)$ be as defined in
\ref{Dgamma}.
Note also that since $K_0(B)={\rm ker}\rho_B,$ ${{U(\widetilde{B})}}/CU(\widetilde{B})\cong 
\Aff(T(\widetilde{B}))/\mathbb{Z} \oplus K_1(B).$
Let $\kappa_{cu}: U({\widetilde{B}})/CU({\widetilde{B}})\to U({\widetilde{B_t}})/CU({\widetilde{B_t}})$ be the map
defined by
\beq
{{\kappa_{cu}|_{J_B(K_1(B))}={{(j_{B_t}^\dag)^{-1}}}|_{J_B(K_1(B))}\andeqn \andeqn
\kappa_{cu}|_{{\rm Aff}(T({\widetilde{B_t}}))/\ZI}=\overline{\gamma^{B_t}}}}
\eneq
 Then, by Theorem 12.8 of \cite{GLII},
there is a \hm\, $h: B\to B_t$ such that $KL(h)=KK({\rm id}_B),$
{{$h^\dag=\kappa_{cu},$}}
and $s(h(b))=\kappa_T^{-1}(s)(b)$
for all $s\in T(B_t)$ and all $b \in B$.

For the last part of the lemma, let $u\in U_0({\widetilde{B}}).$
Write $u=\exp(i 2\pi a)w,$ where $a=\af \cdot 1_{\widetilde{B}}+a_b,$ where $a_b\in B_{s.a.},$ $\af\in \R$
and $w\in CU(\widetilde{B}).$
Moreover, we may assume that $\Delta(u)(\tau)=\af +\tau(a_b)$
for all $\tau\in T({\widetilde{B}})$ (see \cite{Thomsen} and  Corollary 2.12 of {{\cite{GLX-ER}}} as well as (6) of \ref{Dcoset}).
We compute that
\beq
j_{B_t}^\dag(h_t^{\dag}(u))=\overline{\exp(i2\pi (\af \cdot 1_{\widetilde{B}}+h_t(a_b))},
\eneq
 where $\tau(h_t(a_b))=t\tau(a_b)$ for all $\tau\in T(B).$
 Therefore
 \beq
\overline{u(j_{B_t}^\dag(h_t^{\dag}(u)))^*}=\overline{\exp(i2\pi(a_b-h_t(a_b))}.
 \eneq
 Note that
 \beq
 \tau(a_b-h_t(a_b))=(1-t)\tau(a_b)\rforal \tau\in T(B).
 \eneq
 Then \eqref{LBembed-0}  follows from the fact that $\Delta$ is an isometric isomorphism.
\end{proof}

\begin{lem}\label{lem:FMapExists}
Let $B$ be an algebraically simple, $\sigma$-unital
\CA\, and let $C$ be a $\sigma$-unital \CA.
Suppose that $\sigma: C\to B$  is a nonzero \hm.

Then there exists a map: $${{F:}}\, C_+\setminus \{0\}\to \mathbb{N} \times \mathbb{R}$$
such that for every finite subset $\E \subset \A_+ \setminus \{ 0 \}$,
$\sigma$ is $F$-$\E$ full.

\end{lem}

\begin{proof}
Let $A=\sigma(C).$ Then $A\subset B$ is a $\sigma$-unital C*-subalgebra.
Then Proposition 5.6 of \cite{eglnp} applies.

\end{proof}

\begin{df}\label{DWemb}
Let $A$ be a separable \CA. We say that $A$ is \emph{$\W$ embeddable}
if there is a monomorphism  $\phi: A\hookrightarrow \W.$

{{ Since $\W$ is projectionless, if $A$ is $\W$ embeddable, then $A$ is non-unital.
Let $e_A\in A$ be a strictly positive element. Consider $a=\phi(e_A).$
There is an isomorphism  $s: \overline{a\W a}\to \W.$   Then $s\circ \phi: A\to \W$ is
an embedding which maps $e_A$ to a strictly positive element of $\W.$ So, if it is needed, one may assume
that $\phi$ maps strictly positive elements to strictly positive elements.}}

\end{df}

\begin{rmk}\label{RWF}
If $A$ is $\W$ embeddable, then $T_f(A)\not=\emptyset.$ In particular, $A$ is not purely infinite.
Let $\tau_W$ be the unique tracial state of $\W.$ Then the normalization of $\tau_W\circ \phi$ is a faithful tracial state
of $A.$

\end{rmk}

{{Recall that $\zo$ is the unique separable stably projectionless simple \CA\, with finite nuclear dimension which satisfies the UCT
and which has a unique tracial state and $K_0(\zo)=\ZI$ and $K_1(\zo)=\{0\}$ (see  Cor. 15.7 of \cite{GLII}).}}

\begin{thm}\label{T-existoI}
Let $A$ be a separable amenable \CA\, which is $\W$ embeddable and let
$B$ be a separable simple stably projectionless  \CA\, with finite nuclear dimension
and with continuous scale. Suppose that ${\rm ker}\rho_B=K_0(B)$ and both $A$ and $B$ satisfies the UCT.

Then, for any $x\in KL(A,B),$ there is a monomorphism $h: A\to B$ such
that $KL(h)=x.$

\end{thm}

\begin{proof}
By Theorem 15.6 of \cite{GLII}, $B\cong B\otimes \zo.$
It follows from Theorem 10.8  of \cite{GLII} that there exists a sequence of c.p.c.~maps
$\phi_n: A\to B\otimes {\mathcal K}$ such that
\beq
\lim_{n\to\infty}\|\phi_n(a)\phi_n(b)-\phi_n(ab)\|=0\rforal a,b\in A\andeqn
[\{\phi_n\}]=x.
\eneq
\Wlog, we may assume
that $\phi_n: A\to B\otimes M_{r(n)}$ for some sequence
$\{r(n)\}\subset \N.$   Since $B\otimes M_{r(n)}$ is also a separable simple stably projectionless
\CA\ with finite nuclear dimension and with continuous scale, and satisfies the UCT, by part (2) of \ref{LBembed}, {{and}}
by replacing $\phi_n$ by ${{\phi_{{1\over{r(n)}}}}}\circ \phi_n,$ we may assume
that $\phi_n: A\to B.$

Let $\F_1\subset\F_2\subset \cdots  \F_{n},...$ be a dense sequence of the unit ball of $A.$
Let $\{\ep_n\}$ be a sequence of positive numbers such that $\sum_{n=1}^\infty\ep_n<1.$

Fix {{an embedding}} ${{\iota_A:}} A\to \W.$
For each $u\in U(M_m(\widetilde{A}))$ ($m=1,2,...,$), {{$\iota_A(u)$}}
(which we mean $({{\iota_A}}\otimes {\rm id}_{M_m})(u)$)
is in ${{U_0(M_m(\widetilde{\W}))}}.$ Define $L_1: \bigcup_{m=1}^\infty U(M_m(A))\to \R_+$
by $L_1(u)={\rm cel}({{\iota_A}}(u))$ for $u\in \bigcup_{m=1}^\infty U(M_m(A)).$

By {{Lemma}} \ref{LWembB}, there is a {{\hm\,}} $\phi_{w,b}: \W\to B$ which maps strictly positive elements to
strictly positive elements. Put $\sigma_A: =\phi_{w,b}\circ \iota_A: A\to B.$
Note that
\beq\label{T-existoI-se}
{\rm cel}(\sigma_A(u))\le L_1(u)\rforal u\in \bigcup_{m=1}^\infty U(M_m(A)).
\eneq

By Lemma \ref{lem:FMapExists},
let $F: A_+\setminus \{0\}\to \N\times (0,\infty)$ {{be}} such that
$\sigma_A$ is $F$-full in $B.$

Let $L=L_1+2\pi+1.$ Let  $T: \N\times \N\to \ZI_+$ be the map defined by
$T(n,k)=n.$   We {{will apply}} Theorem  \ref{thm:LinStableUniqueness}.
Note, by Proposition \ref{Wprop}, $B\in {\bf C}_{0,0,1,T,7}.$  So let $r_0=r_1=0,$ $s=1$ and $R=1.$

For the above data, and for each $n,$ let $\dt_n>0$ (in place of $\dt$), $\G_n\subset A$ (in place of $\G$) be a finite subset,
${\mathcal P}_n\subset \underline{K}(A)$ (in place of $\mathcal P$) be a  finite subset,
$\U_n\subset U(M_{m(n)}({\widetilde{A}}))\cap J^\sim(K_1(A))$
(in place of $\U$) be a finite subset
(for some integer $m(n)$), $\E_n\subset A_+\setminus \{0\}$ (in place of $\E$)  be a finite subset, and
integer $K_n\ge 1$ (in place of $K$) as provided by
{{Theorem \ref{TK1stuniq}}}
for $\ep_n/2$ (in place of $\ep$) and {{for}} $\F_n$ (in place of $\F$) as well as the given $L$ and $T$ above).
{{Passing to a subsequence, we may  assume that $\phi_n$ is $\G_n$-$\dt_n$-multiplicative, and  $\lceil \phi_n(u)\rceil$ { is} well defined for
all $u\in \U_n.$}}  We may {{also}} assume that $K_n\le K_{n+1},$
$n\in \N.$

{{For each $u\in \U_n,$
define
\beq\label{Texistol-nn0}
z_n^{(1)}=J_B(\Pi_{1,cu}(\lceil\phi_n(u)\rceil){)}.
\eneq
Then,  for each  $u\in {\mathcal U}_n,$
\beq\label{Texistol-nn1}
\lceil \phi(u)\rceil=z_n^{(0)} z_n^{(1)} \andeqn
z_n^{(0)}:=\lceil\phi_n(u)\rceil  J_B(\Pi_{1,cu}(\lceil\phi_n(u)\rceil ^*))\in U_0(M_{m(n)}({\widetilde{B}})).
\eneq}}
{{Define, for each $n,$
\beq
\lambda_n'=\max\{{\rm cel}(\lceil\phi_n(u)\rceil  J_B(\Pi_{1,cu}(\lceil\phi_n(u)\rceil ^*)): u\in {\mathcal U}_n \}.
\eneq
}}
Define
$$
\lambda_n=\max\{{\rm cel}(\lceil \phi_n(u)\rceil \lceil \phi_{n+1}(u^*)\rceil): u\in \U_n\},\,\,\, n=1,2,....
$$
Choose $J_n\in \N$ such that
\beq\label{TexistoI-1}
2(\lambda_n'+\lambda_n)+7\pi/J_n <  1\andeqn J_{n+1} > J_n > K_n, \,\,n=1,2,....
\eneq

Choose $t_n\in (0,1)\cap {\mathbb Q}$ such that
\beq\label{Texistol-2}
t_{n+1}<t_n,\, (1-t_n) > 2J_n(4(K_n+1)+1)t_n,\,\,\, n=1,2,....
\eneq
Define, for $n\in \N,$
\beq\label{Texistil-3-1}
N_n=2J_n(4(K_n+1)+1),\,\,d_n=(1/N_n)(1-t_n),\,\,\,r_{n+1}=(t_n-t_{n+1}).
\eneq
One can check that
\beq\label{Texistol-3}
&&0<t_n<d_n, \,\,\, r_{n+1} <2J_nd_n+(t_n-t_{n+1}) < (2J_n+1)d_n,\\
&&t_n+2J_nd_n+r_{n+1} <  (4J_n+2)d_n,\,\,\,8J_n(K_n+1) > (4J_n+2)K_n\andeqn\\\label{Texistol-3+1}
&&(1-t_{n+1})=r_{n+1}+(1-t_n)=r_{n+1}+2J_nd_n+8J_n(K_n+1)d_n\\\nonumber
&& = r_{n+1}+2J_nd_n+(N_n-2J_n)d_n,\\\label{est-Kn}
&&{{N_n-2J_n\ge 4(K_n+1)2J_n.}}
\eneq
By Lemma \ref{LBembed}, choose $a_{t_n}, a_{1-t_n}\in B$ such that $a_{t_n}a_{1-t_n}=0$ and $d_\tau(a_{t_n})=t_n$ and
$d_\tau(a_{1-t_n})=1-t_n$ for all $\tau\in B.$
Also let $a_{r_n}\in {\rm Her}(a_{1-t_n})$ such that $d_\tau(a_{r_n})=r_n$ for all $\tau\in T(B),$ $n\in \N.$
Let
\beq
&&s_n: B\to B_{n,1}:={\rm Her}(a_{d_n})\subset {\rm Her}(a_{1-t_n}),\\
&&s_n^r: B\to B_{n,r}:={\rm Her}(a_{r_{n}})\subset B,\andeqn\\\label{Texitol-3++}
&&s_n^0: B\to B_{n,0}:={\rm Her}(a_{t_n})\subset {{{\rm Her}(a_{t_{n-1}})\subset}} B
\eneq
be the isomorphisms given by part (2) of Lemma \ref{LBembed},
$n\in \N.$   {{In particular,
$
KL(s_n)=KL(s_n^r)=KL(s_n^0)=KL({\rm id}_B)$ and
\beq\label{Length-n1}
&&(s_n^0)^\dag|_{{{J_B(K_1({\widetilde{B}}))}}}
={(j_{B_{t_n}})^\dag}^{-1}{{|_{J_B(K_1(\widetilde{B}))}}} \andeqn\\
&&{\rm dist}(z, (j_{B_{t_n}})^\dag((s_n^{(0)})^\dag(z))\le (1-t_n){\rm dist}(z,\overline{1})\rforal z\in U_0({\widetilde{B}})/CU(\widetilde{B}).
\eneq}}
{{Moreover, viewing ${\rm Her}(a_{t_{n+1}})\subset {\rm Her}(a_{t_n})$   and denoting $j_{B_{t_{n+1}}}': B_{t_{n+1}}\to B_{t_n}$ to be the embedding,
\beq\label{Length-n2}
{\overline{{{(s_n^{(0)})^{\dag}}}(x)
(j_{B_{t_{n+1}}}')^\dag((s_{n+1}^{(0)})^\dag(x^{-1}))}}=\overline{1}\rforal x\in J_{B_{t_n}}(K_1({\widetilde{B_{t_n}}}))
\eneq
and, as at  the end of proof of \ref{LBembed},
\beq\label{Length-n2+}
&&\hspace{0.2in}{\rm dist}((s_{n}^{(0)})^\dag(z), (j_{B_{t_{n+1}}}')^\dag((s_{n+1}^{(0)})^\dag(z)){{)}}\le (t_n-t_{n+1})
{\rm dist}(z,{\overline{1}})\rforal z\in U_0(M_\infty({\widetilde{B}})/CU(M_\infty({\widetilde{B}}).
\eneq}}

Define ${{\Lambda_{n,0}}}:=s_{n}^0\circ \phi_n: A\to  B_{n,0}.$
We may assume that
\beq\label{KLadded}
[{{\Lambda_{n,0}}}]|_{{\mathcal P}_n}=x|_{{\mathcal P}_n}=
{{[\Lambda_{n+1,0}]|_{{\mathcal P}_n},}}
\eneq
$n\in\N.$ By (2) of Lemma \ref{LBembed}, ${\rm Her}(a_{1-t_n})\cong M_{N_n}({{B_{n,1}}}).$
Moreover,  define, for each $n,$
\beq
&&S_n:=\bigoplus^{N_n} s_n\circ \sigma_A: A\to {\rm Her}(a_{1-t_n})=M_{N_n}({{B_{n,1}}}),\\
&&R_n:=s_n^r\circ \sigma_A: A\to {\rm Her}(a_{r_n}),\\
&&S_n^l:=\bigoplus^{2J_n} s_n\circ \sigma_A: A\to M_{2J_n}(B_{n,1})\subset  {\rm Her}(a_{1-t_n}),\andeqn\\
&&S_n': =\bigoplus^{N_n-2J_n} s_n\circ \sigma_A: A\to M_{N_n-2J_n}({{B_{n,1}}}).
\eneq
Note that each $s_n\circ \sigma_A$ is $F$-full in $B_{n,1}.$
{{By \eqref{Texistol-3+1} and the second part of
Lemma \ref{LWembB},
maps \linebreak
${{\bigoplus^{N_{n+1}}s_{n+1}}}\circ \phi_{w,b}$ and $({{s_{n+1}^r}}\circ \phi_{w,b}\oplus \bigoplus^{2J_n}s_n\circ \phi_{w,b}
\oplus \bigoplus^{N_n-2J_n}s_n\circ \phi_{w,b})$ induce the same map on ${\rm Cu}(\W)=(0, \infty].$
Since  $K_0(\W)=0,$ they induce the same map on ${\rm Cu}^\sim (\W).$
As  $K_i(\W)=\{0\}$ ($i=0,1$),  by Theorem 1.11 of \cite{RI},
there are unitaries
$v_{n,k}\in {{\widetilde{{\rm Her}({{a_{1-t_{n+1}})}}}}}$ such that
\beq\label{Texistol-4}
&&\lim_{k\to\infty}v_{n,k}^*({{\bigoplus^{N_{n+1}}s_{n+1}}}\circ \phi_{w,b}(c))v_{n,k}= ({{s_{n+1}^r}}\circ \phi_{w,b}\oplus \bigoplus^{2J_n}s_n\circ \phi_{w,b}
\oplus \bigoplus^{N_n-2J_n}s_n\circ \phi_{w,b})(c)
\eneq}}
for all $c\in \W.$
By replacing $v_{n,k}$ by $e^{\sqrt{-1} \theta}v_{n,k}$ for some $\theta\in (-\pi, \pi),$
we may assume that $v_{n,k}=1_{\widetilde{{\rm Her}(a_{1-t_{n+1}})}}+{\bar v}_{n,k}$ for some ${\bar v}_{n,k}\in {\rm Her}(a_{1-t_{n+1}}).$
Set $v_{n,k}'=1_{\widetilde{B}} +{\bar v}_{n,k},$ $k=1,2,...,$ and $n=1,2,....$
Define
\beq\nonumber
&&\Psi_n={{\Lambda_{n,0}}}\oplus S_n^l={{\Lambda_{n,0}}}\oplus \bigoplus^{2J_n} s_n\circ \sigma_A: A\to {\rm Her}(B_{n,0}\oplus M_{2J_n}(B_{n,1}))\subset
B_{n,0,1},\\\nonumber
&&{\rm where}\,\,\, B_{n,0,1}:={\rm Her}(B_{n,0}\oplus
M_{2J_n}(B_{n,1}))\oplus {\rm Her}(a_{r_{n+1}}),\\\nonumber
&&\Psi_n'={{\Lambda}}_{n+1,0}\oplus \bigoplus^{2J_n}s_n\circ \sigma_A\oplus s_{n+1}^r\circ \sigma_A: A\to {\rm Her}(B_{{{n+1}},0}\oplus M_{2J_n}(B_{n,1}))\oplus {\rm Her}(a_{r_{n+1}}).
\eneq
Define ${{\Lambda_{n}:=\Lambda_{n,0}}}\oplus S_{n}=\Psi_n\oplus S_n': A\to B$ and
define $\Lambda_{n+1}'={{\Lambda_{n+1,0}\oplus R_{n+1}\oplus S^l_n }} \oplus  S_n'$  $n=1,2,....$
{{Note that
\beq\label{Textstol-4+1}
\Lambda_{n+1}'=\Psi_n'\oplus S_n'.
\eneq}}
By \eqref{Texistol-4},
\beq\label{Texistol-5}
\lim_{k\to\infty}(v_{n,k}')^*{{\Lambda}}_{n+1}(a){{v'_{n,k}}}={{\Lambda}}_{n+1}'(a)\rforal a\in A.
\eneq
Now consider maps $\Psi_n,{{ \Psi_{n}'}}: A\to B_{{{n,}}1,0}$
({{recall $B_{n+1,0}\subset B_{n,0},$ see \eqref{Texitol-3++}).}} Then,
since $K_i(\W)=\{0\},$ {{by \eqref {KLadded},}}
\beq\label{Texistol-6}
[\Psi_n]|_{{\mathcal P}_n}=[\Psi_n']|_{{\mathcal P}_n}=x|_{{\mathcal P}_n}.
\eneq
Note that, viewing ${{\Lambda}}_{n,0}$ and ${{\Lambda}}_{n+1,0}$ as maps
from $A$ into ${\rm Her}(a_{t_n})$ (as $t_{n+1}\le t_n$) (computing in \linebreak
$U(M_{m(n)}({\widetilde{{\rm Her}(a_{t_n})}})/CU(M_{m(n)}({\widetilde{{\rm Her}(a_{t_n})}}))$ and
{{omitting $(j_{B_{t_{n+1}}}')^\dag$), {{we have that}}
\beq
&&\hspace{-0.4in}\overline{\lceil \Lambda_{n,0}(u)\rceil}\overline{ \lceil \Lambda_{n+1,0}(u^*)\rceil}=\overline{s_n^{(0)}(\lceil \phi_n(u)\rceil)}
\overline{s_{n+1}^{(0)}(\lceil\phi_{n+1}(u)\rceil^*)}\\
&&=\overline{s_n^{(0)}(\lceil \phi_n(u)\rceil)s_{n+1}^{(0)}(\lceil\phi_{n}(u)\rceil^*)s_{n+1}^{(0)}(\lceil\phi_n(u)\rceil)\lceil\phi_{n+1}(u)\rceil^*)}\\
&&=\overline{s_n^{(0)}(z_n^{(0)})s_{n+1}^{(0)}({z_n^{(0)}}^*)s_n^{(0)}(z_n^{(1)})s_{n+1}^{(0)}({z_n^{(1)}}^*)s_{n+1}^{(0)}(\lceil\phi_n(u)\rceil\lceil\phi_{n+1}(u)\rceil^*)}\\\label{Length-n3}
&&= \overline{s_n^{(0)}(z_n^{(0)})s_{n+1}^{(0)}({z_n^{(0)}}^*)}\cdot \overline{1} \cdot \overline{s_{n+1}^{(0)}(\lceil\phi_n(u)\rceil\lceil\phi_{n+1}(u)\rceil^*)}\hspace{0.6in}{\rm (by}\,\, \eqref{Length-n2}{\rm)}
\eneq}}
{{(recall the notation in \eqref{Texistol-nn0} and \eqref{Texistol-nn1}).
Recall by Theorem 4.4 of \cite{GLII} (see also Theorem 15.5 and the end of 3.9 of \cite{GLII} for notation),
\beq\label{Length-n4}
{\rm cel}(w)\le 7\pi\rforal w\in CU(M_m(\widetilde{D}))
\eneq
for any hereditary \SCA\, $D$ of $B.$}}

By \eqref{Length-n3}, \eqref{Length-n2+}, and \eqref{Length-n4}
 (computing in $M_{m(n)}({\widetilde{{\rm Her}(a_{t_n})}})$),
\beq
{\rm cel}({{\lceil \Lambda_{n,0}(u)\rceil \lceil \Lambda_{n+1,0}(u^*)}}\rceil)\le (t_n-t_{n+1})\lambda_n'+ \lambda_n+7\pi\rforal u\in \U_n.
\eneq
Then, by Lemma 4.2 of \cite{GLII} and by \eqref{TexistoI-1} (computing in $M_{(2J_n+1)m(n)}({\widetilde{{\rm Her}(a_{t_n})}})$)
\beq\label{Texistol-7}
&&{\rm cel}(\lceil \Psi_n(u)\rceil \lceil ({{\Lambda_{n+1,0}}}\oplus S_n^l)(u)\rceil^*)\\
&&\,\,\,\,\,={\rm cel}(\lceil{{ \Lambda_{n,0}(u)\rceil \lceil \Lambda_{n+1,0}(u)^*}}\rceil \oplus S_n^l(uu^*))
\le (\lambda_n'+\lambda_n+7\pi)/J_n+2\pi.
\eneq
It follows from the definition of $R_{n+1}$
(computing
in $M_{m(n)}({\widetilde{{\rm Her}(a_{r_{n+1}})}})$) that
\beq
{\rm cel}(R_{n+1}(u^*))\le {{L_1(u)}}\rforal u\in  \U_n.
\eneq
Hence
\beq\label{Texistol-8}
{\rm cel}(\lceil \Psi_n(u)\rceil \lceil \Psi_n'(u)^*\rceil)\le {{1+2\pi+L_1(u)}}= L(u)\rforal u\in \U_n.
\eneq
{{Recall $N_n-2J_n\ge 4(K_n+1)2J_n$ (see \eqref{est-Kn}).}}
Using \eqref{Texistol-8}, \eqref{Texistol-6} and the fact that $s_n\circ \sigma_A$ is $F$-full in $B_{n,1},$ and
applying Theorem \ref{thm:LinStableUniqueness}, we obtain a unitary $u_n'\in {\widetilde{B}}$ such that
\beq
(u_n')^*({{\Lambda_{n+1}'(a))u_n'\approx_{\ep_n/2} \Lambda_n(a)}}\rforal a\in \F_n.
\eneq
It follows from \eqref{Texistol-5} that there is a unitary $u_n\in {\widetilde{B}}$ such that
\beq\label{Texistol-9}\
(u_n)^*{{\Lambda_{n+1}(a)u_n\approx_{\ep_n} \Lambda _n(a)}}\rforal a\in \F_n.
\eneq
Define {{$\Lambda_1''(a)=\Lambda_1(a),$ $\Lambda_n''(a)=u_1^*\cdots 
u_{n-1}^*\Lambda_n(a)u_{n-1}\cdots u_1$}}
for all $a\in A,$ $n=1,2,....$   Then, by \eqref{Texistol-9},
\beq
{{\Lambda_n''(a)\approx_{\ep_n} \Lambda_{n+1}''(a)}}\rforal a\in  \F_n, \,\,\,n=1,2,....
\eneq
It follows that, for any $m>n,$
\beq
{{\|\Lambda_n''(a)-\Lambda_m''(a)\| }}<\sum_{j=n}^m\ep_n\rforal a\in {\mathcal F}_n
\eneq
Note that $\lim_{n\to\infty}\sum_{j=n}^\infty \ep_n=0.$
Since $\F_n\subset \F_{n+1}$ and $\bigcup \F_n$ is dense in the unit ball,
we conclude that, for each $a\in A,$ $\{{{\Lambda_n''(a)}}\}$
is Cauchy in $B.$
Let $\Phi(a)= {{\lim_{n\to\infty}\Lambda_n''(a)}}$ for each $a\in A.$  It is clear that $\Phi$ is a positive linear map.
Since
\beq
\lim_{n\to\infty}\|{{\Lambda_n''(ab)-\Lambda_n''(a)\Lambda_n''(b)}}\|=0\rforal a, b\in A,
\eneq
$\Phi: A\to B$ is an \hm. Since $\sigma_A$ is an embedding, $\Phi$ is injective.
Finally, by \eqref{Texistol-6},  we have
\beq
KL(\Phi)=x.
\eneq
\end{proof}

\begin{thm}\label{TExpon}
Let $A$ be a separable amenable \CA\, with UCT.
Suppose that $A$ is $\W$-embeddable.
Let $B$ be a separable simple \CA\, with finite nuclear dimension {{and}}
with continuous scale which satisfies the UCT and
$K_0(B)={\rm ker}\rho_B.$

Let $\kappa\in KL(A, B)$ and let $\kappa_{ku}: K_1(A) \to 
U({\widetilde B})/CU({\widetilde B})$ be  {{a \hm\, which is}} compatible
with $\kappa,$ i.e., 
$$\kappa(z)=\Pi_{1,cu}\circ \kappa_{ku}(z),$$
for all $z \in K_1(A)$.
Then there exists a monomorphism $h: A\to B$ such that 
$$KL(h)=\kappa\andeqn
h^{\ddag}=\kappa_{ku}.$$

\end{thm}


\begin{proof}
Note that we are fixing  {{injective \hm s $J_A: K_1(A)\to U(\widetilde{A})/CU(\widetilde{A})$
and $J_B: K_1(B)\to U(\widetilde{B})/CU(\widetilde{B})$ which split following short exact sequences:}}
\beq\nonumber
0 \longrightarrow U_0(M_\infty(\widetilde{A}))/CU(M_\infty(\widetilde{A}))
\longrightarrow U(M_\infty(\widetilde{A}))/CU(M_\infty(\widetilde{A}))\stackrel{\Pi_{1,cu}}{\rightleftarrows}_{J_A} K_1(A)\to 0\andeqn\\\nonumber
0 \longrightarrow U_0(M_\infty(\widetilde{B}))/CU(M_\infty(\widetilde{B}))
\longrightarrow U(M_\infty(\widetilde{B}))/CU(M_\infty(\widetilde{B}))\stackrel{\Pi_{1,cu}}{\rightleftarrows}_{J_B} K_1(B)\to 0,
\eneq
respectively (see \ref{Dcoset}).
Recall that for a \hm\, $\rho : A \rightarrow B$, we let $\rho^{\ddag}: K_1(A) \rightarrow 
U(\widetilde{B})/CU(\widetilde{B})$ denote the induced map.

Note also that $K_i(A)$ is a countable abelian group for $i=0,1$.  By
Theorem 7.11 of \cite{GLII}, there is a stably projectionless simple \CA\, $C$, in the classifiable
class and with continuous scale, such that
$K_i(C)\cong K_i(A),$ $i=0,1$, and $T(C)=T(B).$
Let $\iota_i: K_i(A)\to K_i(C)$ be a group isomorphism.
By Theorem \ref{T-existoI}, 
let $\phi: A\to C$ be a *-embedding such that $\phi_{*i}=\iota_i$ ($i=0,1$). 
Fix {{an injective \hm\, $J_C: K_1(C)\to U(\widetilde{C})/CU(\widetilde{C})$
such that $\Pi_{1,cu}\circ J_C={\rm id}_{K_1(C)}$}}
and consider the the group homomorphism 
$$\phi_D: K_1(A) \to U_0({\widetilde{C}})/CU({\widetilde{C}})$$ 
induced by $\phi,$ {{i.e., $\phi_D:=\phi^\dag\circ J_A-J_C\circ \phi_{*1}.$
In particular, $\phi_D(x)\in U_0(\widetilde{C})/CU(\widetilde{C}).$}}

Consider the group homomorphism $\lambda: K_1(C) \to 
U(\widetilde{C})/CU(\widetilde{C})$
which is defined by
\beq
\lambda(z)={{J_C(z)}}-\phi_D \circ \iota_1^{-1}(z).
\eneq
for all $z \in  {{K_1(C).}}$

Define $\lambda_1: U(\widetilde{C})/CU(\widetilde{C}) \rightarrow U(\widetilde{C})/CU(
\widetilde{C})$
by 
$$\lambda_1|_{U_0(\widetilde{C})/CU(\widetilde{C})}=
{\rm id}_{U_0(\widetilde{C})/CU(\widetilde{C})} \makebox{  and  }
{\lambda_1}|_{J_C(K_1(C))}=\lambda\circ J_C^{-1},$$
where $J_C^{-1}=\Pi_{1,cu}|_{J_C(K_1(C))}.$
By  {{Lemma 12.10 of \cite{GLII},}}
there is a \hm\, $j: C\to C$ such that $KK(j)=KK({\rm id}_C),$
$j_T={\rm id}_{T(C)}$ and $j^{\dag}=\lambda_1.$

Let $\psi: A\to C$ be defined by $$\psi=j\circ\phi.$$
Then,  for all $x\in K_1(A),$
\beq\nonumber
\psi^\ddag(x)&=&((j\circ \phi)^\dag\circ J_A)(x)=j^\dag\circ \phi^\dag\circ J_A(x)\\\nonumber
&=
&j^\dag(\phi_D(x)+J_C\circ \phi_{*1}(x))=\lambda_1(\phi_D(x)+J_C\circ \phi_{*1}{{(x)}})\\\nonumber
&=&\phi_D(x)+\lambda\circ \phi_{*1}(x)\\\nonumber
&=&\phi_D(x)+J_C\circ \phi_{*1}(x)-\phi_D\circ \iota_1^{-1}(\phi_{*1}(x))\\\label{TExpon-10}
&=&\phi_D(x)+J_C\circ \phi_{*1}(x)-\phi_D(x)={{J_C\circ \iota_1(x).}}
\eneq
By the UCT, since $\iota_i$ is an isomorphism ($i=0,1$), it gives a KK equivalence and hence,
there is a $\zeta\in KK(C,A)$
such that $$\zeta\times KK(\psi)=KK({\rm id}_{C}).$$
Let $\overline{\zeta}$ be the element in $KL(C,A)$ induced by $\zeta.$
By Lemma 12.10 of \cite{GLII}, there is
a \hm\, $h_1: C\to B$ such
that $KL(h_1)=\kappa\circ \overline{\zeta},$ $(h_1)_T^{-1}$ is the identification
of $T(C)$ and $T(B)$ and  
$$h_1^\dag{{|_{J_C(K_1(C))}}}=
\kappa_{ku}\circ \iota_1^{-1}\circ J_C^{-1}.$$
{{It follows from \eqref{TExpon-10} that  if we define $h:=h_1\circ \psi: A\to B,$ then
$$
h^\ddag=h_1^\dag\circ \psi^\dag\circ J_A=(\kappa_{ku}\circ \iota_1^{-1}\circ J_C^{-1})\circ \psi^\ddag
=(\kappa_{ku}\circ \iota_1^{-1}\circ J_C^{-1})\circ J_C\circ \iota_1=\kappa_{ku}.
$$}}
Then one verifies that the map $h.$
satisfies the requirements.

\end{proof}

\section{Quasidisagonal extensions by $\W$}

The following proposition is an easy fact and known to experts.
We include a proof for convenience.
\begin{prop}\label{Pmutproj}
Let $B$ be a $\sigma$-unital \CA\, and let $p\in M(B)\setminus B$ be a projection. 
Suppose that 
$pBp={\rm Her}(a)$ for some $a\in {{B}}_+$. 

Then $a^{1/n}\to p$  (as $n\to\infty$) in the strict topology {{(of $B$ not just of ${\rm Her}(a)$).}}
Moreover, $p$ is the open projection corresponding to ${\rm Her}(a).$

\end{prop}

\begin{proof}
Let $\{e_n\}$ be an approximate identity of $B.$
Then $pe_np$ increasingly converges to $p$ in the strictly topology.
Fix any $x\in B_+.$ Let $\ep>0.$ There exists an integer $N\ge 1$ such that, for all $n\ge N,$
\beq
\|x(p-pe_np)\|<\ep/2
\eneq
Note that $pe_np\in {\rm Her}(a).$ Therefore, there exists $N_n\ge 1$ such that, for all $m\ge N_n,$
\beq
\|a^{1/2m}pe_npa^{1/2m}-pe_np\|<\ep/2(\|x\|^2+1).
\eneq
Note that 
\beq
a^{1/2m}pe_npa^{1/2m}=pa^{1/2m}e_na^{1/m}p.
\eneq 
It follows that
\beq
(p-pa^{1/2m}e_na^{1/m}p)\approx_{\ep/2(\|x\|^2+1)}p-pe_np.
\eneq
Therefore,
\beq
\|x(p-pa^{1/m}p)\|^2\le \|x(p-pa^{1/2m}e_na^{1/2m}p)x\|\approx_{\ep/2} \|x(p-pe_np)x\|<\ep/2.
\eneq
In other words,  as $pa^{1/m}p=a^{1/m},$ 
\beq
\lim_{m\to\infty}\|x(p-a^{1/m})\|=0\rforal x\in B.
\eneq
Similarly, 
\beq
\lim_{m\to\infty}\|(p-a^{1/m})x\|=0\rforal x\in B.
\eneq
So $a^{1/n}\to p$ in  the strict topology. 

On the other hand, 
$a^{1/m}$ also converges  to the open 
projection $p'$ in the strict topology in ${\rm Her}(a).$
It follows that $p'=p.$

\end{proof}


The next lemma should also be known.  

\begin{lem}\label{LK0(M(B))}
Let $B$ be a separable  simple  \CA\,  with continuous scale  such that $B$ and $B\otimes {\mathcal K}$ have almost stable rank one.
Suppose that $Cu(B)=V(B)\sqcup {\rm LAff}_+(T(B)).$ 

Then, if $p, \, q\in M_m(M(B))\setminus M_m(B)$ are two projections (for any integer $m\ge 1$) such that $\tau(p)=\tau(q)$
for all $\tau\in T(B),$ 
then  $p\approx q$  in $M(B).$
Moreover, 
\beq\label{LK0(M(B))-2}
K_0(M(B))\cong {\rm Aff}(T(B))\andeqn
K_0(M(B))_+ \cong {\rm Aff}_+(T(B)). 
\eneq

In fact, 
\beq\label{LK0M(B))-3}
V(M(B)) = V(B) \sqcup {\rm Aff}_+(T(B)).
\eneq

\end{lem}

\begin{proof}
Let $\{e_{i,j}\}$ be a system of matrix units for $\mathcal K.$ 
In what follows we will identify $B$ with $e_{1,1}(B\otimes \mathcal K)e_{1,1}.$
We also note that, in this way, we identify $M(B)$ with $e_{1,1}M(B\otimes \mathcal K)e_{1,1}.$
In what follows in this proof, we also identify $1_m:=\sum_{i=1}^m e_{i,i}$ with 
the unit of $M_m(M(B)).$ 
Moreover, for each $\tau\in T(B),$ we will also use $\tau$ for the extension of $\tau$ on $B\otimes \mathcal K$
as well as on $M(B\otimes \mathcal K).$

Set $C=B\otimes \mathcal K.$

Let $p,\, q\in M(C))\setminus C)$ be two projections such 
that $\tau(p)=\tau(q)$ for all $\tau\in T(B).$
Let $a\in pCp$ and $b\in qCq$ be strictly positive elements of $pCp$ and $qCq,$ respectively.
Note that neither $a$ nor $b$ are projections, as $p, q\not\in C.$
Then the fact $\tau(p)=\tau(q)$ for all $\tau\in T(B)$
implies that
\beq
d_\tau(a)=d_\tau(b)\rforal \tau\in T(B).
\eneq
It follows that $a\sim b$ in ${\rm Cu}(B).$ Since  $B\otimes \mathcal K$ has almost stable rank one, 
there exists a partial isometry $v\in C^{**}$ such 
that  $c:=vav^*$ is a strictly positive element of ${\rm Her}(b),$ 
\beq
v^*va=av^*v=a, \andeqn va, v^*b\in C.
\eneq
{ (See Proposition 3.3 of \cite{Rz} and the paragraph above it).}
By Proposition \ref{Pmutproj}, 
\beq
va^{1/m}\to vp \,\,\,{\rm in\,\,\, the \,\,\, strict\,\,\,topology}.
\eneq
Therefore $v = vp\in M(C).$
Also, by Proposition \ref{Pmutproj}, $va^{1/m}v^*=c^{1/m}\to q$  in the strict topology.
It follows  $vpv^*=q.$

From what has been just proved, we conclude that, if $p, q\in M_m(M(B))\setminus M_m(B)$ (for some integer $m\ge 1$)
are two projections 
and $\tau(p)=\tau(q)$ for all $\tau\in T(B),$ then $p$ and $q$ are equivalent in $M_m(M(B)).$

Let $p\in M_m(M(B))$ be a projection. Then  $\tau(p)$ may be viewed {{as}} a function in ${\rm LAff}_+(T(B)).$ 
 However, $1_m-p\in M_m(M(B))$ is also a projection. Therefore $\tau(1_m-p)\in {\rm LAff}_+(T(B))$ ($\tau\in T(B)$). 
 It follows that $\tau(p)$ is an affine function in ${\rm Aff}_+(T(B)).$
 This implies 
 that the map
 \beq
 \rho: K_0(M(B))\to \Aff(T(B))
 \eneq 
is an order preserving \hm, and we just proved that the map $\rho$ is injective. 

We now show that $\rho$ is surjective.
By the assumption, for any $f\in {\rm Aff}_+(T(B))\setminus \{0\},$  there is a nonzero positive element
$a\in B$   such that $d_\tau(a)=f(\tau)$ for all $\tau\in T(B).$ 
It follows from Kasparov's absorption theorem  {{(Theorem 2 of 
{ \cite{KasparovAbsorb}})}} that 
there is a projection $p_1\in M(B\otimes \K)$ such that  
{{$p_1(B \otimes \K)\cong \overline{a(B \otimes \K)}$, where 
the isomorphism is unitary isomorphism of  
Hilbert $B \otimes \K$-modules.}}  

{{By \eqref{Pmutproj},}} replacing $a$ with a Cuntz equivalent positive
element if necessary, we may assume that $a^{1/m}$ converges to $p_1$
in the strict topology {on $M(B\otimes \K).$} 
It follows that $\tau(p_1)=d_\tau(a)$ for all $\tau\in T(B).$ 

{{There exists an integer $m$ such that $m\ge f(\tau)+1$ for all $\tau\in T(B).$ 
Let $g=m-f.$ Then $g\in {\rm Aff}_+(T(B))\setminus \{0\}.$ From what has just been proved, we obtain a projection 
$q\in M(B\otimes \mathcal K)$ such that $\tau(q)=g.$
\Wlog, we may assume that $p_1\perp q.$ Then $p_1+q=e$ is a projection in $M(B\otimes \mathcal K)$
such that $\tau(e)=\tau(1_m)$ for all $\tau\in T(B).$  From the first part of the proof above, 
we conclude that there is $v\in M(B\otimes \mathcal K)$ such that
\beq
v^*v=e\andeqn vv^*=1_m.
\eneq
This implies that $p:=vp_1v^*\le 1_m.$ In other words, $p\in M_m(M(B))$ (see the first part of the paragraph of this proof).
Note that $\tau(p)=f.$  This proves that the map $\rho$ is surjective. The rest of the proposition also follow.}} 

\end{proof}

\begin{rmk}\label{RmZCuntz}
{{\rm As in the beginning of the proof of Theorem of \ref{LBembed}, 
if $A$ is a separable simple finite \CA\, which is ${\mathcal Z}$-stable, has continuous scale
{{and every 2-quasi-trace of $A$ is a trace,}} 
then $Cu(B)=V(B)\sqcup {\rm LAff}_+(T(B)).$  Later on, we often assume that $B$ is a separable 
simple finite \CA\, which is ${\mathcal Z}$-stable and has continuous scale.}}
\end{rmk}

\begin{thm}\label{K-six}
Let $B$ be a  $\sigma$-unital,  stably projectionless, finite, simple,
$\Z$-stable, amenable  \CA\ with  unique tracial state $\tau_B$.

Then $K_0(M(B))=\R,$ $K_1(M(B))=\{0\},$  {{$K_0(\C(B))=\R\oplus K_1(B),$ and}}
$K_1(\C(B))={\rm ker}\rho_B = K_0(B).$ 
In particular, $K_0(M(\W))=\R,$
$K_1(M(\W))=\{0\},$ 
$K_0(\C(\W))=\R$ and $K_1(\C(\W))=\{0\}.$

Moreover, if $p, q\in M_m(M(\B)){{\setminus M_m(B)}}$ (for some integer $m\ge 1$) and $\tau_B(p)=\tau_B(q),$ 
then 
there exists $v\in M(\B)$ such that $v^*v=p$ and $vv^*=q.$
\end{thm}

\begin{proof}
By \cite{NgU(Razak)},  $K_1(M(B))=\{0\}.$ It follows from Lemma \ref{LK0(M(B))} that 
$K_0(M(B))=\R.$
Thus 
the six-term exact sequence 
\[
\begin{array}{ccccc}
K_0(B) & \rightarrow & K_0(M(B)) & \rightarrow & K_0(\C(B))\\
\uparrow & & & & \downarrow\\ 
K_1(\C(B)) &  \leftarrow & K_1(M(B)) & \leftarrow & K_1(B)\\
\end{array}
\]
becomes
\[
\begin{array}{ccccc}
K_0(B) & \rightarrow & \R& \rightarrow & K_0(\C(B))\\
\uparrow & & & & \downarrow\\ 
K_1(\C(B)) &  \leftarrow & 0 & \leftarrow & K_1(B)\\
\end{array}
\]
Note that the map from $K_0(B)$ into $\R=K_0(M(B))$ is induced by 
the map $\rho_B: K_0(B)\to \Aff(T(B))=\R.$ 
However, by Corollary A7 of \cite{ElliottGongLinNiu}, since $B$ has a unique tracial state,
$K_0(B)={\rm ker}\rho_B.$ In other words, $\rho_B(K_0(B))=\{0\}.$

It follows from the commutative diagram that  
$$
K_1(\C(B))={\rm ker}\rho_B\andeqn
$$
$$
0\to \R\to K_0(\C(B))\to K_1(B)\to 0.
$$
{{Since $\R$ is divisible, we may write $K_0(\C(B))=\R\oplus K_1(B).$}}
{{Note that $B$ is stably projectionless. Therefore}}
the last statement follows from \ref{LK0(M(B))} {{(see also Remark \ref{RmZCuntz}).}}

\end{proof}

\begin{df}
{\rm Let $\A$ be a separable \CA, and let $\B$ be a nonunital  and 
$\sigma$-unital \CA.

A trivial extension $\phi : \A \rightarrow M(\B)$ is said to
be \emph{diagonal} if 
$\phi$ is quasidiagonal as in Definition \ref{df:QDExt} and
Proposition \ref{prop:QDExt}, with the additional property that the maps
$\phi_n$ in Proposition \ref{prop:QDExt} can be taken to be
\hm s. 
 
In the above setting, we often write $\phi = \bigoplus_{n=1}^{\infty}
\phi_n$ {{(converges in the strict topology).}}}
\label{df:DiagonalExtension}
\end{df}

\begin{df}{\bf $\T_d$ extensions}\label{DTextension}\,\,\,
Let $B$ be a separable simple nonunital  \CA, with continuous scale and let  
 $C$ be a separable \CA.  A monomorphism $\sigma: C\to M(B)$ is called 
a  \emph{$\T_d$ extension with model $\sigma_0$}
if $\pi \circ \sigma$ is nonunital and if 
$\sigma$ is a diagonal essential extension of the form
$$\sigma = \bigoplus_{n=1}^{\infty} \bigoplus^n \phi_n\circ \sigma_0 
= \bigoplus_{n=1}^{\infty} (\overbrace{\phi_n\circ \sigma_0 \oplus ... 
\oplus \phi_n\circ \sigma_0}^n).$$
 Here, $\sigma_0: C\to B$ is a fixed injective *-\hm\, such 
that $\sigma_0(e_C)$ 
is a strictly positive element of $B$, where $e_C\in  C$ is a strictly positive element of $C$.  

More precisely,  this means the following: 

\begin{enumerate}
\item There exists a system $\{ b_n \}$ of quasidiagonal units
for $B$. 
\item There exists a nonzero positive element
$b_{n,1} \in \overline{b_n B b_n}$ such that 
$${\rm Her}(b_{n,1})\otimes M_ n\cong M_n(\overline{b_{n,1}B b_{n,1}}) =
\overline{b_n B b_n}$$ for all $n \geq 1.$
Moreover, we may write 
\beq
b_n=\sum_{i=1}^nb_{n,j},
\eneq
where $b_{n,j}:=b_{n,1}\otimes e_{j,j},$ and 
 $\{e_{i,j}\}\subset M_n$ is  a system of matrix units
\item $\phi_n: B\to \overline{b_{n,1}Bb_{n,1}}$ is an isomorphism such that 
${{\phi_n\circ \sigma_0}}(e_C)=b_{n,1}.$
for all $n \geq 1,$ {{and $(\overbrace{\phi_n\circ \sigma_0 \oplus ... 
\oplus \phi_n\circ \sigma_0}^n):  C\to M_n({\rm Her}(b_{n,1}))\subset {\rm Her}(b_n)$}}
is the diagonal map.
\end{enumerate}

\end{df}

\begin{rmk}\label{rmk:Textension}
{\rm With notation as in Definition \ref{DTextension}, let us
suppose that $F: C_+\setminus \{0\}\to \N\times (0, \infty)$ is a map
such that $\sigma_0$ is $F$-full. ($F$ exists  by Lemma \ref{lem:FMapExists}).
Then each $\phi_n\circ \sigma_0$ is also $F$-full. 

Note also that $\oplus_{n=1}^m \bigoplus^n \phi_n\circ \sigma_0(c)$ 
converges strictly 
to $\sigma(c)$ for all $c\in C$ (as $m\to\infty$).
{{Note that $\sigma=\phi\circ \sigma_0.$ }}

Finally, note that 
our definition of $\T_d$ extension requires that
$\pi \circ \sigma$ be a nonunital essential extension.}
\end{rmk}

\begin{rmk}\label{Rsigmakk=0}
With notation as in Definition \ref{DTextension}, note that
if $KK(\sigma_0) = 0$
then, {{since $\sigma=\psi\circ \sigma_0$ and $\phi$ is a *-\hm,}} 
$KK(\pi \circ \sigma) = 0.$

Also, when $C$ is amenable and satisfies the UCT, and when $B$ is 
stably projectionless, $\Z$-stable, and has
unique tracial state,
 since $K_*(M(B)) \cong (\mathbb{R}, 0)$ is divisible,
a sufficient condition for the above is that 
$$K_0(\sigma_0) = 0.$$ 
\end{rmk}

\begin{prop}\label{TextWC}
Let $A$ be a separable amenable \CA\, which is $\W$ embeddable. 
Then there exists a $\T_d$ extension $\tau: A\to \C(\W)$. 
Moreover,
$$KK(\tau) = 0.$$ 
\end{prop}

\begin{proof}
Fix a *-embedding $\sigma_A: A\to \W,$ 
which maps strictly positive elements to strictly positive elements.
Denote by $\tau_W$ the unique tracial state of $\W.$
Fix a system of quasidiagonal units $\{b_k\}$ as in \ref{Dqdunit}. 
Passing to a subsequence if necessary, we may assume that 
\beq
\sum_{k=n+1}^{\infty}d_{\tau_W}(b_k)<\frac{1}{n}d_{\tau_W}(b_n)\rforal n.
\eneq
Let $t_n=\frac{1}{n+1}d_{\tau_W}(b_n),$ $n\in \N.$
There is an element $a_n\in {\rm Her}(b_n)$ with $d_{\tau_W}(a_n) \leq t_n$ 
such that  
$M_n({\rm Her}(a_n)) \subseteq {\rm Her}(b_n)$ (by  {{Theorem 6.6 of}} 
\cite{ElliottRobertSantiago} and by strict comparison). 
There is, for each $n,$ an isomorphism $\phi_n: \W \to {\rm Her}(a_n).$
Define $\sigma: A\to M(\W)$ by
$\sigma(a)=\sum_{n=1}^\infty (\bigoplus^n \phi_n\circ \sigma_A)(a)$ for all 
$a\in A.$   Note that, since $\{b_n\}$ is a system of quasidiagonal units, the sum
above converges in the strict topology for each $a\in A.$
One then checks, from Definition \ref{DTextension}, that 
$\pi\circ \sigma$ is a $\T_d$ extension  with model $\sigma_A.$ 

That $KK(\pi \circ \sigma) = 0$ follows from 
Remark  \ref{Rsigmakk=0}.
\end{proof}


\begin{prop}\label{Pkk0}
Let $C$ a separable amenable \CA\, which is $\W$ embeddable and satisfies the UCT, and let  $\phi: C\to M(\W)$ be a \hm.
Then $\phi_{*0}({\rm ker}\rho_{f,C})=\{0\}$ and $\phi_{*1}=0.$ 

If $X$ is a 
connected and locally connected compact metric space, and 
$C:=C_0(X\setminus \{x_0\})$ for some $x_0 \in X$,  
then $KK(\phi)=0$ and $KK(\pi\circ \phi)=0.$ 

\end{prop}

\begin{proof}
Recall that $K_0(M(\W))=\R$ and $K_1(M(\W))=\{0\}.$ 
The first part follows from the fact that, if $p, q\in M_n(M(\W))$ 
(for some integer $n$) are two projections and 
$\tau_W(p)=\tau_W(q),$  then there exists $v\in M_n(M(\W))$  such that
$v^*v=p$ and $vv^*=q.$ 
(See Theorem \ref{K-six}.)

In case that $C=C_0(X\setminus \{x_0\}),$  {{since $X$ is connected,}} $K_0(C)={\rm ker}\rho_C.$
It follows that $\phi_{*i}=0,$ $i=0,1.$ Since $K_0(M(\W))=\R$ is divisible,
${\rm Ext}_{\mathbb{Z}}(K_1(C), K_0(\R))=\{0\}.$  By the UCT, $KK(\phi)=0.$
Then $KK(\pi\circ \phi)=0$ follows. 

\end{proof}

Denote by $\mathcal D$ the class of simple \CA s  defined in Definition of 8.1 of \cite{eglnp}.
Suppose that $A\in \mathcal D.$ Then, for any integer $k\ge 1,$ $M_k(A)\in \mathcal D$ (see 8.5 of \cite{eglnp}). 
Moreover, $A$ is stably projectionless (see 9.3 of \cite{eglnp}).  We note that $\W\in \mathcal D$ (see 9.6 of \cite{eglnp}).

Let us quote the following lemma for the convenience of the reader.

\begin{lem}[Theorem 4.4 of \cite{GLII}]\label{Lexpincum}
Let $A$ be a separable  simple \CA\, in ${\mathcal D}$  and let $u\in CU(M_m({\widetilde A})).$
Then $u\in U_0({\widetilde A})$ and ${\rm cel}(u)\le {{7\pi}}.$
\end{lem}

\begin{proof}
Note, as mentioned above,  that $M_m(A)\in \mathcal D.$  Let $\pi: M_m(\widetilde{A})\to M_m(\mathbb C)$ be the quotient map. 
Let $w=\pi(u)$ be the scalar unitary. Denote by $W\in M_m(\mathbb C \cdot 1_{\widetilde{A}})$ the same scalar matrix. 
Then $W^*u\in \widetilde{M_m(A)}.$ By Theorem 4.4 of \cite{GLII}, $W^*u\in U_0(M_m(\widetilde{A}))$
and ${\rm cel}(W^*u)\le 6\pi.$  Since $W\in M_m(\mathbb C \cdot 1_{\widetilde{A}}),$ we conclude that $u\in U_0(M_m(\widetilde{A}))$
and ${\rm cel}(u)\le 7\pi.$

\end{proof}

\begin{lem}\label{Uniq-2}
Let $C$ be a separable amenable \CA\, which is $\W$ embeddable and satisfies the UCT.
Let $\sigma :  C\rightarrow M(\W)$ be a
$\T_d$ extension, and let
$\psi : C \rightarrow M(\W)$ be  a {{diagonal}} c.p.c.~map of 
the form $$\psi = \bigoplus_{n=1}^{\infty} \psi_n$$
as in Proposition \ref{prop:QDExt}   {{such that $\pi\circ { \psi}$ is a nonunital essential extension.}}

Then, there is a diagonal extension $h: {{C}}\to \C(\W)$ such that
$$ \pi\circ \sigma \oplus \pi \circ \psi \sim^u  {{\pi \circ \sigma
\oplus h.}}$$
\end{lem}


\begin{proof}
Fix  a strictly positive element  $e_C\in C$  with $\|e_C\|=1.$
{{By working in  $M_2(M(\W)),$
\wilog, we may assume that ${\rm ran}(\psi)\perp {\rm ran}(\sigma)$
(see Proposition \ref{PBDFsum}).}}

Since $\sigma$ is a $\T_d$ extension, using {{a variation on}} the
notation of Definition \ref{DTextension},
we write
$$\sigma = \bigoplus_{n=1}^{\infty} {{\bigoplus^{n+1}}} \phi_n\circ \sigma_0.$$

We {{also}} write $\bigoplus^{{n+1}}\phi_n\circ \sigma_0=
{{\sigma_{n,0}}}\oplus \sigma_{n,{{1}}}\oplus \cdots \oplus  \sigma_{n,n}$
and $\sigma=\bigoplus_{n=1}^\infty {{\bigoplus_{j=0}^n}} \sigma_{n,j}.$

Continuing to follow Definition \ref{DTextension}
let
$$b_{n,j} := \sigma_{n,j}(e_C)$$
and let $b_n$ be as in Definition \ref{DTextension},
for all $n, j$.

Also, let $\{ a_n \}$ be the system of quasidiagonal units
for $\W$
from Proposition \ref{prop:QDExt} that corresponds to
$\{ \psi_n \}$.
{{Recall (see \ref{prop:QDExt}) that
\beq
\lim_{n\to\infty}\|\psi_n(a)\psi_n(b)-\psi_n(ab)\|=0\rforal a, b\in C.
\eneq}}

Since $\sigma$ is a $\T_d$ extension, as in Remark \ref{rmk:Textension},
there exists a map
$F : C_+ \setminus \{ 0 \} \rightarrow \mathbb{N} \times 
(0, \infty)$
such that for all $n, j$,
$\sigma_{n,j} : C\rightarrow \overline{b_{n,j} 
\W b_{n,j}}$
is $F$-full.

Let $\{ \epsilon_n \}_{n=1}^{\infty}$ be a strictly decreasing
sequence in $(0,1)$ such that
$\sum_{n=1}^\infty\epsilon_n <\infty.$

Let ${\mathcal F}_1\subset {\mathcal F}_2\subset \cdots \subset {\mathcal F}_n\subset \cdots $ be
a sequence of finite subsets of the unit ball of $C$ which is dense in the unit ball of $C.$

We will apply Theorem \ref{TK1stuniq}.
Note that, by Proposition \ref{Wprop},  $\W\in {\bf C}_{0,0,1,T,7}$, with
$T$ as in Proposition \ref{Wprop}.  Let $L:=7\pi+1.$
{{As $C$ is given, we fix maps $J,$ $\Pi_{cu}^-$ and $J^\sim$
as in \ref{Dcoset}.}}

For each $n,$ let $\delta_n>0,$  $\G_n\subset C$ be a finite subset, ${\mathcal P}_n\subset \underline{K}(C)$ be a finite
subset, $\U_n{{\subset \U_{n+1}\subset}} J^\sim(K_1{{(C}}))$
be
a finite subset, $\E_n\subset C_+\setminus \{0\}$ be  a finite subset,  and
$K_n$ be an integer associated with $\F_n$ and
{{$\ep_n/4$}} (as well as $F$ and $L$ above) as  provided
by  Theorem \ref{TK1stuniq}
(for \CA s in ${\bf C}_{0,0,1,T,7}$).

We may assume that $\delta_{n+1} < \delta_n$,  $\G_n \subseteq \G_{n+1}$,
{{$K_n < K_{n+1}$}}, and $\U_n \subset U(M_{m(n)}(\widetilde{C}))$
for all $n$.  \Wlog, we may assume that each $\psi_n$ is $\G_n$-$\dt_n$-multiplicative and $\lceil \psi_n(u) \rceil$ is well defined for all $u\in \U_n.$

Moreover, \wilog, we may also assume (see  Theorem 14.5 of
\cite{Lnloc}) that ,
for any $n,$ there is a group \hm\,
$$
\lambda_n: G(\Pi_1(\U_n))\to U(M_{m(n)}(\widetilde{{\rm Her}(a_n)}))/
CU(M_{m(n)}(\widetilde{{\rm Her}(a_n)}))\cong \Aff(T(\widetilde{\W}))/\mathbb{Z}
$$
such that
\beq\label{Uni-2-e-1}
{\rm dist}(\lambda_n(x), \Pi_{cu}(\lceil \psi_n(J^\sim(x))\rceil))<
\frac{1}{16\pi (n+1)} \rforal x\in \Pi_1(\U),
\eneq
where $G(\Pi_1(\U))$  is the subgroup generated  by the finite subset $\Pi_1(\U).$
Since $\Aff(T(\widetilde{{\rm Her}(a_n)}))/\mathbb{Z}$ is divisible, there is a \hm\, $\bar{\lambda_n}: K_1(C)\to \Aff(T(\widetilde{\W}))/\mathbb{Z}$
such that $\bar{\lambda_n}$ extends $\lambda_n.$

It follows from Theorem \ref{TExpon} that, for each $n,$  there is a monomorphism $h_n: {{C}} \to {\rm Her}(a_n)$
such that $KL(h_n)=KL(\psi_n)=0$ and
\beq\label{Uni-2-n-1}
h_n^\ddag= {{\bar{\lambda_n}}}
\eneq
Define $H: A\to M(\W)$ by
$H=\bigoplus_{n=1}^\infty h_n.$ Note, by \ref{Dqdunit}, the sum converges strictly and
$H$ gives a diagonal extension.


Since $\W$ is KK contractible, we may assume that
\beq\label{Uniq-2-KK0}
[\Sigma_{k=n}^m h_k]|_{{\mathcal P}_n}=[\Sigma_{k=n}^m\psi_k]|_{{\mathcal P}_n}=0\rforal m\ge n, \,\,\, n=1,2,....
\eneq

Throwing away finitely many terms and relabelling
if necessary, we may assume that
$$\sum_{n=1}^{\infty} d_{\tau_W}(a_n) < d_{\tau_W}({{b_{K_1,0}}}).$$

Let $\{ n_k \}_{k=1}^{\infty}$ be a subsequence of $\mathbb{Z}^+$
with
$n_1 = 1$ and
$n_k +2 < n_{k+1}$ for all $k$ such that
$$\sum_{l=n_k}^{\infty} d_{\tau_W}(a_l) < d_{\tau_W}({{b_{K_k, 0}}}).$$
By \eqref{Uni-2-n-1} and \eqref{Uni-2-e-1}, for any $u\in \U_{n_k},$  there is $v_l\in CU(M_{m(l)}(\widetilde{{\rm Her}(a_l)}))$ such that
\beq\label{Uniq-2-e-2-1}
h_l(u)\lceil \psi_l(u)\rceil ^*\approx_{1/16\pi(n+1)} v_l.
\eneq
It follows from Lemma \ref{Lexpincum} that, for all $u\in \U_{n_k}$,
\beq\label{Uniq-2-e-2}
{\rm cel}((\Sigma_{l=n_k}^{n_{k+1}-1}h_l)(u)\lceil (\Sigma_{l=n_k}^{n_{k+1}-1}\psi_l)(u)\rceil^*)\le
7\pi+1,
\eneq
 {{where the length is computed in $M_{m(n_k)}({\rm Her}(\sum_{l=n_k}^{n_{k+1}-1}a_l))$.}}

{{Since $\W$ has stable rank one, there is a unitary $U_k'\in {\widetilde{\W}}$ such that
\beq
(U_k')^*((\sum_{l=n_k}^{n_{k+1}-1} a_l)\W (\sum_{l=n_k}^{n_{k+1}-1} a_l ))U_k'
{{\subseteq}} {{\overline{b_{K_k,0} \W b_{K_k,0}}.}}
\eneq
For each $k,$ consider two maps
$$
{\rm Ad}\, U_k'\circ (\Sigma_{l=n_k}^{n_{k+1}-1}\psi_l),\,\,\,{\rm Ad}\, U_k'\circ (\Sigma_{l=n_k}^{n_{k+1}-1}h_l) : C\to  
{\rm Her}({{b_{K_k,0}}})\,\cong \W.
$$}}

Recall that $\psi_n$ is $\G_n$-$\delta_n$-multiplicative and
$\phi_n\circ {{\sigma_0}}$ is $F$-full
for all $n.$  Also keeping in mind of \eqref{Uniq-2-e-2} and
\eqref{Uniq-2-KK0} {{and
applying}}  Theorem \ref{TK1stuniq},
for all $k$,  there is a unitary ${{u_k'}}\in 
M_{K_k + 1}(\widetilde{{\rm Her}({{b_{K_k,0}}})})$
such
that
\beq\label{uniq-add-0}
u'_k({{{U_k'}^*}}\sum_{l=n_k}^{n_{k+1}-1}h_l(c){{U_k'}}+
{{\sum_{l=1}^{K_k} \sigma_{K_k, l}(c)}})(u_k)^*
\approx_{{{\epsilon_k/4}}}
(U_k')^*\sum_{l=n_k}^{n_{k+1} - 1} \psi_l(c)U_k' +
\sum_{l=1}^{K_k} \sigma_{K_k, l}(c)
\eneq
for all $c\in \F_k.$

For each $k,$ there are $e_k\in {\rm Her}(b_{{{K_k,0}}})_+$
and $e_k'\in {{{U_k'}^*{\rm Her}(\sum_{l=n_k}^{n_{k+1}-1}a_l)_+U_k'}}$
with $\|e_k\| \le 1$ and $\|e_k'\|\le 1$ such that, for all $c\in \F_k,$
\beq\label{uniq-add-1}
e_k'{U_k'}^*\sum_{l=n_k}^{n_{k+1} - 1} \psi_l(c)U_k'e_k'&\approx_{\ep_n/16}& {U_k'}^*\sum_{l=n_k}^{n_{k+1} - 1} \psi_l(c)U_k,\\\label{uniq-add-2}
e_k'{U_k'}^*\sum_{l=n_k}^{n_{k+1}-1}h_l(c)U_k'e_k'&\approx_{\ep_n/16}& {U_k'}^*\sum_{l=n_k}^{n_{k+1}-1}h_l(c)U_k',\\
\label{May720201AM} \sum_{l=1}^{K_k}{{e_k}}  \sigma_{K_k, l}(c)e_k
&\approx_{\ep_n/16}& \sum_{l=1}^{K_k} \sigma_{K_k, l}(c)
\eneq
\noindent where in  (\ref{May720201AM}), we identify
$M_{K_k}(Her(b_{K_k, 0}))$  with $Her(\sum_{l=1}^{K_k} b_{K,l})$.

Set $X_k=U_k'e_k'+\diag(\overbrace{e_k,e_k,\cdots, e_k}^{K_k}),$ $k=1,2,....$
Note $e_k'{U_k'}^*d=0$ for all
$d \in ran(\sigma)$.
Then, for all $c\in \F_k,$  by \eqref{uniq-add-1},
\beq\label{uniq-add-10}
&&X_k\left({U_k'}^*\sum_{l=n_k}^{n_{k+1} - 1} \psi_l(c)U_k' +
\sum_{l=1}^{K_k} \sigma_{K_k, l}(c)\right)X_k^*\approx_{\ep_n/16} \sum_{l=n_k}^{n_{k+1} - 1} \psi_l(c) +
\sum_{l=1}^{K_k} \sigma_{K_k, l}(c)\,\,\, \text{and},\,\, {\text{by\,\, \eqref{uniq-add-2},}}\\\label{uniq-add-11}
&&X_k^*(\sum_{l=n_k}^{n_{k+1}-1}h_l(c)+{{\sum_{l= 1}^{K_k} \sigma_{K_k, l}(c)}})X_k\approx_{\ep_n/16}
{U_k'}^*\sum_{l=n_k}^{n_{k+1}-1}h_l(c)U_k'+ {{\sum_{l=1}^{K_k} \sigma_{K_k, l}(c)}}.
\eneq

{{For all $k$, put $u_k=X_ku_k'X_k^*.$  Note $u_k\in {\rm Her}\left(\sum_{j = n_k}^{n_{k+1} - 1} a_j  + 
\sum_{l=1}^{K_k} b_{K_k, l}\right).$
Then we have, by  \eqref{uniq-add-11}, \eqref{uniq-add-0}, and  \eqref{uniq-add-10},}} for all $c\in \F_k,$
\beq\label{Uniq-2-conv}
u_k (\sum_{l=n_k}^{n_{k+1}-1}h_l(c)+
{{\sum_{l=1}^{K_k} \sigma_{K_k, l}(c)}})u_k^*\approx_{\ep_n/2}
\sum_{l=n_k}^{n_{k+1} - 1} \psi_l(c) + \sum_{l=1}^{K_k} \sigma_{K_k, l}(c).
\eneq


Let
$$Y =_{df} \sum_{j=1}^{\infty} {{u_j}}  \in M(\W),$$
where the sum converges strictly. {{Note that $\|Y\|\le 1.$}}

Let $c\in {\mathcal F}_k.$  Put
\beq\nonumber
&&\xi(c)_m=\sum_{k=1}^m\left(u_k
\left(\sum_{l=n_k}^{n_{k+1}-1} h_l(c)+\sum_{l=1}^{K_k} \sigma_{K_k, l}(c) \right)u_k^*
-
\sum_{l=n_k}^{n_{k+1} - 1} \psi_l(c) +
\sum_{l=1}^{K_k} \sigma_{K_k, l}(c)\right),\\\nonumber
&&\,\,\, m=1,2,...., \andeqn\\\nonumber
&&\hspace{1in} \xi(c)=Y({{H(c) \oplus}} \sigma(c))Y^*-
\psi(c)\oplus \sigma(c).
\eneq
It follows from \eqref{Uniq-2-conv} that (since $\sum_{k=m}^\infty \ep_k\to 0$ as $m\to\infty$)
\beq
\lim_{n\to\infty}\|\xi(c)_n-\xi(c)\|=0.
\eneq
Since $\xi(c)_n\in \W,$ one concludes that $\xi(c)\in \W$ for all $c\in {\mathcal F}_k,$ $k=1,2,....$
Thus, for any $c\in {\mathcal F}_k,$
\beq\label{Uniq-2+n1-1}
&&\pi(Y) (\pi\circ H(c)+ \pi \circ \sigma(c) {{)}} \pi(Y)^*
= \pi \circ \psi(c) + \pi \circ \sigma(c).
\eneq

{{
By a similar argument,
for any $c\in {\mathcal F}_k,$

\beq\label{May820208AM}
&&\pi\circ H(c)+ \pi \circ \sigma(c)
= \pi(Y)^* (\pi \circ \psi(c) + \pi \circ \sigma(c))\pi(Y).
\eneq
}}

Since ${\mathcal F}_k\subset {\mathcal F}_{k+1}$ for all $k$ and $\cup_{k=1}^\infty {\mathcal F}_k$ is dense in the unit ball
of $C,$ \eqref{Uniq-2+n1-1} and \eqref{May820208AM} implies that

\beq\label{Uniq-2+n3}
&&\pi(Y)(\pi\circ H(c)+ \pi \circ \sigma(c)) \pi(Y)^*
= \pi \circ \psi(c) + \pi \circ \sigma(c)\rforal c\in C
\eneq
and
\beq\label{Uniq-2+n4}
&&\pi\circ H(c)+\pi \circ \sigma(c) =
\pi(Y)^*( \pi \circ \psi(c) + \pi \circ \sigma(c) ) \pi(Y) \rforal c\in C.
\eneq

{{Set $d:=\sum_{n=1}^\infty (a_n + b_n).$  Then $ Y 
\in {\rm Her}(d).$
Since ${\rm Her}(d)^\perp\not=\{0\},$
$\pi(Y)^*\pi(Y)$ and $\pi(Y) \pi(Y)^*$
are each not invertible.
Since $C(\W)$ is purely infinite,
it has weak cancellation.}} Hence, by  Corollary 1.10 of \cite{LinExtRR0III},
$\pi(Y)=u(\pi(Y)^*\pi(Y))^{1/2},$ where $u\in U(C(\W)).$
By \eqref{Uniq-2+n3} and \eqref{Uniq-2+n4},
\beq
\pi(Y)^*\pi(Y)( \pi\circ H(c)+\pi \circ \sigma(c) )\pi(Y)^*\pi(Y)= \pi \circ H(c) +\pi \circ \sigma(c)\rforal c\in C.
\eneq
Hence
\beq\label{Uniq-2-n10}
&&(\pi(Y)^*\pi(Y))^{1/2}(\pi\circ H(c)+\pi \circ \sigma(c) )(\pi(Y)^*\pi(Y))^{1/2}=\pi\circ H(c)+\pi \circ \sigma(c)\rforal c\in C.
\eneq
Therefore,  by \eqref{Uniq-2-n10} and \eqref{Uniq-2+n3}),
\beq\nonumber
u(\pi\circ H(c)+\pi\circ \sigma(c))u^*&=&\pi(Y)(\pi\circ H(c)+\pi\circ \sigma(c))\pi(Y)^*\\\nonumber
&&=\pi\circ \psi(c)+\pi\circ \sigma(c)\rforal c\in C.
\eneq

Since $K_1(\C(\W)) = 0$ and since $\C(\W)$ is simple purely infinite,
$u$ can be lifted to a unitary in $M(\W)$.

\end{proof}



\begin{cor}\label{Lexpo-zeroabsb}
In Lemma \ref{Uniq-2}, if $\pi\circ \psi$ is in fact a diagonal extension, i.e., 
$\psi=\bigoplus_{n=1}^\infty \psi_n,$ where each $\psi_n$ is a \hm, and if 
$\psi_n^\ddag=0,$
then 
$$
\pi\circ \psi\oplus \pi\circ \sigma\sim^u \pi \circ \sigma.
$$

\end{cor}

\begin{proof}
Note that $KL(\psi_n)=0$ for all $n.$
Since $\psi_n^\ddag=0,$ $\psi_n(u)\in CU(M_{m(n)}(\W))$ (instead of \eqref{Uniq-2-e-2-1}) for all $u\in J^\sim(K_1(A))\cap U(M_{m(n)}(\widetilde{{\rm Her}(a_n)})).$
Therefore, in the proof of Lemma \ref{Uniq-2},  \eqref{Uniq-2-e-2} becomes, 
\beq
{\rm cel}\left(\sum_{l=n_k}^{n_{k+1}-1}\psi_l(u) \right)
\le 7\pi+1\rforal u\in \U_{n_k}.  
\eneq
Therefore, the proof of Lemma \ref{Uniq-2}
works when we take $h_n=0$ for all $n.$
\end{proof}

\begin{lem}\label{Lexp-permuzero}
Let $C$ be as in Lemma \ref{Uniq-2}.
Fix a sequence of \hm s 
$$\lambda_n\in {\rm Hom}(K_1(C), U(\W)/CU(\W)).$$
Then, for every system $\{ b_n \}$ of quasidiagonal units for $\W$,
there is a 
diagonal \hm\, $H :=\bigoplus_{n=1}^\infty h_n: C\to M(\W),$ where 
$h_n: A\to {\rm Her}(b_n)$ is a \hm. 
Moreover,
for each $n,$ $h_n^\ddag=\lambda_m$ for some $m,$ and, for each $k,$  
there are infinitely many $h_{n}$'s  such that $h_{n}^\ddag=\lambda_k.$
\end{lem}

\begin{proof}
Let $\{ b_n \}$ be a system of quasidiagonal units for $\W$.  
Write $\N=\cup_{n=1}^\infty S_n,$ where each $S_n$ is a  
countably infinite set, and 
$S_i\cap S_j=\emptyset,$ if $i\not=j.$
 For each $j\in S_n,$ choose a \hm\, $h_j: C\to {\rm Her}(b_j)$ such 
that $h_j^\ddag=\lambda_n$ (see Theorems \ref{T-existoI} and \ref{TExpon}). 
Then 
it is easy to check that {{$H :=\bigoplus_{k\in \N} h_k$
satisfies}} the requirements of the lemma.
\end{proof}

\begin{lem}\label{Lpermzero21}
Let $C$ be a separable amenable \CA\, which is $\W$ embeddable and satisfies the UCT.
Let $\sigma :  C\rightarrow M(\W)$ be a
$\T_d$ extension, and let
$\psi : C \rightarrow M(\W)$ be a {{diagonal c.p.c.~map}}
with 
the form $$\psi = \bigoplus_{n=1}^{\infty} \psi_n$$
as in Proposition \ref{prop:QDExt}, for which $\pi \circ \psi$ is nonunital {{essential extension.}}
Then, 
$$ \pi\circ \sigma \oplus \pi \circ \psi \sim^u  \pi\circ \sigma.$$ 

As a consequence, $$KK(\pi \circ \psi) = 0.$$
\end{lem}


\begin{proof}
By Lemma \ref{Uniq-2}, we may assume that $\psi$ is a (nonunital) diagonal 
extension.
So suppose that 
$\psi=\Sigma_{n=1}^\infty \psi_n: C\to M(\W)$ is a diagonal \hm, 
where each $\psi_n: C\to {\rm Her}(a_n)$ is a \hm, 
and where $\{a_n\}$ is a system of 
quasidiagonal units for $\W$. 

Denote $\lambda_{2n}=\psi^\ddag_n$ and $\lambda_{2n-1}=-\lambda_{2n},$ 
for $n=1,2,....$ 

Let $H: C\to M(B)$ be as in  Lemma \ref{Lexp-permuzero},
associated with the present $\{\lambda_n\}.$
It is easy to see that there is a permutation $\gamma:\N\to \N$
such that
$$
h_{\gamma(2n-1)}^\ddag=-h^\ddag_{\gamma(2n)},\,\,\, \makebox{ for all  }
 n=1,2,.....
$$

Define $b_k=a_{\gamma(2k-1)}+a_{\gamma(2k)},$ for all $k=1,2,....$ 
Then $\{b_k\}$ is also a system of quasidiagonal units. 
Let $h_{n,0}: C\to {\rm Her}(a_{\gamma(2n-1)}+a_{\gamma(2n)})$
be defined by $h_{n,0}(c)=h_{\gamma(2n-1)}(c)+h_{\gamma(2n)}(c)$
for all $c\in C$ and for all $n$. 
Now define $H_0: C\to M(B)$ by $H_0(c)=\bigoplus_{n=1}^\infty h_{n,0}.$
Then $H_0$ is unitarily equivalent to $H$ (see Lemma \ref{Lpermutation}).
However, $h_{n,0}^\ddag=0,$ for all $n$.   
It follows from Corollary \ref{Lexpo-zeroabsb}
that
$$
\pi\circ H_0\oplus \pi\circ \sigma\sim^u \pi\circ \sigma.
$$
Therefore 
\beq\label{Lexp-permzero2-1}
\pi\circ H \oplus \pi\circ \sigma\sim ^u \pi\circ \sigma.
\eneq


Then $\psi \oplus H$ is another diagonal extension and 
by the same argument as that for $H$,
\beq
\pi\circ \psi\oplus \pi\circ H \oplus
\pi\circ \sigma
\sim^u \pi\circ \sigma.
\eneq
Hence, by \eqref{Lexp-permzero2-1},
\beq\label{Nuniq-2-fe}
\pi\circ \psi\oplus \pi\circ \sigma
\sim^u\pi\circ \sigma,
\eneq

Since $KK(\pi \circ \sigma) = 0$, $KK(\pi \circ \psi) = 0$.

\end{proof}

\begin{lem}\label{L-quasiext}
Let $A$ be a separable amenable \CA\, which is $\W$ embeddable {{and satisfies the UCT.}}
Suppose that $\tau: A\to \C(\W)$ is an essential extension with $KK(\tau)=0.$ 
Then $\tau$ is a quasidiagonal extension.

\end{lem}

\begin{proof}
Let $\sigma : A \rightarrow \C(\W)$ be a $\T_d$ extension.
Note that
$KK(\pi\circ \sigma)=0.$  Consider  the unitizations
 ${\widetilde{\pi\circ \sigma}}, \makebox{ } 
{\widetilde{\tau}}: {\widetilde{A}}\to \C(\W).$ Then $KK({\widetilde{\pi\circ \sigma}})=KK({\widetilde{\tau}}).$
By Theorem 2.5 of \cite{LinFullext}, there exists a sequence 
$\{ u_n \}$
of unitaries 
in $\C(\W)$ such 
that $\lim_{n\to\infty} u_n^*\pi\circ \sigma(a) u_n=\tau(a)$ for all $a\in A.$
By Theorem \ref{Tappt=qd}, since $\W$ has continuous scale and $\tau$ is
nonunital,  
$\tau$ is quasidiagonal.
\end{proof}

\begin{lem}\label{L-Uniq-3}
Let $A$ be a separable amenable \CA\, which is $\W$ embeddable {{and satisfies the UCT.}}
Suppose that $\tau: A\to \C(\W)$ is a quasidiagonal essential extension. 
Then, for any $\T_d$ extension $\sigma : A \rightarrow M(\W)$,  
\beq
\tau\sim^u \pi\circ \sigma.  
\eneq
\end{lem}

\begin{proof}
By Theorem \ref{Tgroup}, 
\beq
\tau\sim^u \pi\circ \sigma\oplus \tau_0
\eneq
for some essential extension 
$\tau_0: A\to \C(\W).$  Since $KK(\tau)=0$ and $KK(\pi\circ \sigma)=0,$ 
$KK(\tau_0)=0.$
By Lemma \ref{L-quasiext}, $\tau_0$ is quasidiagonal. 
Then, by {{Proposition \ref{prop:QDExt}}} and  Lemma \ref{Uniq-2},  
\beq
\tau\sim^u \pi\circ \sigma\oplus \tau_0\sim^u \pi\circ\sigma.
\eneq

\end{proof}

 We have the following
K theory characterization of quasidiagonality (see the paragraph before
Theorem \ref{thm:SufficientConditionQD} for some 
brief history and references):

\begin{prop}\label{QuasidiagonalKTh}
Let $A$ be a separable amenable \CA\, which is $\W$ embeddable  and satisfies the UCT, and let
$\tau : A \rightarrow \C(\W)$ be an essential extension.
Then the following statements are equivalent:
\begin{enumerate}
\item $KK(\tau) = 0$. 
\item $\tau$ is quasidiagonal.
\item $\tau$ is unitarily equivalent to an essential trivial diagonal extension.
\item $\tau$ is unitarily equivalent to every essential 
trivial diagonal extension.
\item $\tau$ is in the class of zero of $\Ext^u(A, \W)$.
\end{enumerate}
\end{prop}

\begin{proof}
{{Let us recall that $A$ is nonunital as it is a \SCA\, of a stably projectionless \CA\, $\W.$
It follows from Theorem \ref{thm:ExtIsAGroup} that $\Ext^u(A, \W)$ is a group.}}

That (1) $\Leftrightarrow$ (2) follows from Lemmas \ref{L-quasiext}
and \ref{Lpermzero21}. 

That (2) $\Rightarrow$ (4) and (3) $\Rightarrow$ (4) 
follows immediately from Lemma \ref{L-quasiext} which
says that every essential quasidiagonal extension (including every essential
trivial diagonal extension) is unitarily equivalent to every $\T_d$ extension.

(4) $\Rightarrow$ (2) and (4) $\Rightarrow$ (3)  are immediate.

That (4) $\Rightarrow$ (5) follows from the {{facts that  $\Ext^u(A,\W)$ is a group and}} 
if $\rho$ is a trivial
diagonal extension then so is $\rho \oplus \rho$.

(5) $\Rightarrow$ (2): From the (4) $\Rightarrow$ (5), we know that the
neutral element of $\Ext^u(A, \W)$ is the class of an essential trivial
diagonal {{extension.}}  But then, any essential extension which is unitarily
equivalent to a trivial diagonal extension is a trivial diagonal extension. 

\end{proof}

\section{Classification of Extensions by $\W$}

\begin{lem}\label{Lell-2}
Let $B$ be a {{non-unital}} separable simple \CA\, with a unique tracial state $t_B$ such 
that $M_n(B)$ has almost stable rank one for all $n$.
Suppose that ${\rm Cu}(B) = V(B) \sqcup {\rm LAff}_+(T(B))\cong V(B) \cup (0, \infty)$.

Let $A$ be a separable exact \CA\, {{with a faithful tracial state  which satisfies the UCT.}}

Then, for any $t \in T_f(A)$ and $r \in (0,1]$, there is an embedding
$\phi_A : A \rightarrow M(B)$ such that
$t_B \circ \phi_A = r t$ {{and $\pi\circ \phi_A$ is injective.}}

\end{lem} 

\begin{proof}
{{Fix $t\in T_f(A).$}}
By   {{Theorem A of  \cite{Scha},}} let $D$ be a unital simple AF-algebra with unique
tracial state $\tau_D$ and let $\psi : A \rightarrow D$ be a *-embedding
such that 
$t = \tau_D \circ \psi.$

{{By Lemma \ref{LK0(M(B))},}} 
$$K_0(M(B)) \cong \mathbb{R},
\andeqn
{{V(M(B)) \cong V(B) \sqcup (0, \infty).}}$$

Let 
{{$\lambda : K_0(D) \rightarrow 
K_0(M(B))$
be the \hm\,}}  defined by
$$\lambda([p]) = r t_D([p])   \rforal  p \in Proj(D \otimes \K).$$
Note this gives an ordered semigroup \hm\, $\lambda_V:V(D)\to (0,\infty)\cup \{0\}\subset V(M(B)).$
{{Note also  here $(0, \infty)\cap V(B)=\emptyset.$
By Lemma 4.2 of \cite{PR}, there is a \hm\, $\phi_0: D\to M(B)\otimes {\mathcal K}$
such that $V(\phi_0)=\lambda_V.$   \Wlog, one may assume that $\phi_0(D)\subset M_m(M(B))$
for integer $m\ge 1.$}}
{{One also has $[\phi_0(1_D)]=r\le 1.$   It follows from Lemma \ref{LK0(M(B))} that there is a unitary
$U\in M_m(M(B))$ such that 
$$
U^*\phi_0(1_D)U\le 1_{M(B)}.
$$
Define $\phi(d)=U^*\phi_0(d)U$ for all $d\in D.$}}
Since $D$ is simple, $\phi$ is an embedding.
Then set  $$\phi_A := \phi \circ \psi.$$    {One checks that the embedding $\phi_A$ meets the requirement.}

\end{proof}

\begin{rmk}\label{Rhomtf}
{\rm{Recall ${\rm Hom}(K_0(A), \R)_{T_f}$ is defined in \ref{Drho}.
Several comments about it are in order. 
Firstly, under current assumptions,
$K_0(A)_+$ might be zero; hence,
one may not  have order preserving \hm s in
${\rm Hom}(K_0(A), \R).$ Secondly, there could still be a pairing 
$\rho_A: K_0(A)\to \Aff(T(A))$ even in
the case that $K_0(A)_+=\{0\}.$ Therefore  an element in ${\rm Hom}(K_0(A), \R)$
need not be induced by a \hm\, from $A$ (to $M(B)$).
Thirdly, there is a possibility that, given two tracial states $t_1, t_2\in T_f(A),$
one might have
$r_A(t_1)=r_A(t_2).$
In other words,  ${{r_A(t_1)}}$ and ${{r_A}}(t_2)$ may give the same element in ${\rm Hom}(K_0(A), \R)_{T_f}$
and, of course, they will not be distinguished.}}
\end{rmk}

We are ready to present the following classification of essential extensions of
 $\W:$

\begin{thm}\label{MTnoncom}
Let $A$ be a separable amenable \CA\, which is $\W$ embeddable and satisfies the UCT.

(1) If $\tau_1, \tau_2: A\to \C(\W)$ are two essential extensions,
then $\tau_1\sim^u \tau_2$ if and only if $KK(\tau_1)=KK(\tau_2).$

(2)  The map
\beq
\Lambda: \Ext^u(A,\W)\to KK(A, \C(\W))\cong {\rm Hom}(K_0(A), \R)
\eneq
defined by $\Lambda([\tau])=KK(\tau)$  is a group isomorphism.

(3) An essential extension $\tau$ is trivial and diagonal if and only if $KK(\tau)=0,$
and all trivial and diagonal extensions are unitarily equivalent.
In fact, the essential trivial diagonal extensions induce the neutral
element of $\Ext^u(A, \W)$.

(4) An essential extension $\tau$ is trivial if and only if
there exist $t \in T_f(A)$ and $r \in (0,1]$ such that 
$$
\tau_{*0}(x)=r\cdot r_A(t)(x)\rforal x\in K_0(A).
$$

(5) Moreover,  let $\T$ be the set of unitary equivalence classes
 of essential trivial extensions of $A$ by $\W.$
Then
\beq\nonumber
\Lambda(\T)={{\{r\cdot h: r\in (0,1]\andeqn h\in {\rm Hom}(K_0(A), \R)_{T_f(A)}\}}}\,\,\,{{({\rm see}\,\,\, {\rm{Definition}}\,\ref{Drho})}}
\eneq

(6) All quasidiagonal essential 
extensions are  trivial and are unitarily equivalent.

(7) In the case  ${\rm ker}\rho_{f,A}=K_0(A),$   
all trivial extensions are unitarily equivalent.
Moreover, an essential extension 
$\tau$ is trivial if and only if $KK(\tau)=\{0\}.$

(8)  In the case ${\rm ker}\rho_{f,A}\not=K_0(A),$ 
there are essential trivial extensions of $\W$ by $A$ which
are not quasidiagonal, and not all essential trivial extensions 
are unitarily equivalent (see (5) above).

\end{thm}

\begin{proof}
Statements (3) and (6) follow from Proposition \ref{QuasidiagonalKTh}.

(2):
That $KK(A, \C(\W))={\rm Hom}(K_0(A),\R)$
follows from the UCT, since $K_*(\C(\W))$ is divisible and
$K_1(\W) = 0$.
{{ Note that ${\rm Hom}(K_0(A), \R)$ is an abelian group.}} Recall, by Theorem \ref{thm:ExtIsAGroup}, that $\Ext^u(A, \W)$ is {{also}}
an 
abelian group.  
It is obvious that $\Lambda$ is a semigroup \hm. 
By \ref{QuasidiagonalKTh},
$\Lambda$ sends zero to zero, therefore $\Lambda$ is a group \hm. 

We next show that $\Lambda$ is  surjective.
Let $x\in KK(A, \C(\W))$ be given.  
Note that $K_0({\widetilde{A}})=K_0(A)\oplus \ZI.$ Define $\eta\in {\rm Hom}(K_0({\widetilde{A}}), \C(\W))$
by $\eta|_{K_0(A)}=x$ and $\eta([1_{\tilde A}])=[1_{\C(\W)}].$ Then $\eta$ gives an element
in $KL({\widetilde{A}}, \C(\W)).$ It follows from Corollary 8.5 of \cite{LinFullext} that there is a \hm\,
$\tau_1: {\widetilde{A}}\to \C(\W)$ such that $KK(\tau_1)=\eta.$ Define $\tau=\tau_1{|_{A}}.$
Then {{one computes}} $KK(\tau)=x.$ Since $x$ is arbitrary {{in $KK(A, \C(\W))$,}}  the map $\Lambda$ is surjective.

It remains to prove that $\Lambda$ is injective.  But, by Proposition
\ref{QuasidiagonalKTh}, if $[\psi] \in \Ext^u(A, \W)$
is such that $\Lambda([\psi]) = 0$, i.e., 
$KK(\psi) = 0$, then $[\psi] = 0$ in $\Ext^u(A, \W)$.
Hence, $\Lambda$ is injective.  This completes the proof of (2).

(1) follows from (2).

(4): Say that $\tau: A\to C(\W)$ is an essential trivial
extension.  Then there is a monomorphism
$H: A\to M(\W)$ such that $\pi\circ H=\tau.$ Let $t_1(a)=t_W\circ H(a)$ for all $a\in A.$
Then $t_1$ is a faithful trace on $A$ {{with $\|t_1\|\le 1.$}}
Let $r=\|t_1\|.$ Then $t(a)=t_1(a)/r$ is a faithful
tracial state of $A.$
{{Hence}} $H_{*0}(x)=r\rho_A(t)(x)$ for all $x\in K_0(A).$
It follows that $\tau_{*0}(x)=r\cdot \rho_A(t)(x)$ for all $x\in  K_0(A).$

Conversely, suppose now that $t\in T_f(A)$, $r \in (0,1]$
and $\tau: A\to \C(\W)$
are such that $\tau_{*0}(x)=r\cdot \rho_A(t)(x)$ for all $x\in K_0(A).$
By {{Lemma \ref{Lell-2},}} there is a monomorphism
$\psi_A: A\to M(\W)$ such that $t_W\circ \phi_A(a)=r\cdot t(a)$ 
for all $a\in A.$
Then $\pi\circ \psi_A: A\to \C(\W)$ is an essential trivial extension such that
\beq
{{(\pi\circ \psi_A)_{*0}}}=\tau_{*0}.
\eneq
Hence, $KK(\pi\circ \psi_A)=KK(\tau).$  
So, by (1), $\tau\sim^u \pi\circ \psi_A.$
It follows that $\tau$ is trivial. This completes the proof of (4). 

(5) follows from (4).

(7):  If  
 $K_0(A)={\rm ker}\rho_{f,A}$ and  
$H: A\to M(\W)$ is a monomorphism, 
then 
$H_{*0}=0.$  {{As mentioned before, $K_0(\C(\W))=\R$ is divisible and $K_1(\C(\W))=\{0\},$ this implies that}}  $KK(\pi \circ H) = 0$.
Thus, (7) follows from (3).

(8):  Suppose that  ${\rm ker}\rho_{f,A}\not=K_0(A).$ 
Then there is $t\in T_f(A)$ such that $\rho_A(t)\not=0.$
Then $\Lambda(\T)\not=\{0\}.$ So by (5), there is a trivial essential extension $\tau$ such that
$KK(\tau)\not=0.$ By (1),  $\tau$ is not unitarily equivalent to a diagonal trivial extension, and,
by (6), $\tau$ is not even quasidiagonal.

\end{proof}

For the second question of the introduction, we offer the 
following statement: 

\begin{cor}\label{CWW}
There is, up to unitary equivalence,
 only one essential extension of the form:
\beq
0\to \W\to E\to \W\to 0.  
\eneq
Moreover, this extension splits.
\end{cor}

\begin{proof}
By Theorem \ref{MTnoncom},
\beq
\Ext^u(\W, \W)={\rm Hom}(K_0(\W),\R)=\{0\}.
\eneq

\end{proof}

{{As one expected,}} {{the}} \CA s that we 
were originally interested {{in}} do
satisfy the hypothesis of Theorem \ref{MTnoncom}  {{and even (7) of Theorem \ref{MTnoncom}.}}

\begin{prop}\label{LCWemb}
Let $X$ be a connected and locally connected compact 
metric space and let $x_0\in X$ be a point.
Then $C:=C_0(X\setminus \{x_0\})$ is $\W$ embeddable.
Moreover, {{$K_0(C)={\rm ker}\rho_{C}.$}}
\end{prop}

\begin{proof}
{{We first show that there is an embedding $\iota_C: C_0(X\setminus \{x_0\})\to C_0((0,1]).$
By the Hahn--Mazurkiewicz theorem, there exists a continuous surjection $s_0:
[1/2,1] \twoheadrightarrow X.$ Let $y_0=s_0(1/2)\in X.$ By the assumption on $X,$ 
there is a continuous path $s_1: [0,1/2]\to X$ such that $s_1(0)=x_0$ and $s_1(1/2)=y_0.$
Define $s: [0,1]\to X$ by $s|_{[0,1/2]}=s_0$ and $s|_{[1/2,1]}=s_1.$
Then $s: [0,1]\to X$ is a continuous surjection which induces an embedding
$\iota_C: C\to C_0((0,1]).$ Since $\W$ is projectionless, one easily embeds $C_0((0,1])$ into $\W.$
This shows that $C$ is $\W$-embeddable. }}

{{Let $C_1=C(X)={\widetilde{C}}.$ 
Since $X$ is connected,  $\rho_{C_1}$ is the rank function and $\rho_{C_1}(K_0(C_1))=\ZI.$
The short exact sequence
$$
0\to K_0(C)\to K_0(C_1)\to \ZI \to 0
$$
 also shows that $K_0(C)={\rm ker}\rho_{C_1}={\rm ker}\rho_C.$
}}
\end{proof}

The following is  a corollary of Theorem \ref{MTnoncom} (and Corollary \ref{LCWemb}).

\begin{thm}
Let $X$ be a connected and locally connected compact
 metric space, let $x_0\in X$
and let $C:=C_0(X\setminus \{x_0\}).$

(1) All trivial essential  extensions $\tau_0$ are unitarily equivalent (in $\C(\W)$) and hence, are unitarily equivalent
to a diagonal extension $\pi\circ \sigma$ in $\T_d.$

(2) $[\pi \circ \sigma]$ is the class of zero in $\Ext^u(C, \W)$.

(3) There is a
a group isomorphism
$$\Ext^u(C, \W) \cong KK(C, \C(\W))={\rm Hom}(K_0(C),\R).$$
\label{thm:MainTh1}
\end{thm}

There are many non-commutative \CA s which are $\W$ embeddable including
$\W$ itself.  In fact, we have the following.

\begin{prop}\label{AF-emb-1}
{{Let $A$ be a stably projectionless algebraically simple separable  \CA\,
with finite nuclear dimension 
which satisfies
the UCT.   Suppose that ${\rm ker}\rho_A=K_0(A).$}} Then $A$ is $\W$ embeddable.

\end{prop}

\begin{proof}
{{Let $B=A\otimes {\mathcal K}$ and $C=B\otimes Q,$ where $Q$ is the UHF-algebra with $(K_0(Q), K_0(Q)_+,[1_Q])=
({\mathbb Q,} {\mathbb Q}_+, 1).$ It follows  that there exists an element $c\in C_+^1$ such that
${\overline{cCc}}$ has continuous scale (see Remark 5.2 of \cite{eglnp}).  Let $e_A\in A_+^1$ be a strictly positive element 
of $A.$ Since $A$ is algebraically simple, ${\rm Per}(A)=A.$ It follows that $e_A\otimes 1_Q$ is in 
${\rm Ped}(C).$ Note also $c\in {\rm Ped}(A).$ It follows that there is an integer $n\ge 1$ such that $\la e_A\ra \le n\la c\ra.$
\Wlog, we may assume that $e_A\in 
M_n(\overline{cAc})$ (recall $C=A\otimes {\mathcal K}\otimes Q$). 
Put $D=M_n(\overline{cAc})\otimes Q.$ Then $A$ is embedded in to $D.$ 
It suffices to show that $D$ is $\W$ embeddable.}}

Note that $D$ is a stably projectionless simple \CA\, with continuous scale and satisfies the UCT.
Moreover, ${\rm ker}\rho_D=K_0(D).$  Furthermore, $D$ has finite nuclear dimension.
Since $D$ is stably projectionless, $T(D)\not=\emptyset.$
By Theorem 15.5 of \cite{GLII} (see the last line of the proof too), $D\in {\mathcal D}_0.$  We then apply Theorem 12.8 
of \cite{GLII} as $\W$ has the form $B_T$ (with $K_i(\W)=0$ ($i=0,1$) and unique tracial state). We choose 
$\kappa=0$ and $\kappa_T: T(D)\to T(\W),$ by mapping all points to one point and with $\kappa_{cu}$
compatible with $(\kappa, \kappa_T).$  Thus Theorem 12.8 {{of \cite{GLII}}} provides an embedding from $D$ to $\W.$

\end{proof}

\begin{rmk}\label{Rmwemb}
{{ Note that $A\otimes Q$ is ${\mathcal Z}$-stable and nuclear. Therefore, by a recent result \cite{CETWW},
it has finite nuclear dimension. So the assumption that $A$ has finite nuclear dimension can be removed.
Consider any separable simple stably projectionless \CA\, $A$ of stable rank one with $QT(A)=T(A)=\emptyset.$ 
Assume that ${\rm Cu}(A)={\rm LAff}_+(T(A))$  and has  continuous scale. Then 
${\rm Cu}^\sim(A)=K_0(A)\cup {\rm LAff}_{0+}^\sim (T_0(A))$  (see  Theorem 6.2.3 of \cite{RI}, also Theorem 7.3 of \cite{eglnp}
for the statement as well as the notation).  By Corollary A.7 of \cite{ElliottGongLinNiu},
there exists $\tau\in T(A)$
such that $\rho_A(x)(\tau)=0$  for all $\tau\in T(A).$  Note that ${\rm Cu}^\sim(\W)={\rm LAff}_{0+}(T_0(\W))=\R\cup\{+\infty\}.$
Define $\lambda: Cu^\sim(A)\to  {\rm LAff}_{0+}^\sim (T_0(\W)))$   by $\lambda|_{K_0(A)}=0$
and $\lambda(f)(t_W)=f(\tau)$ for all $f\in {\rm LAff}_{0+}(T_0(A)).$ This gives a morphism 
from ${\rm Cu}^\sim(A)$ to ${\rm Cu}^\sim(\W).$
Let us also assume that $A$ is an inductive limit of 1-dimensional non-commutative CW complexes
with trivial $K_1(A).$ 
Then,  by Theorem 1.0.1 of \cite{RI}, $A$ is $\W$ embeddable. 
In fact, even in this case, $K_1(A)$ is trivial can be removed 
(see section 6 of \cite{GLIII}).}}

\end{rmk}

Recall that there is a separable simple stably projectionless \CA\, $\zo$ with a unique tracial state
$\tau_z,$ with finite nuclear dimension and satisfies the UCT such
that $K_0(\zo)=\ZI={\rm ker}\rho_{\zo}$ and $K_1(\zo) = 0$.
  By \cite{GLII}, there is only one such simple \CA\, up to isomorphism.
  Note that $T(A)=T(A\otimes\zo.)$
Moreover, as abelian groups,
\beq
K_i(A\otimes \zo)\cong K_i(A),\,\,\, i=0,1.
\eneq

\begin{prop}\label{AF-emb}
Let $A$ be a separable exact \CA\, which satisfies the UCT and has a
faithful amenable tracial state. Then $C:=A\otimes \zo$ is $\W$ embeddable
and ${\rm ker}\rho_{f,C}=K_0(C).$

\end{prop}

\begin{proof}
It follows from Theorem A of \cite{Scha} that there is a unital simple AF-algebra $B$
with a unique tracial state $\tau_B$
and a monomorphism $\phi: A\to B$ (which preserves the trace).
There is also an embedding $\psi_{z,w}: \zo\to \W.$
Thus we obtain {{an}} embedding $\phi_A:=\phi\otimes \psi_{z,w}: A\otimes \zo\to
B\otimes \W.$ Note that $K_0(B\otimes \W)=\{0\}.$
Since $B\otimes \W$ has only one tracial state (namely $\tau_B\otimes \tau_W$),
 by Theorem 7.5 of \cite{ElliottGongLinNiu}, $B\otimes \W\cong \W.$ 
Thus, $A\otimes \zo$ is $\W$ embeddable.

To see the last part of the statement, 
note that every  $y\in K_0(A\otimes \zo)$ may be written
as $x\otimes x_0,$ where $x_0\in K_0(\zo)=\ZI$ is a generator. Note 
that every faithful tracial state of $A\otimes \zo$ has
the form $\tau\otimes \tau_z,$ where $\tau\in T_f(A).$ But $\tau(x\otimes x_0)=\tau(x)\tau_z(x_0)=0.$
\end{proof}

\begin{thm}\label{MTnoncom1}
Let $B$ be a separable amenable \CA\, which has a faithful tracial state and  satisfies the UCT,
and let $A=B\otimes \zo.$

(1) If $\tau_1, \tau_2: A\to \C(\W)$ are two essential extensions,
then $\tau_1\sim^u \tau_2$ if and only if $KK(\tau_1)=KK(\tau_2).$

(2)  The map
\beq
\Lambda: \Ext^u(A,\W)\to KK(A, \C(\W))\cong {\rm Hom}(K_0(A), \R)
\eneq
defined by $\Lambda([\tau])=KK(\tau)$  is a group isomorphism.

(3) An essential extension $\tau$ is trivial if and only if $KK(\tau)=0,$
and all essential trivial extensions are unitarily equivalent.

\end{thm}

For the rest of this section, we consider essential extensions of the form
$$
0\to \W\to E\to C(X)\to 0,
$$
where $X$ is a connected and locally connected compact metric space.

\begin{lem}\label{lem:ProjectionLifting}
Let $p \in \C(\W)$ be a nonzero projection such that
$[p]_{K_0(\C(\W))} \in (0,1).$
Then $p$ can be lifted to a projection in $M(\W)$.

Moreover, if $p\not=1_{\C(\W)}$ and $[p]\not\in (0,1),$
then $p$ cannot be lifted to a nonzero projection in $M(\W).$

\end{lem}

\begin{proof}
Say that $[p]_{K_0(\C(\W))} = r \in (0,1)$.
By  {{Corollary 4.6 of}} \cite{LinNg}
(see also  Section 5 of \cite{KNZPICor}),
let $Q \in M(\W)\setminus \W$ be a projection such that
$\tau(Q) = r.$
Therefore,
by our computation of $K_0(\C(\W))$, and since $\C(\W)$ is simple purely infinite,
$$\pi(Q) \sim p
\,\,\,{\rm in}\,\,\,\C(\W).
$$
Since $\C(\W)$ is simple purely infinite and since
$\pi(Q) \neq 1 \neq p,$
{{there is  a unitary}} $u \in \C(\W)$ such that
$$u \pi(Q) u^* = p.$$
Since $\C(\W)$ is simple purely infinite and since $K_1(\C(\W)) = 0$, 
$u$ lifts to a unitary  {{$U \in M(\W)$.}}
It follows that {{$\pi(UQU^*)=p.$}}

The last part follows from the fact that if $P\in M(\W)$
is a non-zero projection,  then $\tau_W(P)\in (0,1].$
\end{proof}

\begin{lem}
Let $X$ be a connected and locally connected compact
 metric space and let $x_0\in X.$
Suppose that  $\phi : C(X) \rightarrow \C(\W)$ is
an essential extension.
Then there exists a proper subprojection
$p \leq \phi(1)$
such that
$$p \phi(f) = \phi(f)p=\phi(f)
\rforal f \in C_0(X \setminus \{ x_0 \}).$$

Moreover, for all $s \in (0,1)$ we
may choose $p$
such that there is $P\in M(\W)$ for which 
$\pi(P)=p$ and $\tau_W(P)=s.$
\label{lem:May202019}
\end{lem}

\begin{proof}
Let $e_C$ be a strictly positive element of $C_0(X\setminus \{x_0\})$ for
which $\| e_C \| = 1$. 
Let $B={\rm Her}(\phi(e_C))\subset \C(\W).$ Note that ${\rm sp}(e_C)=[0,1].$
Write  $e=\phi(1).$ If $ey=0$ for all $y\in B^\perp,$ then $(1-e)y=y=y(1-e)$
for all $y \in B^{\perp}$.  
This implies that $B^\perp=(1-e)\C(\W)(1-e).$  Then, by Theorem 15 of \cite{Ped},
$B=(B^\perp)^\perp=e\C(\W)e.$   So $B$ is unital. This 
contradicts that ${\rm sp}(e_C)=[0,1].$
Therefore there is $y\in (B^\perp)_+$ such that $eye\not=0.$
Since $e\C(\W)e$ has real rank zero, there is a projection $p_1\in e\C(\W)e$
such that $p_1 \neq e$ and $p_1b=bp_1=b$ for all $b\in B$. 

Since $\C(\W)$ is simple purely infinite, we can find a projection
$q_1 \in (e - p_1) \C(\W)(e - p_1)$ with $q_1 \neq e - p_1$ such that
$[p_1 + q_1] = s \in (0,1)$.
If we define $p := p_1 + q_1$ then $p b = b$ for all $b \in B$.

Also, by {{Lemma \ref{lem:ProjectionLifting},}}
$p$ lifts to a projection in 
$P \in M(W)$. Necessarily, $\tau_W(P) = s$. 
\end{proof}

\begin{thm}\label{thm:Injectivity2}
Let $X$ be a connected and locally connected compact metric space and
$\phi, \psi : C(X) \rightarrow \C(\W)$
{{be}}  essential extensions.

(1)  {{Then}} $KK(\phi) = KK(\psi)$ if and only if 
$$\phi \sim \psi.$$

(2) If both $\phi$ and $\psi$ are unital or both are non-unital, then
$KK(\phi)=KK(\psi)$ if and only if
$$\phi \sim^u \psi.$$

(3) The map
$$\Lambda: {{\Ext(C(X),\W)}}\to KK(C(X), \C(\W))={\rm Hom}(K_0(C(X)), \R)$$
defined by $\Lambda([\tau])=KK(\tau)$  is a group isomorphism.

(4) The zero element of {{$\Ext(C(X), \W)$}} (or $KK(C(X), \C(\W))$) is not the 
class of
a trivial {{(splitting)}} extension. 

(5)   Let $\tau: C(X)\to \C(\W)$ be an essential extension.
Then $\tau$ is trivial
if and only if
$$
KK(\tau|_{C_0(X\setminus \{x_0\})})=0\andeqn \makebox{  either  }
[\tau(1_{C(X)})]\in (0,1) \makebox{  or  } \tau(1_{C(X)}) = 1_{\C(\W)}.
$$

\end{thm}

{{
\begin{rmk}
 Suppose that $p\in \C(\W)$ and $[p]=[1_{\C(\W)}]$ but $p\not=1_{\C(\W)}.$
Let $\phi: C(X)\to \C(\W)$  be  an essential extension with $KK(\phi|_{C_0(X\setminus \{x_0\})})=0$
and $\phi(1)=p.$
Then  $\phi\sim \phi_0$ for a splitting essential extension with $\phi_0(1)=1_{\C(\W)}.$ 
But $\phi$ is not a trivial extension itself as $p$ cannot be lifted to a projection in $M(\W)$ (see \ref{lem:ProjectionLifting}).
\end{rmk}
}}

\begin{proof}
(1):   Suppose that $\phi \sim \psi$.
Then there is a  
$w\in \C(\W)$ such that $w^*\phi(c)w=\psi(c)$ for all $c\in  C(X)$
with $w^*w= \psi(1)$ and $ww^*=\phi(1).$ 
Since {{$M_2(\C(\W))$}} is simple {{and}} purely infinite,
there exists a unitary $W \in {{M_2(\C(\W))}}$
such that $W \psi(1) = w$.  
Then 
$W^*\phi(c)W=\psi(c)$ for all $c\in C(X).$ 
It follows that $KK(\phi)=KK(\psi).$

Conversely, suppose that $KK(\phi)=KK(\psi).$
Fix $x_0\in X$ and $s\in (0,1).$
Let $p := \phi(1)$ and $q := \psi(1)$.
By Lemma \ref{lem:May202019},  choose a proper subprojection  $p_0\le p$ such that $p_0\phi(f)=\phi(f)p_0=\phi(f)$ for all $f\in C_0(X\setminus \{x_0\}),$
and a projection $P_0\in M(\W)$ such that $\pi(P_0)=p_0$ and $\tau_W(P_0)=s$.
The same argument shows that there is a proper subprojection 
$q_0\le q$ such that $q_0\psi(f)=\psi(f)q_0=\psi(f)$
for all $f\in C_0(X\setminus \{x_0\}),$ and  there exists a projection $Q_0\in M(\W)$ such that $\pi(Q_0)=q_0$ and
$\tau_W(Q_0)=s.$

Since $\C(\W)$ is purely infinite simple, 
$p_0 \sim q_0$ and $p - p_0 \sim q - q_0$. \Wlog, we may assume that $p_0=q_0$
and $\phi(1) = p = q = \psi(1)$.  

Since $D:=P_0\W P_0\cong \W,$ $M(D)=P_0M(\W)P_0$ and $\pi(\C(D))=p_0\C(\W)p_0,$
one may view $\phi|_{C(X\setminus \{x_0\})}, \psi|_{C(X\setminus \{x_0\})}$ 
as maps from $C(X\setminus \{x_0\})$ to $\C(D).$
By applying Proposition \ref{LCWemb} and Theorem \ref{MTnoncom} , one 
obtains a unitary $u_0\in \C(D)=p_0\C(\W)p_0$ such
that
$$
u^*\phi(c)u=\psi(c)\rforal c\in C(X\setminus \{x_0\}).
$$
If we define $v := u + (p - p_0)$, then
$v$ is a partial isometry in $\C(\B)$ such that
$vv^* = \phi(1)$, $v^* v = \psi(1)$ and
$$v^*\phi(c)v =\psi(c)\rforal c\in C(X).
$$
This completes the proof of (1).

(2): If $\phi \sim^u \psi$ then it is immediate that 
$KK(\phi) = KK(\psi)$.  Hence, we only need to prove the converse
{{direction.}} 

So suppose that $KK(\phi) = KK(\psi)$. 
Suppose that both  $\phi(1_{C(X)})$ and $\psi(1_{C(X)})$
are  equal to $1_{\C(\W)}.$  
Then by (1), since $KK(\phi) = KK(\psi)$,  $\phi\sim \psi.$
In other words, there is $w\in \C(\W)$ such that
$w^*\phi_1(x)w=\phi_2(x)$ for all $x\in C(X).$
{{Moreover, by  Definition \ref{Dmul},
$$
1_{\C(\W)}=\phi_1(1_{C(X)})=ww^*\andeqn
1_{C(\W)}=\phi_2(1_{C(X)})=w^*w.
$$}} 
Hence $w$ is a unitary.
Since $K_1(\C(\W)) = 0$ and $\C(\W)$ is simple {{and}} purely infinite,
every unitary in $\C(\W)$ can be lifted to a unitary in $M(\W)$.
Hence, $\phi \sim^u \psi$.

Now suppose that both $\phi$ and $\psi$ are nonunital. 
Then by (1), since $KK(\phi) = KK(\psi)$, $\phi \sim \psi$.
So we have a partial isometry $v \in \C(\W)$ such that 
$v v^* = \phi(1)$, $v^* v = \psi(1)$ and
$v^* \phi(c) v = \psi(c)$ for all $c \in C(X)$.  
Since $\phi(1), \psi(1)$ are proper subprojections of $1_{\C(\W)}$ and since
$\C(\W)$ is simple {{and}} purely infinite, we can find a unitary $u \in \C(\W)$
such that  $\phi(1) u =  v$.  
Hence, $u^* \phi(c) u = \psi(c)$ for all $c \in C(X)$.
Since $\C(\W)$ is simple {{and}} purely infinite and $K_1(\C(\W)) = 0$,
$u$ can be lifted to a unitary in $M(\W)$. So $\phi \sim^u \psi$.

\noindent This completes the 
{{proof of}} (2).

(3)  The injectivity of the group homomorphism $\Lambda$ follows from {{(1).}}
Hence, it remains to prove that $\Lambda$ is surjective.

Let $\alpha \in KK(C(X), \C(\W))$ be given.   {{Fix a point $x_0\in X.$
Let $\iota: C_0(X\setminus \{x_0\})\to C(X)$ be the embedding and 
$q: C(X)\to \mathbb C$ be the quotient map. So we have the following
the splitting short exact sequence 
$$
0\to C_0(X\setminus \{x_0\})\stackrel{j}{\longrightarrow}  C(X)\stackrel{q}{\longrightarrow} {\mathbb C}\to 0
$$
which gives the splitting short exact sequence:}}
{{
\beq\label{KKshortexact}
0\to KK(\mathbb C, \C(\W))\stackrel{[q]}{\longrightarrow} KK(C(X), \C(\W))\stackrel{[{{j}}]}{\longrightarrow} KK(C_0(X\setminus \{x_0\}), \C(\W){{)}}\to 0.
\eneq
}}
 Consider the Kasparov
product  
$$\beta := [\iota] \times \alpha \in KK({{C_0(X \setminus \{ x_0 \}),}} \C(\W)).$$
By Proposition \ref{LCWemb} and Theorem \ref{MTnoncom},
let $\phi : C_0(X \setminus \{ x_0 \}) \rightarrow \C(\W)$ be an essential
extension such that 
$KK(\phi) = \beta.$ 

Note that $\phi$ can be extended to a monomorphism $C(X) \rightarrow \C(\W)$
which brings $1_{\C(X)}$ to $ 1_{\C(\W)}$.  So by  
Lemma \ref{lem:May202019}, 
let $p \leq 1_{\C(\W)}$ be a proper subprojection such that 
$p \phi(f) = \phi(f)$
for all $f \in C_0(X \setminus \{ x_0 \})$.
Since $\C(\W)$ is simple {{and}} purely infinite, {{one may choose}} 
$q \leq 1_{\C(\W)} - p,$  a proper subprojection, such that 
$$
[p + q ]=\alpha([1_{C(X)}]).
$$
Let $\phi_1 : C(X) \rightarrow \C(\W)$ be the nonunital
essential extension given by
$$\phi_1 |_{C(X \setminus \{ x_0 \})} = \phi \makebox{  and  }
\phi_1(1_{C(X)}) = p + q.$$
{{To see that $KK(\phi_1)=\af,$ let us note that, we have shown that $[\iota]\times \af=[\iota]\times KK(\phi_1))$
in the exact sequence \eqref{KKshortexact}. 
However, since $[q\circ\phi_1(1_{C(X)})]=\af([1_{C(X)}]),$  we have 
$\phi_1\times [q]=\af\times [q].$ 
By \eqref{KKshortexact}, 
$$
\af=KK(\phi_1).
$$}}
{{Therefore}} $\Lambda([\phi_1]) = \alpha$
as required.  This completes the proof of the surjectivity of $\Lambda$, and
hence, the proof of (3). 

(4): Say that   
$KK(\phi)=0.$ Then $[\phi(1)]=0\in \R=K_0(\C(\W))$.
By Lemma \ref{lem:ProjectionLifting}, 
there is no non-zero projection $P\in M(\W)$ such that $\pi(P)=p.$ It follows that $\phi$ is not liftable.

(5):  Say that $\tau: C(X) \rightarrow \C(\W)$ is an essential trivial
extension.  Then by Theorem \ref{thm:MainTh1},      
$KK(\tau |_{C_0(X \setminus \{ x_0 \})}) = 0$.
Also, $\tau(1_{C(X)}) \in \C(\W)$ must be liftable to a nonzero projection in 
$M(\W)$. Hence, by Lemma \ref{lem:ProjectionLifting},
either $[\tau(1_{C(X)})]_{K_0(\C(\W))} \in (0, 1)$ or 
$\tau(1_{C(X)}) = 1_{\C(\W)}$.  

Conversely, suppose that 
$\tau: C(X)\to \C(\W)$ is an essential extension such that 
 $KK(\tau |_{C_0(X \setminus \{ x_0 \})}) = 0$ and either
$[\tau(1_{C(X)})]_{K_0(\C(\W))} \in (0,1)$ or 
$\tau(1_{C(X)}) = 1_{\C(\W)}$.

By the hypotheses on $\tau(1_{C(X)})$ and by Lemma  
\ref{lem:ProjectionLifting}, $\tau(1_{C(X)})$ can be lifted to a (nonzero) 
projection
$P \in M(\W)$.

Consider the extension
$\tau |_{C_0(X \setminus \{ x_0 \})} : C_0(X \setminus \{ x_0 \})
\rightarrow \C(P \W P)$.
Since $P \W P \cong \W$ and since $KK(\tau  |_{C_0(X \setminus \{ x_0 \})})
= 0$, it follows, by Theorem \ref{thm:MainTh1}, that
there is monomorphism $H_0 : C_0(X \setminus \{ x_0 \}) \rightarrow
P M(\W) P$ such that 
$$\pi \circ H_0 = \tau |_{C_0(X \setminus \{ x_0 \})}.$$
Let $H : C(X) \rightarrow M(\W)$ be the monomorphism given by
$$H |_{C_0(X \setminus \{ x_0 \})} = H_0 \makebox{  and  }
H(1_{C(X)}) = P.$$
Then $\pi \circ H = \tau$, i.e., $\tau$ is trivial.
This completes the proof.

\end{proof}

\begin{cor} Let $\mathbb{T}^n$ be the n torus. 
\begin{enumerate}
\item $\Ext(C_0(\mathbb{T}^n \setminus \{ 1 \}), \W) = \mathbb{R}^{2^{n-1}-1}$.
\item $\Ext({C(\mathbb{T}^n)}, \W) = \mathbb{R}^{2^{n-1}}.$
\end{enumerate} 
\end{cor}

\vspace{0.2in}

\section{Classification of some non-simple \CA s}

This section is for the second goal of our original research {{plan.}}
We study non-simple \CA s $E$  which are essential extensions of the form
\beq\label{Shortexinnons}
0\to \W\to E\stackrel{\pi_E}{\longrightarrow} A\to 0,
\eneq
where $A$ is some separable  stably finite simple \CA\,  with finite nuclear dimension
such that ${\rm ker}\rho_A=K_0(A).$
In other words, $E$ has a unique ideal $I\cong \W$ and $E/\W$ is 
a separable stably finite simple \CA\,  with finite nuclear dimension such that
${\rm ker}\rho_A=K_0(A).$ 
Denote by $\E$ the class of such \CA s which satisfy the UCT (Warning:
here we do not assume that $A$ is fixed, but is any separable  stably finite simple \CA\,  with finite nuclear dimension
such that ${\rm ker}\rho_A=K_0(A)$ which satisfies the UCT).

Let $\E_c$ be the subclass  of those \CA s  $E$ in $\E$ such that $A:=\pi_E(E)$ has continuous scale,
where $\pi_E: E\to E/\W$ is the quotient map.

 In general, if $E$ is an essential extension by $\W$ 
  then $E$ is a subalgebra of $M(\W).$ 
Recall that we identify the unique tracial state $\tau_W$ on $\W$
with its unique extension to a tracial state on  
 $M(\W)$,
 which
we also denote by $\tau_W$. 
Therefore, $\tau_W$ also induces a tracial state on $E$
 which, again, we denote by $\tau_W.$
 There is a group \hm\, $\lambda_E: K_0(E)\to \R$ induced by $\tau_W,$ i.e., $\lambda_E(x)=\tau_W(x)$
 for all $x\in K_0(E).$
 Since $A=E/\W$ 
 and since $K_i(\W)=\{0\}$ ($i=0,1$), 
by the six-term exact sequence in $K$-theory,
 one computes that $\pi_{E *i}: K_i(E)\to K_i(A)$ is a group
 isomorphism ($i=0,1$).

\begin{lem}\label{LKtE}
Let $E$ be an essential extension of the form
$$
0\to \W\to E\stackrel{\pi_E}{\longrightarrow} A\to 0,
$$
where $A$ is a separable amenable \CA\, with $K_0(A)={\rm ker}\rho_{f, A}$ which is $\W$ embeddable 
and satisfies the UCT. Let $\psi: A\to \C(\W)$ be the Busby invariant for
the above extension.
 
Then $$\psi_{*,0} =\lambda_E\circ {\pi_E}_{*0}^{-1} \makebox{ in }{\rm Hom}(K_0(A), \R).$$

\end{lem}

\begin{proof}
Denote by $\pi: M(\W)\to \C(\W)$ the quotient map. 
One has the following commutative diagram:
\[
\begin{array}{cccc}
K_0(E) & \stackrel{\lambda_E}\longrightarrow & K_0(M(\W))=\R\\
\downarrow_{(\pi_E)_{*0}} & &  \downarrow_{\pi_{*0}}\\ 
K_0(A) & \stackrel{\psi_{*0}}{\longrightarrow} & K_0(\C(\W))=\R.\\
\end{array}
\]
By Theorem \ref{K-six}, $\pi_{*0}$ is a group isomorphism.  
Since ${\pi_E}_{*0}$ is also a group  isomorphism, the lemma follows.

\end{proof}

Before defining the classification invariant, we recall some definitions
and other items.

Again, let $E \in \E_c$ and
let $A:=\pi_E(E).$ Then, as per our definitions, $A$ is a separable  stably finite simple continuous scale \CA\, with finite nuclear dimension
such that ${\rm ker}\rho_A=K_0(A).$ Since $A$ is stably finite, $A$ is stably projectionless. (Here is short proof: Suppose, for {{contradiction,}} that   
$p\in M_m(A)$ is a nonzero projection for some integer $m\ge 1.$  
Since
$T(A)\not=\emptyset$, $\tau(p)\not=0.$ This contradicts that ${\rm ker}\rho_A=K_0(A).$)

If $\tau\in T(A),$ then $\tau\circ \pi_E\in T(E).$ 
The map $(\pi_E)_T: T(A)\to T(E)$, defined by $(\pi_E)_T(\tau)(b)=\tau(\pi_E(b))$ for all $b\in E$ and $\tau\in T(A)$,  
is an affine homeomorphism onto a closed convex subset of $T(E).$
Denote by $T_A$ the closed convex subset $(\pi_E)_T(T(A)).$
Then $T(E)$ is the convex hull of $T_A$ and $\tau_W.$ Since $A$ has continuous scale, $T(A)$ 
is compact. It follows that $T(E)$ is compact.   Note that $T_A$ is a face of $T(E).$

Let $S(K_0(\widetilde{E}))$ be the state space of $K_0(\widetilde{E}),$ i.e., 
the set of all group \hm s $s: K_0(\widetilde{E})\to \R$ such that $s(x)\ge 0$ for all $x\in K_0(\widetilde{E})_+$ and 
$s([1_{\widetilde{E}}])=1.$ 
Denote $S(K_0(E)):=\{s|_{K_0(E)}: s\in S(K_0(\widetilde{E}))\}.$

The map $r_E: T(E)\to S(K_0(E))$ is defined by $r_E(\tau)(x)=\tau(x)$ 
for all $x\in K_0(A)$ and $\tau\in T(E).$

Now we can define our classification invariant. 

\begin{df}\label{Dell}
{\rm Let $E\in \E_c.$ The Elliott invariant ${\rm Inv}(E)$ is defined as follows:
\beq
{\rm Inv}(E)=(K_0(E), K_1(E), T(E), r_E).
\eneq}  
\end{df}

Let $E_1, E_2\in \E_c.$ We say that ${\rm Inv}(E_1)$ and ${\rm Inv}(E_2)$ are isomorphic, 
and write ${\rm Inv}(E_1)\cong {\rm Inv}(E_2),$ if there is an 
{{isomorphism}} 
$$
\Gamma: {\rm Inv}(E_1)=(K_0(E_1), K_1(E_1), T(E_1), r_{E_1})\to 
 {\rm Inv}(E_2)=(K_0(E_2), K_1(E_2), T(E_2), r_{E_2}),
 $$
i.e., 
if there is a group isomorphism $\Gamma_i: K_i(E_1)\to K_i(E_2),$
$i=0,1,$ and an affine homeomorphism $\Gamma_T: T(E_1)\to T(E_2)$, which 
maps $T_f(E_1)$  onto $T_f(E_2)$, such that
\beq\label{Dinv-paring0}
r_{E_2}(\tau)(\Gamma_0(x))=r_{E_1}(\Gamma_T^{-1}(\tau))(x)\rforal x\in K_0(E_1)\,\,\, {\rm and}\,\, \tau\in T(E_2).
\eneq

\begin{thm}\label{MTclass}
Let $E_1, E_2\in \E_c.$ Then $E_1\cong E_2$ if and only if 
\beq
{\rm Inv}(E_1)\cong {\rm Inv}(E_2).
\eneq
Moreover, if $\Gamma: {\rm Inv}(E_1)\to {\rm Inv}(E_2)$ is an isomorphism, 
then there exists an isomorphism $\Psi: E_1\to E_2$ such that
$\Psi$ induces $\Gamma.$
\end{thm}

\begin{proof}
We have two short exact sequences
\beq\nonumber
&&0\to \W\to E_1\stackrel{\pi_{E_1}}{\longrightarrow} A\to 0\andeqn\\\nonumber
&&0\to \W\to E_2\stackrel{\pi_{E_2}}{\longrightarrow} B\to 0.
\eneq
Both $A$ and $B$ are separable simple stably projectionless  {{\CA s}}  with finite nuclear dimension and  
continuous scale 
which satisfy the UCT. Moreover, 
$K_0(A)={\rm ker}\rho_A$ and $K_0(B)={\rm ker}\rho_B.$

Suppose that $\Gamma: {\rm Inv}(E_1)\to {\rm Inv}(E_2)$ is 
an  isomorphism. 
{{Then,}} one has a group isomorphism $\gamma_i:={\pi_{E_2}}_{*i}\circ \Gamma_i\circ {\pi_{E_1}}_{*i}^{-1}: K_i(A)\to K_i(B),$ $i=0,1.$ 
Note that 
 $\Gamma$ induces an affine  homeomorphism $\Gamma_T: T(E_1)\to T(E_2)$
which maps $T_f(E_1)$ to $T_f(E_2).$ Since $\Gamma_T$ is an affine homeomorphism, 
it maps extreme points to extremal points.  It follows that $\Gamma_T(\tau_W)=\tau_W.$
Since the extreme points of the face $T_A$ are also extreme points of $T(E),$ $\Gamma_T$ maps 
$T_A$ onto $T_B$ which is also an affine homeomorphism.  {{Since both $E_1$ and $E_2$ satisfy the UCT,
$A$ and $B$ also satisfy the UCT. }}
It follows that,   by the classification results in \cite{GLII} (see Theorem 13.1 and Theorem 15.5 of \cite{GLII}),
 there is an isomorphism $\phi: A\to B$
such that $\phi_{*i}=\gamma_i,$ $i=0,1,$ and $\phi_T= (\pi_{E_1})_T^{-1}
\circ \Gamma_T^{-1} \circ
(\pi_{E_2})_T.$
Recall that $\lambda_{E_i}: K_0(E_i)\to \R$ is 
defined by $\lambda_{E_i}(x)=\tau_W(x)$ for all $x\in K_0(E_i).$ 

In \eqref{Dinv-paring0}, with $\tau=\tau_W,$  we have (as $\Gamma_T$ maps $\tau_W$ to $\tau_W$)
\beq
\tau_W(\Gamma_0(x)) = \tau_W(x) \tforal x\in K_0(E_1). 
\eneq
In other words
\beq\label{MTclass-10}
 \lambda_{E_2}\circ \Gamma_0 =\lambda_{E_1}.
\eneq
Let $\sigma_A: A\to \C(\W)$ and 
$\sigma_B: B\to \C(\W)$ be the Busby invariants associated with $E_1$ and 
$E_2$ respectively.   
Define an essential extension $\psi: A\to \C(\W)$ by $\psi(a)=\sigma_B\circ \phi.$
Hence,
\begin{eqnarray*}
\psi_{*,0} 
& = &  (\sigma_B)_{*,0} \circ \phi_{*,0} \\
& = & (\lambda_{E_2} \circ {\pi_{E_2}}_{*,0}^{-1}) \circ ({\pi_{E_2}}_{*,0}
\circ \Gamma_0 \circ {\pi_{E_1}}_{*,0}^{-1}) \hspace{0.4in}\makebox{  (by Lemma \ref{LKtE}
 and since  } \phi_{*,0} = \gamma_0 \makebox{)}\\
& = & \lambda_{E_2} \circ \Gamma_0 \circ {\pi_{E_1}}_{*,0}^{-1} 
=\lambda_{E_1} \circ {\pi_{E_1}}_{*,0}^{-1}  \hspace{0.7in} \makebox{ (by  
(\ref{MTclass-10}) )}\\ 
& = & {{(\sigma_A)}}_{*,0}  \hspace{2.2in}\makebox{  (by Lemma \ref{LKtE}).}
\end{eqnarray*} 

Hence, 
\beq
KK(\psi)=KK(\sigma_A).
\eneq
It follows from Theorem \ref{MTnoncom} that there is a unitary $U\in M(\W)$ such 
that
\beq\label{MTclass-11}
\pi(U^*)\psi(a)\pi(U)=\sigma_A(a)\rforal a\in A.
\eneq
Note that, by \eqref{MTclass-11}, 
\beq
U^*eU\in E_1\rforal e\in E_2.
\eneq 
Define $\Psi: E_2\to E_1$ by
\beq
\Psi(e)=U^*eU\rforal e\in E_2.
\eneq
It is a monomorphism. Note $\Psi(\W)=\W.$ By \eqref{MTclass-11}, $\Psi$ is surjective. 
So $\Psi$ is an isomorphism. Moreover, as constructed, one checks that $\Psi$ induces $\Gamma.$

Conversely, if there is an isomorphism $\Psi: E_1\to E_2$, then ${\Psi}$ induces 
an isomorphism $\Gamma: {\rm Inv}(E_1)\to {\rm Inv}(E_2).$
\end{proof}

Towards classifying 
{{\CA s}} 
in $\E$,
we again recall some terminology and other items.

Let $A$ be a \CA\, with ${\tilde T}(A)\not=\{0\}$ and 
with a strictly positive element $e_A.$ 
Denote by $\Sigma_A\in {\rm LAff}({\tilde T}(A))$ the lower semicontinuous affine 
function defined by
$$
\Sigma_A(\tau)=\lim_{n\to\infty} \tau(f_{1/n}(e_A))\rforal \tau\in {\tilde T}(A).
$$
{{One notes that $\Sigma_A$, {{as a lower semi-continuous function on ${\tilde T}(A)$,}} 
is independent of the choice of $e_A.$}}

\begin{df}
For $E\in \E,$ define 
\beq
{\rm Inv}(E)=(K_0(E), K_1(E), {\tilde T}(E), \Sigma_E, \lambda_E).
\eneq
\end{df}

Denote by ${\tilde T}_f(E)$ the set of all faithful traces in 
$\widetilde{T}(E)$, i.e., the set of all $\tau \in \widetilde{T}(E)$
for which $\tau(a)\not=0$ for every  
$a\in {\rm Ped}(E)_+\setminus \{0\}.$
Write $A=\pi_E(E).$ 
Let $(\pi_E)_T: {\tilde T}(A)\to {\tilde T}(E)$ be an affine \hm\, 
defined by $(\pi_E)_T(\tau)(e)=\tau\circ \pi_E(e)$ for all $e\in {{{\rm Ped}(E)}}.$
The cone ${\tilde T}(E)$ is generated by $(\pi_E)_T({\tilde T}(A))$ and $\tau_W.$

Let $E_1, E_2\in \E.$ We say that ${\rm Inv}(E_1)$ and ${\rm Inv}(E_2)$ are isomorphic, 
and write ${\rm Inv}(E_1)\cong {\rm Inv}(E_2),$ if there is {{an isomorphism}}
$$
\Gamma: {\rm Inv}(E_1)=(K_0(E_1), K_1(E_1), {\tilde T}(E_1), \Sigma_{E_1}, \lambda_{E_1})
\cong {\rm Inv}(E_2)=(K_0(E_2), K_1(E_2), {\tilde T}(E_2), \Sigma_{E_2}, \lambda_{E_2}),
$$

i.e., 
if there is a group isomorphism $\Gamma_i: K_i(E_1)\to K_i(E_2),$
$i=0,1,$ and a topological cone isomorphism  $\Gamma_T: {{{\tilde T}(E_1)\to {\tilde T}(E_2)}}$ which 
maps $T_f(E_1)$  onto $T_f(E_2)$ such that
\beq\label{Dinv-paring-one}
\lambda_{E_2}\circ \Gamma_0=\lambda_{E_1} \andeqn
\Sigma_{E_2} \circ \Gamma_T = \Sigma_{E_1}.
\eneq

\begin{lem}\label{Lliftcuntz}
Let $E\in \E$ be an essential extension
of the form:
\beq\label{Liftcuntz-1}
0\to \W\to E\stackrel{\pi_E}{\longrightarrow} A\to 0.
\eneq
Suppose that $e_1, e_2\in E_+$ such that $d_\tau(\pi_E(e_1))=d_\tau(\pi_E(e_2))$
for all $\tau\in {\tilde T}(A)$ and $\W\subset {\rm Her}_E(e_i),$ $i=1,2.$
Then  there is an isomorphism 
\beq
\psi:{\rm Her}_E(e_1)\cong {\rm Her}_E(e_2)
\eneq
such that ${{KL(\psi)=KL({\rm id}_{E})}}$ and $\tau\circ \psi(e)=\tau(e)$ for all $\tau\in {\tilde T}(E).$
\end{lem}

\begin{proof}
Since $A$ has stable rank one (see Theorem 11.5 of \cite{GLII}),
it follows from \cite{CEI} (see  Proposition 3.3 of \cite{Rz}; see also the paragraph above Proposition 3.3 of \cite{Rz} and \cite{LinCuntz}) 
that there is an element  $u\in A^{**}$ such 
that $u{\rm Her}_A(\pi_E(e_1))u^*={\rm Her}_A(\pi_E(e_2)).$ Moreover $u\pi_E(e_1), u^*\pi_E(e_2)\in A$ 
and $u^*u=p$  and $uu^*=q,$ where $p$ and $q$ are open projections of $A$ corresponding 
to $\pi_E(e_1)$ and $\pi_E(e_2),$ respectively.
Let $x=u\pi_E(e_1)\in A.$  
Since $A$ has stable rank one,  by Theorem 5 of  \cite{PedU}, 
for each $n,$  there is a unitary $u_n\in {\widetilde{A}}$ 
such that
\beq
{{u_n\pi_E(f_{1/n}(e_1))={{u}}\pi_E(f_{1/n}(e_1)).}}
\eneq
Since $K_1(\C(\W))=\{0\}$ (see Theorem \ref{K-six}) 
and $\C(\W)$ is purely infinite {{simple}}, 
there is a unitary $w_n\in M(\W)$ such that $\pi(w_n)=u_n.$ Therefore $w_n\in {\widetilde{E}}.$
Since 
\beq
{{u_n\pi_E(f_{1/n}(e_1))u_n^*}}\in {\rm Her}_A(\pi_E(e_2)),
\eneq
\beq
{{w_nf_{1/n}(e_1)w_n^*}}  \in {\rm Her}_E(e_2).
\eneq
(Recall that $\W \subseteq {\rm Her}_E(e_2)$.)
It follows that, for all $n$,   
\beq
{{f_{1/n}(e_1)}}  \lesssim   e_2.
\eneq
Therefore 
\beq
e_1\lesssim e_2.
\eneq
Symmetrically,  
\beq
e_2\lesssim e_1.
\eneq
{{Hence $e_1\sim e_2.$}}
By \eqref{Liftcuntz-1} and the fact that $K_i(\W)=\{0\},$ $i=0,1,$ applying part (ii) of Proposition 4  of \cite{LR},
$E$ has stable rank one. It follows from  \cite{CEI}
 that there is an isomorphism 
$$
\psi: {\rm Her}_E(e_1)\cong {\rm Her}_E(e_2)
$$
such that $\psi(a)=U^*aU$ for some partial isometry $U\in E^{**}$  such that 
${{U^*a, Ub}}\in E$ for all $a\in {\rm Her}(e_1)$ and $b\in {\rm Her}(e_2),$ 
$U^*U=P,$ $UU^*=Q,$ where $P$ is the open projection corresponding to $e_1$ 
and $Q$ is the open projection corresponding to $e_2,$ respectively. 

{{Let $z=U^*e_1\in E.$ 
Since $E$ has stable rank one as we just proved, by  Theorem 5 of \cite{PedU}, 
for each $n,$ there is a unitary $V_n\in {\widetilde{E}}$
such that $V_nP_n=U^*P_n,$ where $P_n$ is the spectral projection of $e_1$ in $A^{**}$
corresponding to $(1/n, \|z\|].$ 
It follows that ${{V_n f_{1/n}(e_1) a f_{1/n}(e_1)V_n^*}}\in 
{\rm Her}(e_2)$ and 
{{
\beq
\lim_{n\to\infty}V_naV_n^*=
\lim_{n \to \infty} V_n f_{1/n}(e_1) a f_{1/n}(e_1) V_n^* = 
\lim_{n\to\infty}U^* f_{1/n}(e_1)a f_{1/n}(e_1)U=\psi(a)
\eneq}}
for all $a \in Her(e_1)$.   
It follows that $KL(\psi)=KL({\rm id}_{E}).$}}
\end{proof}

\begin{thm}\label{MTextbyW}
Let $E_1, E_2\in \E.$ 
Then $E_1\cong E_2$ if and only if there is an isomorphism
$$
\Gamma: {\rm Inv}(E_1)=(K_0(E_1), K_1(E_1), {\tilde T}(E_1), \Sigma_{E_1}, \lambda_{E_1})
\cong {\rm Inv}(E_2)=(K_0(E_2), K_1(E_2), {\tilde T}(E_2), \Sigma_{E_2}, \lambda_{E_2}).
$$
Moreover, if {{such an}}  isomorphism $\Gamma$ exists,  there is an isomorphism $\psi: E_1\to E_2$ which 
induces $\Gamma.$
\end{thm}

\begin{proof}
We begin with  two short exact sequences
\beq\nonumber
0\to \W\to E_1\stackrel{\pi_{E_1}}{\longrightarrow} A\to 0\andeqn
0\to \W\to E_2\stackrel{\pi_{E_2}}{\longrightarrow} B\to 0.
\eneq
Both $A$ and $B$ are separable simple stably projectionless C*-algebras, with finite nuclear dimension which {{satisfy}} the UCT. Moreover, 
$K_0(A)={\rm ker}\rho_A$ and $K_0(B)={\rm ker}\rho_B.$

Let $e_0\in{{ (E_1)_+}}\setminus\{0\}$ with $\|e_0\|=1$ be a strictly positive element of {{$E_1$}}.  
Let $e_1={{f_{1/2}(e_0)}}$. Choose $a_0\in A_+\setminus \{0\}$ such that
$a_0\le \pi_{E_1}(e_1)$  and 
{{$d_\tau(a_0)$ is continuous on}} ${\tilde T}(A)$ {{(see 11.11  of \cite{eglnp} and  Theorem 15.5 of \cite{GLII}).  So ${\rm Her}(a_0)$ has continuous scale (see, for example, Proposition 5.4 of \cite{eglnp}).}}
Choose $a_1'\in  {{(E_1)_+}}$ such that ${{f_{1/8}(e_0)}}a_1'=a_1'$ 
{{and $\pi_{E_1}(a_1')=a_0$
(see Lemma 7.2 of \cite{eglnp}).}} Then $a_1'\in \rm{Ped}({{E_1}}).$
Let $e_W\in \W$ be a strictly positive element in $\W.$ Since ${\rm Ped}(\W)=\W,$  $a_1''=a_1'+e_W\in {\rm Ped}({{E_1}})_+.$ 
Let $a_1=a_1''/\|a_1''\|.$ 
{{Then}} $a_1\in M(\W)_+.$ It follows that $\tau_W(f_{1/n}(a_1))\le 1$ for all $n.$ 
{{Obviously,
$$e_W \lesssim a''_1 \sim a_1.$$}}
We conclude that $d_{\tau_W}(a_1)=1.$ 
{{Since every $t\in {\tilde T}(E_1)$   has the form $\af\cdot  t_A\circ \pi_{E_1}+(1-\af)\cdot t_W,$
where $t_A\in {\tilde T}(A)$ and $0\le \af\le 1,$}} 
one also  verifies that $d_\tau(a_1)$ is continuous on ${\tilde T}(E_1).$

Put $A_1={\rm Her}(a_0)$ and $E_{1,c}={\rm Her}(a_1).$ 
Note that $\pi_{E_1}(E_{1,c})=A_1.$ So $E_{1, c}\in \E_c.$

Let $g\in \Aff({\tilde T}(B))$ such that $g\circ \Gamma_T(\tau)=
d_\tau({{\pi_{E_1}(a_1)}})$
for all $\tau\in {\tilde T}(A).$ 
By  Theorem 11.11 of \cite{eglnp} and  Theorem 15.5 of \cite{GLII}, there exists $b_0\in B_+$ such that $d_t(b_0)=g(t)$ for all $t\in {\tilde T}(B).$
Let $B_1={\rm Her}(b_0).$ Then $B_1$ also has continuous scale (see Proposition 5.4 of \cite{eglnp}). 
Choose $b_1'\in (E_2)_+$ such that $\pi_{E_2}(b_1')=b_0$ and 
$b_1''={{b'_1}}+e_W.$
Set $b_1=b_1''/\|b_1''\|\in E_2\subset M(\W).$ Then, for any $n,$ 
$\tau_W(f_{1/n}(b_1))\le 1.$  It follows that $d_{\tau_W}(b_1)\le 1.$
{{Note that}}  $e_W\lesssim b''_1 \sim  b_1.$ Therefore 
$d_{\tau_W}(b_1)=1.$ 
{{Note, for each $\tau=\af t_B\circ \pi_{E_2}+(1-\af)t_W\in {\tilde T}(E_2),$ where $t_B\in {\tilde T}(B)$ and 
$0\le \af\le 1,$ 
$$
d_\tau(b_1)=\af t_N(b_0)+(1-\af).
$$}}

Define $E_{2,c}={\rm Her}(b_1).$  Then $E_{2,c}\in \E_c.$
We note that $E_{i,c}$ is a full hereditary \SCA\, of $E_i$ which contains $\W$ as an ideal ($i=1,2$).
In particular, $K_i(E_{j,c})=K_i(E_j),$ $i=0,1$ and $j=1,2.$

Consider $T_a=\{\tau\in {\tilde T}(E_1): d_\tau(a_1)=1\}.$
Since ${{d_\tau(a_1)}}$ is continuous {{on ${\tilde T}(E_1)$}}, 
 {{$T_a$}} is a compact convex subset of ${\tilde T}(E_1).$
Note that $\Gamma_T(T_a)=T_g
=\{t\in {\tilde T}(E_2): {{d_{\tau}(b_1)}}=1\}.$
Moreover, $\Gamma_T$ maps $T_a$ {{affinely and 
homeomorphically}} onto $T_g.$ 

Let $\gamma_1: T_a\to T(E_{1,c})$ by
$\gamma_1(\tau)(e)=\tau(e)$ for all $e\in E_{1,c}$ and $\tau\in T_a.$ 
Then $\gamma_1$ is an affine homeomorphism. 
Let $\gamma_2: T_g\to T(E_{2,c})$ be defined by
$\gamma_2(t)(d)=t(d)$ for all $d\in E_{2,c}$ and $t\in T_g.$ Then $\gamma_2$ is {{also}} an affine homeomorphism.

Now define 
\beq
\Gamma': (K_0(E_{1,c}), K_1(E_{1,c}), T(E_{1,c}), r_{E_{1,c}})\to 
(K_0(E_{2,c}), K_1(E_{2,c}), T(E_{2,c}), r_{E_{2,c}})
\eneq
as follows.  $\Gamma_i':=\Gamma_i: K_i(E_{1,c})=K_i(E_1)\to K_i(E_2)=K_i(E_{2,c}),$ $i=0,1,$ {{and}}
$\Gamma_T':=\gamma_2\circ \Gamma_T\circ \gamma_1^{-1}.$ 

We also check 
that, since $\lambda_{E_2}\circ \Gamma_0=\lambda_{E_1},$
for any $x\in K_0(A),$ 
\beq\label{Dinv-paring-2}
\tau_W(\Gamma_0'(x))=\tau_W(\Gamma_0(x))=\tau_W(x).
\eneq
Since $K_0(A)={\rm ker}\rho_A$ and $K_0(B)={\rm ker}\rho_A,$
\eqref{Dinv-paring-2} implies that
\beq\label{Dinv-paring}
r_{E_{2,c}}(\tau)(\Gamma_0'(x))=r_{E_{1,c}}(\Gamma_T^{-1}(\tau))(x)\rforal x\in K_0(E_{1,c})\,\,\, {\rm and}\,\, \tau\in T(E_{2,c}).
\eneq
It follows from Theorem  \ref{MTclass} that there exists an isomorphism 
$\Psi: E_{1, c}\to E_{2,c}$ which induces $\Gamma'.$ 
This provides (still denote by $\Psi$) an isomorphism 
$\Psi: E_{1,c}\otimes {\mathcal K}\to E_{2,c}\otimes {\mathcal K}.$

{{By Brown's stable isomorphism theorem (\cite{Brstable}), 
we may view $E_1$ as a full hereditary \SCA\, of $E_{1,c}\otimes {\mathcal K}.$
Then  we obtain 
an embedding $\Psi|_{E_1}: E_1\to E_{2,c}\otimes {\mathcal K}.$}} Let $e_1$ be a strictly positive element of $E_1$
and $e_2'=\Psi(e_1).$ 
Let $e_2$ be a strictly positive element of $E_2.$ 
Since $\Psi$ induces $\Gamma'$ and since
$\Sigma_{E_2} \circ \Gamma_T = \Sigma_{E_1}$, we have that
$$
d_\tau({{e_2'}})= d_{\tau}(e_2) {{\rforal \tau \in {\tilde T}(E_2).}}
$$

Finally, by applying {{Lemma}} \ref{Lliftcuntz},  there is an isomorphism 
$\psi:{\rm Her}(e_2')\cong {\rm Her}(e_2)$ with $KL(\psi)=KL({\rm id}_{E_2})$ and 
preserves the traces. Therefore 
\beq
E_1\cong E_2
\eneq
and the isomorphism preserves $\Gamma'.$

\end{proof}

\begin{rmk}\label{Rnotrivlaext}
{\rm We do not include a classification statement for the essential extensions 
of the form in \eqref{Shortexinnons} for the case that $A$ {{does}} not have continuous scale in section 7. Theorem \ref{MTextbyW}
is a classification with different favor.  It should be noted though that, if $A$ does not have continuous scale,
then there may not have any trivial extensions of the form in \eqref{Shortexinnons}.
To see this, consider the case that $A=A\otimes {\mathcal K}.$ Then $A$ does not have any faithful 
tracial states. If there were a monomorphism $j: A\to M(\W)$ such that
\beq
\pi_E\circ j={\rm id}_A,
\eneq
then $\tau_W\circ j$ would be a faithful tracial state of $A.$ This is not possible. So no essential extensions 
of the form in \eqref{Shortexinnons} splits.   This explains, partially,  why we choose not to include this case in 
section 7. 

}

\end{rmk}

\vspace{0.4in}


\section{Extensions by a simple \CA\, in $\I$}

In this section we consider the essential extensions of the form:
$$
0\to B\to E\to C\to 0,
$$
where $B$ is a separable simple \CA\, with a unique tracial state and with finite nuclear dimension 
and satisfies the UCT, and $C$ is  a 
separable amenable \CA\, which is $\W$ embeddable. 

\begin{df}\label{DII}
Denote by $\I$ the class of all nonunital {{stably projectionless}}
separable simple amenable ${\Z}$-stable \CA s with a unique tracial state
which satisfies the UCT.
\end{df}
Note that if $B\in I,$ then 
$K_0(M(B))=\R,$  $K_1(M(B))=\{0\}$ and $K_1(\C(B))=K_0(B)$
 and  {{$K_0(\C(B))=\R\oplus K_1(B)$}}
 (see Theorem \ref{K-six}).
{{\CA s in $\I$ {{have}} been classified by their Elliott invariant in \cite{GLII}. 
All \CA s in $\I$ have stable rank one.  Moreover ${\rm ker}\rho_B=K_0(B)$
for every \CA\, $B\in \I$ (see Lemma \ref{lem:Feb2620203PM}).
 We will also use the fact that every hereditary \SCA\, of \CA s in 
$\I$ is in $\I.$}}

\begin{lem}\label{DTextforI-n}
Let $B\in \I.$
Let $A$ be a separable amenable \CA\, which is $\W$ embeddable. 
Then there are $\T_d$ extensions (see \ref{DTextension}) 
$\sigma: A\to \C(B)=M(B)/B$
with model $\sigma_A$ which factors through $\W$ and, in particular,
$KK(\sigma_A)=0.$

\end{lem}

\begin{proof}
Fix an embedding $\iota_A: A\to \W$ and an embedding $\iota_{w,B}: \W\to B$ 
(given, for example, by \ref{LWembB})
which maps strictly positive elements to strictly positive elements {{(see  \ref{DWemb}).}}
Put $\sigma_A=\iota_{w,B}\circ \iota_A: A\to B.$
Denote by $\tau$ the unique tracial state of $B.$
Fix a system of quasidiagonal units $\{b_k\}$ as in \ref{Dqdunit}. 
Passing to a subsequence, \wilog, we may assume that 
\beq
\sum_{k=n+1}^{\infty}d_\tau(b_k)<(1/n)d_\tau(b_n)\rforal n.
\eneq
Let $t_n=(1/n)d_\tau(b_k),$ $n\in \N.$
There is an element $a_n\in {\rm Her}(b_n)$ with $d_\tau(a_n)=t_n$ (since $B$
has strict comparison and since ${\rm Cu}(B) \cong V(B) \sqcup (0, \infty]$ ---see Proposition 11.11 of \cite{eglnp}
and Theorem 15.5 of \cite{GLII}). 
Moreover, ${\rm Her}(b_n)\cong M_n({\rm Her}(a_n)).$
By part (2) of \ref{LBembed}, there is, for each $n,$ a \hm\, $\phi_n: B\to {\rm Her}(a_n).$
Define $\sigma: A\to M(B)$ by
$\sigma(a)=\sum_{n=1}^\infty (\bigoplus^n \phi_n\circ \sigma_A)(a)$ for all 
$a\in A.$ One then checks, from Definition \ref{DTextension}, that 
$\pi\circ \sigma$ is a $\T_d$ extension  with model $\sigma_A$ which factors 
through $\W$ and $KK(\sigma_A)=0.$
\end{proof}

\begin{lem}\label{GLuniq}
Let $C$ be a separable amenable \CA\,  which is $\W$ embeddable and satisfies the UCT. Suppose 
$K_i(C)$ is finitely generated ($i=0,1$).  Let $B\in {\mathcal I}$ and let
$\pi\circ \sigma: C\to M(B)/B$ {{be}} {{an essential}}  $\T_d$-extension with the model map
$\sigma_C$ {{factoring through $\W.$}}
Then, for any {{essential}} quasidiagonal extension $\tau_q: {{C}}\to M(B)/B,$ there is a 
trivial diagonal {{essential}} extension $\sigma_d: {{C}}\to M(B)/B$ such that
\beq
\tau_q\oplus \pi\circ \sigma\sim^u  \sigma_d\oplus \pi\circ \sigma.
\eneq

\end{lem}

\begin{proof}
Exactly as at the beginning of the proof of \ref{Uniq-2},
\wilog, we may assume 
that ${\rm ran}(\sigma) \perp {\rm ran}(\psi).$

We write 
$$\sigma = \bigoplus_{n=1}^{\infty} \bigoplus^{{n+1}} \phi_{{n}}\circ {{\sigma_C}}$$
as in Definition \ref{DTextension}  and Lemma \ref{DTextforI-n}. 

Since $\tau_q$ is quasidiagonal, we may write 
$\psi=\bigoplus_{n=1}^\infty\psi_n$ such that $\pi\circ \psi=\tau_q$ with  a 
system of quasidiagonal units $\{a_n\}$ 
from Proposition \ref{prop:QDExt} that corresponds to 
$\{ \psi_n \}$.  Recall
that
\beq
\lim_{n\to\infty}\|\psi_n(ab)-\psi_n(a)\psi_n(b)\|=0\rforal a, b\in C.
\eneq

We now write $\bigoplus^{{n+1}}\phi_n\circ {{\sigma_C}}=\sigma_{n,{0}}\oplus \sigma_{n,{{1}}}\oplus \cdots \oplus  \sigma_{n,n}$
and $\sigma=\bigoplus_{n=1}^\infty \bigoplus_{{j=0}}^n \sigma_{n,j}.$

Since $\sigma$ is a $\T_d$ extension 
{{as in  \ref{rmk:Textension},}}
there exists a map
$F : C_+ \setminus \{ 0 \} \rightarrow \mathbb{N} \times 
(0, \infty)$
such that for all $n, j$,
$\sigma_{n,j} : C\rightarrow \overline{b_{n,j} 
B b_{n,j}}$
is $F$- full.

Following Definition \ref{DTextension}
let
$$b_{n,j} := \sigma_{n,j}(e_C)\,\, \rforal n,j.$$

Let $\{ \epsilon_n \}_{n=1}^{\infty}$ be a strictly decreasing
sequence in $(0,1)$ such that
$\sum_{n=1}^\infty\epsilon_n <\infty.$
Let ${\mathcal F}_1\subset {\mathcal F}_2\subset \cdots \subset {\mathcal F}_n\subset \cdots $ be
a sequence of finite subsets of the unit ball of $C$ which is dense in the unit ball of $C.$

We will apply {{Theorem \ref{TK1stuniq}.}}
Note, by Proposition \ref{Wprop},  {{every \SCA\, of $B$ is in}}  ${\bf C}_{0,0,1,T,7}$, with
$T$ as in Proposition \ref{Wprop}.  Let $L:=7\pi+1.$
{{Recall that, as $C$ is given, we fix maps $J,$ $\Pi_{cu}^-$ and $J^\sim$ as  in \ref{Dcoset}.}}

For each $n,$ let $\delta_n>0,$  $\G_n\subset C$ be a finite subset, ${\mathcal P}_n (\subset {\mathcal P}_{n+1})\subset \underline{K}(C)$ be a finite
subset, $\U_n(\subset \U_{n+1})\subset J^\sim(K_1({{C}}))$
be
a finite subset, $\E_n\subset C_+\setminus \{0\}$ be a finite subset,  and $K_n$ be an integer associated with $\F_n$ and $\ep_n$ (as well as $F$ and $L$ above) as  provided
by  Theorem \ref{TK1stuniq}
(for \CA s in ${\bf C}_{0,0,1,T,7}$).

We may assume that $\delta_{n+1} < \delta_n$,  $\G_n \subseteq \G_{n+1}$,
{{$K_n< K_{n+1},$}}
$\U_n \subset U(M_{m(n)}(\widetilde{C})),$ {{$\U_n\subset \U_{n+1}$}}
for all $n$.  \Wlog, we may assume that each $\psi_n$ is $\G_n$-$\dt_n$-multiplicative and $\lceil \psi_n(u) \rceil$ is well defined for all $u\in \U_n.$

Moreover, \wilog, we may also assume that (see Theorem 14.5 of \cite{Lnloc}),
for any $n,$ there is a group \hm\,
$$
\lambda_n: G(\Pi_{cu}(\U_n))\to U(M_{m(n)}(\widetilde{{\rm Her}(a_n)}))/CU(M_{m(n)}(\widetilde{{\rm Her}(a_n)})).
$$
such that
\beq\label{Uni-2-e-1-n}
{\rm dist}(\lambda_n(x), \Pi_{cu}(\lceil \psi_n(J^\sim(x))\rceil))<1/16\pi (n+1)\rforal x\in \Pi_{cu}(\U_n),
\eneq
where $G(\Pi_{cu}(\U_n))$  is the subgroup generated  by the finite subset $\Pi_{cu}(\U_n).$
Recall that $\Pi_{cu}\circ \Pi_{cu}^-{{(J^\sim (x))}}=J(x)$ for all $x\in K_1(C).$
\Wlog, we may assume that ${{{\mathcal P}_n}}\cap K_1(C)$ generates the same
group as $\Pi\circ \Pi_{cu}(\U_n)$ does (in the current case, $K_1(C)$ is assumed to be finitely generated).

Moreover, since $K_i(C)$ is finitely generated, we may assume that
$KL(\psi_n)$ is well defined, and, since $\lambda_n$ is determined by $\psi_n,$ we may also assume
that   $\lambda_n\circ J$ is compatible with $KK(\psi_n).$

Again,  throwing away finitely many terms and relabelling
if necessary, we may assume that
$$\sum_{n=1}^{\infty} d_{\tau}(a_n) < d_{\tau}{{(b_{K_1,0})}},$$
where $\tau$ is the unique tracial state of $B$.
Let $\{ n_k \}_{k=1}^{\infty}$ be a subsequence of $\mathbb{Z}^+$
with
$n_1 = 1$ and
$n_k +2 < n_{k+1}$ for all $k$ such that
$$\sum_{l=n_k}^{\infty} d_{\tau}(a_l) < d_{\tau}(b_{K_k, {{0}}}).$$
Since $B$ has stable rank one, there is a unitary $U_k'\in {\widetilde{B}}$ such that
\beq
{{(U_k')^*}}((\sum_{l=n_k}^{n_{k+1}-1} a_l)B (\sum_{l=n_k}^{n_{k+1}-1} a_l ))U_k'\in \overline{b_{K_k,{{0}}} B b_{K_k,{{0}}}}.
\eneq
Put $B_{k,0}={\rm Her}((\sum_{l=n_k}^{n_{k+1}-1} a_l)B (\sum_{l=n_k}^{n_{k+1}-1} a_l )).$
Note $(U_k')^*B_{k,0}{{U_k'}}\subset {\rm Her}(b_{K_k,{{0}}}).$

For each ${{k,}}$ by Theorem \ref{TExpon},
{{there is a \hm\, $h_k: C\to B_{k,0}$ such}}
that $KL(h_k)=\sum_{l=n_k}^{n_{k+1}-1} KK(\psi_l)$ and $h_k^\ddag=\bar{\lambda}_n\circ J.$

In other words,
\beq\label{GTuniq-KK0}
[\Sigma_{l=n_k}^{n_{k+1}-1} \psi_l]|_{{\mathcal P}_k}=[h_k]|_{{\mathcal P}_k}
\andeqn {h_k^\ddag}|_{\Pi_{cu}(\U_k)} =
{{\lambda_n}}|_{\Pi_{cu}(\U_k)}.
\eneq

By  the second part of
\eqref{GTuniq-KK0} and \eqref{Uni-2-e-1-n}, for any $u\in \U_{n_k},$  there is $v_l\in CU(M_{m(l)}(\widetilde{{\rm Her}(a_l)}))$ such that
\beq\label{Uniq-3-e-2-1}
h_l(u)\lceil \psi_l(u)\rceil ^*\approx_{1/16\pi({{l}}+1)} v_l.
\eneq
It follows from Lemma \ref{Lexpincum} that, for all $u\in \U_{n_k}$ {{(computed in $M_{m(n_k)}(B_{k,0})$),}}
\beq\label{Uniq-3-e-2}
{\rm cel}({\rm Ad}\,U_k'\circ (\Sigma_{l=n_k}^{n_{k+1}-1}h_l)(u)\lceil {\rm Ad}\,U_k'\circ (\Sigma_{l=n_k}^{n_{k+1}-1}\psi_l)(u)\rceil^*)\le
7\pi+1.
\eneq

For each $k,$ {{consider}}  two maps ${\rm Ad}\, U_k'\circ (\Sigma_{l=n_k}^{n_{k+1}-1}\psi_l),\, {\rm Ad} U_k'\circ h_k: C\to {\rm Her}(b_{K_k,0}).$

Recall that $\psi_n$ is $\G_n$-$\delta_n$-multiplicative and $\phi_n\circ \sigma_A$ is $F$-full
for all $n.$  Also keeping in mind of \eqref{GTuniq-KK0}
and \eqref{Uniq-3-e-2}.
Applying  Theorem \ref{TK1stuniq},
for each $k$,  there is a unitary $u_k'\in
M_{K_k + 1}(\widetilde{{\rm Her}(b_{K_k,{{0}}})}$
such
that
\beq\label{sec10-uniq-1}
&u_k'\left({U_k'}^* h_k(c) U_k' \bigoplus \sum_{l=1}^{K_k} \sigma_{K_k, {{l}}})(c)\right)u_k^*
\approx_{\epsilon_k}
{U_k'}^*\sum_{l=n_k}^{n_{k+1} - 1} \psi_l(c)U_k' +
\sum_{l=1}^{K_k} \sigma_{K_k, l}(c)
\eneq
for all $c\in  \F_k.$

Define $H: M(B)$ by $H(c)=\bigoplus_{k=1}^\infty h_k(c)$
for all $c\in C.$ Note the sum converges strictly and $H$ is a \hm.
Set $\sigma_d:=\pi\circ H.$

{{By}} exactly the same argument as in the later part of the proof of Lemma \ref{Uniq-2}, from \eqref{sec10-uniq-1}, 
we obtain a unitary $u\in \C(B)$ such that
\beq
&&u(\sigma_d(c)\oplus \pi\circ \sigma(c))u^*=\pi\circ \psi(c)+\pi\circ \sigma(c)\rforal c\in C.
\eneq

Now since {{${\rm Her}(\pi\circ \psi(e_C)+\pi\circ \sigma(e_C))^\perp\not=\{0\},$}} as $C(B)$ is purely infinite simple,
there is a non-zero projection $e_1\in {\rm Her}{{(\pi\circ \psi(e_C)+}}\pi\circ \sigma(e_C))^\perp\not=\{0\}.$
There is a unitary $v\in e_1\C(B)e_1$ such that $[v]=[u^*].$
Let $u_1=u(v\oplus (1-e_1)).$  Replacing $u$ by $u_1,$ we may assume that
$u_1\in U_0(\C(B)).$  Therefore we may assume that there a unitary $U\in M(B)$ such that $\pi(U)=u.$

\end{proof}

\begin{cor}\label{Lzeroabsb}
In Lemma \ref{GLuniq}, if $\tau_q$ is in fact a diagonal extension, i.e.,
$\tau_q=\pi\circ \Psi,$ where $\Psi: C\to M(B)$ is defined by
$\Psi(c)=\bigoplus_{n=1}^\infty \psi_n,$ where $KK(\psi_n)=0$ and $\psi_n^\ddag=0,$
then
$$
\pi\circ \Psi\oplus \pi\circ \sigma\sim^u  {{\pi\circ}}\sigma.
$$

\end{cor}

\begin{proof}
In the proof of Lemma  \ref{GLuniq},  let each $\psi_n$ be a \hm\, such that  $KK(\psi_n)=0$
and $\psi_n^\ddag=0.$
Since $\psi_n^\ddag=0,$ $\psi_n(u)\in CU(M_{m(n)}(\W))$ (in stead of \eqref{Uni-2-e-1-n})
for all $u\in J^\sim(K_1(A))\cap U(M_{m(n)}(\widetilde{{\rm Her}(a_n)})).$
Therefore, as in the proof of Lemma \ref{GLuniq},  \eqref{Uniq-3-e-2}
becomes,
\beq
{\rm cel}(\sum_{l=n_k}^{n_{k+1}-1}\psi_l(u))\le 7\pi+1\rforal u\in \U_{n_k},
\eneq
if  we choose $h_l(u)=1.$
Therefore the proof works when we use $h_n=0$ for all $n.$
In other words,
$$
\pi\circ \sigma\sim^u \pi\circ \Psi\oplus \pi\circ\sigma.
$$

\end{proof}

\begin{lem}\label{Lpermuzero}
Let $C$ and $B$ be as in  \ref{GLuniq}.
Fix two sequences
\beq
\{x_n\}\subset KL(C, {{B}})\andeqn
\{y_n\}\in {\rm Hom}(K_1(C), U(B)/CU(B))
\eneq
 such
that $x_n$ and $y_n$ are compatible, {{i.e., $x_n(z)=\Pi_{1, cu}(y_n(z))$ for all $z\in K_1(C).$}}
There is a diagonal \hm\, $h_d:=\bigoplus_{n=1}^\infty h_n: C\to M(B),$ where
$h_n: A\to {\rm Her}(b_n)$ is a \hm\, and where
$\{{{b_n}}\}$ is a system of quasidiagonal units. Moreover,
for each $m,$ $KK(h_m)=x_n$ and $h_m^\ddag=y_n$ for some $n,$ and, for each $k,$
there {{are}} infinitely many $h_{n_k}$'s  such that $KK(h_{n_k})=x_k$ and $h_{n_k}^\ddag=y_k$
(at the same time).

\end{lem}

\begin{proof}
 Write $\N=\cup_{n=1}^\infty S_n,$ where each $S_n$ is an infinite
countable set, and
$S_i\cap S_j=\emptyset,$ if $i\not=j.$
 For each $j\in S_n,$ choose a \hm\, $h_j: C\to {\rm Her}(b_j)$ such
that $KK(h_j)=x_n$ and $h_j^\ddag=y_n$ (see Theorem \ref{T-existoI}). Then {{set}} $h_d:=\bigoplus_{k\in \N} h_k.$
It is ready to check that $h_d$ satisfies the requirements of the lemma.

\end{proof}

\begin{lem}\label{Lpermzero2}
Let $C$ and $B$ be as in  \ref{GLuniq}.
Then, any two diagonal  {{essential}} trivial extensions are unitarily equivalent.

\end{lem}

\begin{proof}

{{Let}} $\Psi=\bigoplus_{n=1}^\infty \psi_n:  C\to M(B)$
be any  diagonal map, where $\psi_n: C\to\overline{b_nBb_n}$ is a \hm, {{and}}
where $\{b_k\}$ is a system of quasidiagonal units.
Let $x_{2n-1}=KK(\psi_n),$ $y_{2n-1}=\psi_n^\ddag,$ $x_{2n}=-KK(\psi_n)$ and $y_{2n}=-\psi_n^\ddag,$
$n=1,2,....$  Note that $x_n$ and $y_n$ are compatible.

Let $h_d: C\to M(B)$ be as in  Lemma \ref{Lpermuzero} associated with the sequence
$\{x_n\}$ and $\{y_n\}.$
We claim that there is a permutation $\lambda:\N\to \N$
such that
$$
KK(h_{\lambda(2n-1})=-KK(h_{\lambda(2n)})\andeqn h_{\lambda(2n-1)}^\ddag=-h_{\lambda(2n)}^\ddag,\,\,\, n=1,2,.....
$$
For $n=1,$ there is $\gamma(1)\in \N$ such that $KK(h_{\gamma(1)})=-KK(h_1)$ and $h_{\gamma(1)}=-h_1^\ddag.$

Then, define $\lambda(1)=1$ and $\lambda(2)=\gamma(1).$
Suppose that $\lambda$ has be defined on $\{1,2,...,2n\}$
such that
$$
KK(h_{\lambda(2k-1)})=-KK(h_{\lambda(2k)}) \andeqn h_{\lambda(2k-1)}^\ddag=-h_{\lambda(2k)}^\ddag,\,\,\,k=1,2,...,n.
$$
Choose the smallest integer $m$ such that {{$m\not=\{\lambda(1), \lambda(2),...,\lambda(2n)\}.$}}
Define $m=\lambda(2n+1).$ Note that $KK(h_m)\in \{x_n\}.$  Find $m'\in \N\setminus \{\lambda(j): 1\le j\le 2n\}\cup \{m\}$
such that $KK(h_{m'})=-KK(h_{\lambda(2n+1)}).$ Note {{that}} one also has $h_{m'}^\ddag=-h_{\lambda(2n+1)}^\ddag.$
Define $m'=\lambda(2n+2).$ The claim follows from the induction.

Define $a_k=d_{\lambda(2k-1)}+d_{\lambda(2k)},$ $k=1,2,...,$
Then $\{a_k\}$ is also a system of matrix units.
Let $h_{n,0}: C\to {\rm Her}(d_{\lambda(2n-1)}+d_{\lambda(2n)})$
be defined by $h_{n,0}(c)=h_{\lambda(2n-1)}(c)+h_{\lambda(2n)}(c)$
for all $c\in C.$
Now define $H_0: C\to M(B)$ by $H_0(c)=\bigoplus_{n=1}^\infty h_{n,0}.$
Then $H_0$ is unitarily equivalent to $h_d$ (see \ref{Lpermutation}).
However, $KK(h_{n,0})=0$ and $h_{n,0}^\ddag=0.$
It follows from Corollary \ref{Lzeroabsb} that
$$
\pi\circ H_0\oplus \pi\circ \sigma\sim^u \pi\circ \sigma.
$$
Therefore
\beq\label{Lpermzero2-1}
\pi\circ h_d\oplus \pi\circ \sigma\sim ^u  {{\pi\circ}}\sigma.
\eneq



Hence $\Psi\oplus h_d$ is another diagonal map and
if we write $\Psi\oplus h_d=\bigoplus_{n=1}^\infty h_n',$
then, $KK(h_n')\in \{x_n\}$ and, for each $x_n,$
there are infinitely many $\{n_k\}\subset \N$
with $KK(h_{n(k)}')=x_n$ and $h_{n(k)}'= y_n$ ($x_n$ and $y_n$ are compatible).  From what has been proved,
we conclude that
\beq
(\pi\circ \Psi\oplus \pi\circ h_d)\oplus \pi\circ \sigma \sim^u \pi\circ \sigma.
\eneq
Then, by \eqref{Lpermzero2-1},
$$
\pi\circ \Psi\oplus \pi\circ \sigma\sim^u \pi\circ \Psi\oplus (\pi\circ h_d\oplus \pi\circ \sigma)
\sim^u\pi\circ \sigma.
$$

\end{proof}

\begin{thm}\label{GMT}
Let $B\in \I$ and let $A$ be a separable amenable \CA\,  which is $\W$ embeddable and satisfies the UCT. Suppose
$K_i(A)$ is finitely generated ($i=0,1$).

(1) If $\tau_1, \tau_2: A\to \C(B)$ are two essential extensions,
then $\tau_1\sim^u \tau_2$ if and only if $KK(\tau_1)=KK(\tau_2).$

(2)  The map $\Lambda: {{\Ext^u(A,B)}}\to KK(A, \C(B))$
defined by $\Lambda([\tau])=KK(\tau)$  is a group isomorphism.

(3) An essential extension $\tau$ is trivial and diagonal if and only if $KK(\tau)=0,$
and all trivial and diagonal extensions are unitarily equivalent.

(4) All quasidiagonal {{essential}} extensions are  trivial and are unitarily equivalent.

(5) If ${\rm ker}\rho_{f,A}=K_0(A),$   then all trivial {{essential}} extensions are unitarily equivalent.
Moreover, $\tau$ is trivial if and only if $KK(\tau)=\{0\}.$

(6)  If ${\rm ker}\rho_{f,A}\not=K_0(A),$ then there are essential trivial extensions of $A$ by $B$ which
are not quasidiagonal and not all essential trivial extensions are unitarily equivalent.


\end{thm}

\begin{proof}
Let us prove (3) first.
Suppose that $KK(\tau)=0.$
Consider a $\T_d$ extension $\pi\circ\sigma$ as in Lemma \ref{DTextforI-n}. Then, $KK(\sigma)=0.$
It follows that $KK(\pi\circ \sigma)=0.$
Let $\widetilde{\pi\circ \sigma},\,\tilde{\tau}: {\widetilde{A}}\to \C(B)$ be the unital extensions of $\pi\circ \sigma$
and $\tau,$ respectively. By Theorem 2.5 of \cite{LinFullext},
there exists a sequence of unitaries $v_n\in \C(B)$ such
that
\beq
\lim_{n\to\infty} v_n^*(\pi\circ \sigma(a))v_n=\tau(a)\rforal a\in  A.
\eneq
Since $A$ is not unital and $\C(B)$ is purely infinite, by \ref{W=S}, we obtain a sequence of
unitaries $u_n\in M(B)$ such that
\beq
\lim_{n\to\infty}\pi(u)^*(\pi\circ \sigma(a))\pi(u_n)=\tau(a)\rforal a\in A.
\eneq
By Theorem \ref{Tappt=qd},  $\tau$ is a quasidiagonal extension.
By Lemma \ref{GLuniq} and Lemma \ref{Lpermzero2},
$$
\tau\oplus \pi\circ \sigma\sim^u \pi\circ h_d\oplus \pi\circ \sigma\sim^u \pi\circ \sigma.
$$
On the other hand, by Theorem \ref{Tgroup},
$$
\tau\sim^u \pi\circ \sigma\oplus \tau_1
$$
for some essential extension $\tau_1.$ We then compute that $KK(\tau_1)=KK(\tau)=0.$
From what has just been proved,
$$
\tau_1\oplus \pi\circ \sigma\sim^u \pi\circ \sigma.
$$
It follows that
\beq\label{GMT-10}
\tau\sim^u \pi\circ \sigma.
\eneq
This shows that if $KK(\tau)=0$ then $\tau\sim^u \pi\circ \sigma,$ in particular,
$\tau$ is trivial and diagonal.

Conversely, suppose $\tau: A\to \C(B)$ is a trivial diagonal extension.
Then, by \ref{Lpermzero2}, $\tau\sim^u {{\pi\circ  \sigma}}.$ It follows that $KK(\tau)=0.$
This proves (3).



Let $\Lambda: \Ext^u(A, B)\to KK(A, \C(B))$ be the map defined by
$\Lambda([\tau])=KK(\tau).$ It is a semigroup \hm.

To see it is surjective, let $x\in KK(A, \C(B)).$ Note that $KK({\widetilde{A}}, \C(B))=KK(A,\C(B))\oplus KK({\mathbb C}, \C(B))
=KK(A, \C(B))\oplus K_0(\C(B)).$
Let $y:=x\oplus [1_{\C(B)}]\in KK(C, \C(B))\oplus K_0(\C(B)).$ By  Corollary 8.5 of \cite{LinFullext},
there exists a monomorphism $\phi: {\widetilde{A}}\to \C(B)$  such that $KK(\phi)=y.$
Put $\phi_0:=\phi|_A: A\to \C(B).$ Then $KK(\phi_0)=x.$
This shows that $\Lambda$ is surjective.

Fix an essential extension $\tau: A\to \C(B).$ Since $\Lambda$ is surjective, there exists $\tau_{-1}$ such that
$KK(\tau_{-1})=-KK(\tau).$ Then $KK(\tau\oplus \tau_1)=0.$ By part (3),
\beq
\tau\oplus \tau_{-1}\sim^u \pi\circ \sigma.
\eneq
Let $\tau_1: A\to \C(B)$ be any essential extension with $KK(\tau_1)=KK(\tau).$ Then
\beq\label{GMT-2-11}
\tau\oplus \pi\circ \sigma\sim^u \tau\oplus ((\tau_{-1}\oplus \tau_1)\sim^u (\tau\oplus \tau_{-1})\oplus \tau_1
\sim^u \pi\circ \sigma\oplus \tau_1.
\eneq
On the other hand, by Theorem \ref{Tgroup},   for some essential extension $\tau',$
\beq
\tau\sim^u \pi\circ \sigma\oplus \tau'.
\eneq
One then computes that $KK(\tau')=KK(\tau).$ Therefore
\beq
\tau\sim^u \pi\circ \sigma\oplus \tau'\sim^u \pi\circ \sigma\oplus \tau.
\eneq
Hence, also, $\tau_1\sim^u \pi\circ \sigma \oplus \tau_1.$
By \eqref{GMT-2-11},
\beq
\tau\sim^u \tau_1.
\eneq
This implies that $\Lambda$ is one-to-one. Since $KK(A,\C(B))$ is a group,
this implies that $\Ext^u(A, B)$ is a group
with zero $[\pi\circ \sigma].$ Moreover $\Lambda$ is a group \hm.
This proves (1) and (2).

To see (4) holds, let $\tau_q$ be a quasidiagonal extension.
Then, by Lemma \ref{GLuniq},
$$\tau_q\oplus \pi\circ \sigma\sim^u \tau_d\oplus \pi\circ \sigma.$$
By (2) and (3), $KK(\tau_q)=0.$ It follows from (3) that $\tau_q\sim^u \pi\circ \sigma.$
Thus (4) holds.

To see (5), let $\tau=\pi\circ H$ for some monomorphism $H: A\to M(B).$
Since we now assume $K_0(A)={\rm ker}\rho_{f,A},$ $H_{*0}=0.$
Hence
$KK(\tau)=0.$  Then (5) follows from (3).

Finally, for (6), we note that, if ${\rm ker}\rho_{f,A}\not=K_0(A),$
then there is $t\in T_f(A)$ such that $\rho_A(t)\not=0.$

By Lemma \ref{Lell-2},
for any $t\in (0,1),$  there is a monomorphism
$\psi_{A,r}: A\to M(B)$ such that $t_B\circ \phi_{A,r}(a)=r\cdot t(a)$ for all $a\in A.$
Then $\pi\circ \psi_{A,r}: A\to \C(B)$ is an essential trivial extension such that
\beq
KK(\pi\circ \psi_{A,r})\not=0.
\eneq
Thus we produce an essential trivial extension which is
not unitarily equivalent to the trivial diagonal extension  $\pi\circ\sigma$ ($KK(\pi\circ \sigma)=0$).
In fact, if $r_1, r_2\in (0,1)$ and $r_1\not=r_2,$ then
$\pi\circ \psi_{A, r_1}$ is not unitarily equivalent to $\pi\circ \psi_{A, r_2}.$

\end{proof}

\begin{rmk}\label{KKnotation}
One may notice that Theorem \ref{GMT}  does not describe exactly
what the set $\T,$ i.e., the trivial essential extensions look like.
The next statement will do that.

{{Recall that $KK(A, M(B))={\rm Hom}(K_0(A), \R).$  So we may view  ${\rm Hom}(K_0(A), \R)_{T_f(A)}$ as  a subset of $KK(A, M(B)).$}}
Let $[\pi]: KK(A,M(B))\to KK(A, \C(B))$ be a \hm\, induced by
the quotient map $\pi: M(B)\to \C(B).$ Define
$$
N=[\pi]{{(\{r\cdot h: r\in [0,1]: h\in {\rm Hom}(K_0(A), \R)_{T_f(A)}\}).}}
$$

\end{rmk}

\begin{thm}\label{LT}
Let $B\in \I$ and let $A$ be a separable amenable \CA\,  which is $\W$ embeddable and satisfies the UCT. Suppose
$K_i(A)$ is finitely generated ($i=0,1$).
Then

(i)  the map $[\pi]$ is one-to-one on ${\rm Hom}(K_0(A), \R)_{{{T_f(A)}}},$   and,

(ii) an essential extension $\tau: A\to \C(B)$ is trivial, if and only if,
$$
\Lambda([\tau])=KK(\tau)\in  N=[\pi]({{\{r\cdot h: r\in [0,1],\, h\in {\rm Hom}(K_0(A), \R)_{T_f(A)}\}}}).
$$
Moreover, $\Lambda$ is one-to-one on $\T,$ the set of unitarily equivalence classes
of trivial essential extensions of $A$ by $B.$

\end{thm}

\begin{proof}
Fix $t\in T_f(A)$ and $r\in (0,1].$ By {{Lemma \ref{Lell-2},}}
there is a monomorphism
$\psi_{A, r}: A\to M(B)$ such that $t_B\circ \psi_{A,r}(a)={{r\cdot}} t(a)$
for all $a\in A.$ Thus $KK(\psi_{A,r})\in {\rm Hom}(K_0(A), \R)_{T_f}.$
Note that $K_0(\C(B))=\R\oplus K_1(B)$ and $[\pi]$ maps
injectively from ${\rm Hom}(K_0(A), \R)$ into ${\rm Hom}(K_0(A), \C(B))$
as $\pi_{*0}$ is injective from $\R$ into $\R\oplus K_1(B).$
This proves part (i).

For (ii), suppose that $\tau$ is an essential trivial extension.
Then there is a monomorphism $H: A\to M(B)$ such that
$\tau=\pi\circ H.$ Then $t_B\circ H$ is a bounded trace with $r:=\|t_B\circ H\|\le 1.$
Therefore $t_B\circ H=r\cdot t$ for some $t\in T_f(A).$
It follows that $KK(H)\in {\rm Hom}(K_0(A), \R)_{T_f}.$ Therefore
$\Lambda([\tau])\in [\pi]( {\rm Hom}(K_0(A), \R)_{T_f}).$

{{The converse}} follows {{from}} the proof of part (i).
To see this, let $\tau: A\to \C(B)$ an essential extension such that $KK(\tau)\in N.$
So there is $\lambda\in {\rm Hom}(K_0(A), \R)_{{{T_f(A)}}}$ such that $KK(\tau)=[\pi]\circ {{r\cdot }} \lambda$
{{for some $r\in [0,1].$}}
Let $\psi_{A,r}: A\to M(B)$ be constructed in the proof of part (i) so that $KK(\psi_{A,r})=\lambda.$
Then $KK(\pi\circ \psi_{A,r})=KK(\tau).$ By part (2) of  Theorem \ref{GMT},
$\tau\sim^u \pi\circ \psi_{A,r}.$

The ``Moreover" part  also follows from the part (2) of Theorem \ref{GMT}.

\end{proof}

\begin{rmk}\label{Rfinal}
{\rm {{Theorem \ref{LT}  uses $N$ to describe the trivial essential extensions under the assumption of this section (see also Theorem \ref{GMT}).
When ${\rm ker}\rho_{f,A}\not=K_0(A),$ $N\not=\{0\}.$ In fact there are uncountably many different 
elements in $N.$  Moreover, ${\mathcal T}$ is not a semigroup. 
One first note that,  for any  $\lambda\in {\rm Hom}(K_0(A), \R))_{T_f(A)}$ and $r\in (1, \infty),$ if 
$\tau$ is an essential extension with $KK(\tau)=[\pi]\circ (r\cdot \lambda),$ then $\tau$ is not a splitting extension,
since there is no \hm\, $H: A\to M(B)$ such that $H_{*0}=r\cdot \lambda.$
Suppose that $\lambda\in {\rm Hom}(K_0(A), \R)_{T_f(A)},$ $\tau\in {\mathcal T}$ such that
$\Lambda(\tau)=\lambda$ and $\tau_0\in {\mathcal T}$ is another essential trivial extension.
Then $\tau\dot{+} \tau_0\in \T$ if and only if $\Lambda([\tau_0])=0,$ namely, $\tau_0$ is 
a trivial diagonal extension.  This shows that ${\mathcal T}$ is not a semigroup.}}

{{Note that, by the UCT, there is a short exact sequence 
\beq
0\to {\rm ext}_{\mathbb Z}(K_*(A), K_{*-1}(B))\to KK(A, \C(B))\to {\rm Hom}(K_*(A), K_*(\C(B)))\to 0.
\eneq}}
Suppose $\tau$ is an essential extension with $\tau_{*1}=0$ and $\tau_{*0}\in {\rm Hom}(K_0(A),\R)_{T_f(A)}.$
One realizes that  $KK(\tau)$ may not be in $N.$

} 
\end{rmk}

\vspace{0.2in}

\noindent
Huaxin Lin: hlin@uoregon.edu.\\

\noindent
Ping Wong  Ng: ping.ng@louisiana.edu.


\begin{thebibliography}{0}

\bibitem{ArvesonDuke}
W. Arveson, \emph{Notes on extensions of C*-algebras.} Duke Math. J.
\textbf{44} (1977),  329--355.

\bibitem{BR} B. Blackadar and M. R\o rdam, \emph{Extending states on preordered  semigroups and existence of 
quasitraces on \CA s}, J. Algebra {\bf 152}, (1992),\, 240-247.

\bibitem{BrownUCT1} L. G. Brown, \emph{Operator algebras and algebraic K-theory.}
Bulletin of the AMS \textbf{81} (1975),  1119--1121.

\bibitem {Brstable}  {{L. G. Brown, {\em Stable isomorphism of hereditary subalgebras of $C^*$-algebras},  Pacific J. Math.  {\bf 71} (1977),  335--348.}}


\bibitem{BrownUCT2} L. G. Brown, \emph{The universal coefficient theorem for $Ext$
and quasidiagonality.}  Operator algebras and group representations, Vol. I (Neptun, 1980),
60--64, Monogr. Stud. Math., 17, Pitman, Boston, MA, 1984.   


\bibitem{BrownInterpolation} L. G. Brown, \emph{Interpolation by projections in
C*-algebras of real rank zero.}  Journal of Operator Theory
\textbf{26} (1991), 383--387.

\bibitem{BDFOriginal} L. G. Brown, R. G. Douglas and P. A. Fillmore,
\emph{Unitary equivalence modulo the compact operators
and extensions of C*-algebras.} Proceedings of a Conference on Operator
Theorem (Dalhouse Univ., Halifax, N.S., 1973), pp58--128. Lecture Notes
in Math., Vol 345, Springer, Berlin, 1973.

\bibitem{BDFAnnals} L. G. Brown, R. G. Douglas and P. A. Fillmore,
\emph{Extensions of C*-algebras and K-homology,} Ann. Math.
\textbf{105} (1977), 265--324.


\bibitem{BrownToms} N. Brown and A. Toms, 
\emph{Three applications of the Cuntz semigroup.} Int. Math. Res. Not.
IMRN \textbf{19} (2007), Art. ID rnm068, 14.

 {{\bibitem{CETWW}  J. Castillejos, S. Evington, A. Tikuisis, S. White, W. Winter
 {\em Nuclear dimension of simple C*-algebras}, preprint,  arXiv:1901.05853.}}

\bibitem{CEI} K. T. Coward, G. A. Elliott and C. Ivanescu.
\emph{The Cuntz semigroup as an invariant for C*-algebras.}
J. Reine Angew. Math. \textbf{623} (2008), 161--193. 



\bibitem{CP}  J. Cuntz and G. K.  Pedersen, {\em Equivalence and 
traces on \CA s},
 J. Funct. Anal.  {\bf 33} (1979),  135--164.
 


\bibitem{HS} P. De la Harpe and G. Skandalis,
\emph{Determinant associe a une trace sur une algebre de Banach.}
Ann. Inst. Fourier \textbf{34} (1984),  241--260.




\bibitem{DE} {{M.~D{\u a}d{\u a}rlat and S.  Eilers, {\em On the classification of nuclear {$C^*$}-algebras},
 Proc. London Math. Soc. {\bf 85} (2002),  168--210.}}


\bibitem
{eglnp}  G. A.   Elliott, G. Gong, H. Lin and Z. Niu,
{\em Simple stably projectionless \CA s
of generalized tracial rank one},   {{J. Non-commutative geometry, to appear.
(arXiv:1711.01240).}}



\bibitem{ElliottGongLinNiu} 
G. Elliott,  G. Gong, H. Lin and Z. Niu,
\emph{The classification of simple separable KK-contractible C*-algebras
with finite nuclear dimension.} Preprint.
A copy is available at
http://arxiv.org/pdf/1712.09463.

\bibitem{ElliottRobertSantiago}
G. Elliott, L. Robert and L. Santiago,
\emph{The cone of lower semicontinuous traces on a C*-algebra.}
Amer. J. Math. \textbf{133} (2011), 969--1005.


\bibitem{GLX}
{{G.~Gong, H.~Lin, and Y.~Xue, \emph{Determinant rank of {C*}-algebras}, Pacific J. Math. {\bf 274} (2015), 405-436.
 }}


\bibitem{GongLinNonunitalAUE}  
G. Gong and H. Lin, \emph{On classification of nonunital simple amenable
C*-algebras, I.} Preprint.  
http://arxiv.org/pdf/1611.044440.

\bibitem{GLII} G. Gong and H. Lin, \emph{On classification  of  non-unital  amenable  simple  C*-algebras, II},
preprint, arXiv:1702.01073.

\bibitem{GLIII} G. Gong and H. Lin, \emph{On classification  of  non-unital  amenable  simple  C*-algebras, III},
a draft.

\bibitem{GLN} G.~Gong, H.~Lin, and Z.~Niu, {\em Classification of finite simple amenable ${\mathcal Z}$-stable $C^*$-algebras},
preprint, arXiv:1501.00135.


\bibitem{GHext} K.  Goodearl and D. Handelman, {\em Rank functions and $K_0$ of regular rings},
 J. Pure Appl. Algebra  {\bf 7} (1976), 195--216.

\bibitem{Halmos10} P. R. Halmos,
\emph{Ten problems in Hilbert space.} Bull. Amer. Math. Soc.
\textbf{76} (1970), 887--933.  



\bibitem{Jacelon}  
B. Jacelon, A simple, monotracial, stably projectionless C*-algebra,
J. Lond. Math. Soc. \textbf{87} (2013), 365--383.

\bibitem{JiangSu}
X. Jiang and H. Su, \emph{On a simple unital projectionless C*-algebra.}
Amer. J. Math. \textbf{121}  (1999), 359--413.


\bibitem{KNZStrictComp} V. Kaftal, P. W. Ng and S. Zhang,
\emph{Strict comparison of positive elements in multiplier algebras.}
Canad. J. of Math., \textbf{69} (2017),  373--407.

\bibitem{KNZMin} V. Kaftal, P. W. Ng and S. Zhang,
\emph{The minimal ideal in a multiplier algebra.} Journal of Operator
Theorem \textbf{79} (2018),  419--462.


\bibitem{KNZPICor} V. Kaftal, P. W. Ng and S. Zhang,
\emph{Purely infinite corona algebras.} Journal of
Operator Theory \textbf{82} (2019),  307--355.






\bibitem{KasparovAbsorb}  G. Kasparov, \emph{Hilbert C*-modules: Theorems
of Stinespring and Voiculesc.} J. Operator Theory \textbf{4} (1980), 133-150.


\bibitem{KirchbergPhillipsEmbed} E. Kirchberg and N. C. Phillips,
\emph{Embedding of exact C*-algebras in the Cuntz algebra $O_2$.}  
J. Reine Angew. Math. \textbf{525} (2000), 17--53.



\bibitem{LinContScaleI} H. Lin,
\emph{Simple C*-algebras with continuous scales and simple corona
algebras.} Proceedings of the AMS. \textbf{112} (1991),
 871--880.

\bibitem{LinExtII} H. Lin,
\emph{Extensions by C*-algebras of real rank zero. II.}
Proc. London Math. Soc. (3) \textbf{71} (1995), 
641--674.


\bibitem{LinWeylVonNeumannII}
H. Lin, \emph{Generalized Weyl--von Neumann theorems. II.} Math. Scand.
\textbf{77} (1995), 129--147.




\bibitem{LinAlmostMult} H. Lin, \emph{Almost multiplicative morphisms and some
applications.} J. Operator Theory \textbf{37} (1997), 193--233. 

\bibitem{LinExtRR0III} H. Lin,
\emph{Extensions by C*-algebras of real rank zero, III.}
Proc. London Math. Soc., \textbf{76} (1998), 634-666.

\bibitem{Lnamj98} {{H.~Lin, \emph{Classification of simple \CA s with  unique traces},
Amer. J. Math., {\bf 120} (1998), 1289--1315.}}


\bibitem{LinSimpleCorona} H. Lin, \emph{Simple corona algebras.}
Proceedings of the American Mathematical Society
\textbf{132} (2004),  3215--3224.



\bibitem{LinExtQuasidiagonal}  H. Lin,
\textit{Extensions by simple C*-algebras:  Quasidiagonal
extensions.} Canad. J.
Math. \textbf{57} (2) (2005), 351--399.


\bibitem{LinFullext} H. Lin, \emph{Full extensions and approximate unitary equivalence},
 Pacific J. Math.  {\bf 229} (2007), 389--428.

 \bibitem{LnTAI} H. Lin, {\em Simple nuclear \CA s of tracial topological rank one}, J. Funct. Anal. {\bf 251} (2007),  601--679.


\bibitem{Lnclasn}
{{H.~Lin, \emph{Asymptotic unitary equivalence and classification of simple
  amenable {C*}-algebras}, Invent. Math. \textbf{183} (2011),  385--450.}}

\bibitem{LinCuntz} H. Lin, \emph{Cuntz semigroups of C*-algebras of stable
rank one and projective Hilbert modules.} Preprint, 2010.


\bibitem{LinExp} H. Lin,
\emph{Exponentials in simple $\Z$-stable C*-algebras.}
J. Funct. Anal. \textbf{266} (2014),  754--791. 


\bibitem{Lnloc}  H.  Lin, {\em Locally AH algebras},  Mem. Amer. Math. Soc. {\bf  235} (2015), no. 1107, vi+109 pp. ISBN: 978-1-4704-1466-5; 978-1-4704-2225-7.

\bibitem{LnHah}
H.~Lin, \emph{Homomorphisms from {AH}-algebras},   J. Topol. Anal. {\bf 9} (2017),  67--125.
(arXiv: 1102.4631v1 (2011).)



\bibitem{LinNg} H. Lin and P. W. Ng,
\emph{The corona algebra of stabilized Jiang--Su algebra.} J. Funct. Anal.
\textbf{270} (2016),  1220--1267.

\bibitem{LR} H. Lin and M. R\o rdam, \emph{Extensions of inductive limits of circle algebras},  J. London Math. Soc. {\bf 51} (1995),  603--613.

\bibitem{NgNonstableAbsorb} P. W. Ng, \emph{Nonstable absorption.}
Houston Journal of Mathematics \textbf{44} (2018), 
975--1017.

\bibitem{NgRR0PICorona}  P. W. Ng,
\emph{Real rank zero for purely infinite corona algebras.}
Preprint. 

\bibitem{NgU(Razak)}  P. W. Ng,
\emph{On the unitary group of the multiplier algebra of 
the Razak algebra.}
Preprint.  Accepted for publication at Studia Mathematica.

\bibitem{NgRobinExtFunctor} P. W. Ng and T. Robin,
\emph{Functorial properties of $\Ext_u(C(X), \B)$ when
$\B$ is simple with continuous scale.}  Accepted for 
publication at the Journal of Operator Theory.


\bibitem{NgRobinQuasidiagonal} P. W. Ng and T. Robin,
\emph{Generalized quasidiagonality for extensions.}  Accepted for publication
at the Banach Journal of Mathematical Analysis. 

\bibitem{NgPICorExt} P. W. Ng,
\emph{Purely infinite corona algebras, and extensions.} Preprint. 


\bibitem{Ped} G. K. Pedersen,
{\em $SAW^*$-algebras  and corona algebras,   contribution to non-commutative topology},
J. Operator Theory, {\bf 15} (1986), 15-32.

 \bibitem{PedU} G. K. Pedersen, {\em Unitary extensions and polar decompositions in a \CA},  J. Operator Theory {\bf 17} (1987),  357--364.


\bibitem{PedersenGPOTS} G. K. Pedersen, The corona construction. Operator
Theory:  Proceedings of the 1988 GPOTS-Wabash Conference (Indianapolis, IN,
1988), 49--92, Pitman Res. Notes, Math. Ser., 225, Longman Sci. Tech.,
Harlow 1990.  

\bibitem{PR} {{F.  Perera and M R\o rdam, {\em AF-embeddings into \CA s of real rank zero},  J. Funct. Anal. {\bf 217} (2004), 142--170.}}

\bibitem{Razak}  S. Razak, \emph{On the classification of simple stably 
projectionless C*-algebras.} Canad. J. Math. \textbf{54} (2002),
138--224.

\bibitem{RieffelSR} M. Rieffel, \emph{Dimension and stable rank in the K-theory of
C*-algebras.} Proc. London Math. Soc. (3), \textbf{46} (1983), 301--333.


\bibitem{RobertTraces} L. Robert \emph{On the comparison of positive elements
of a C*-algebra by lower semicontinuous traces.}  Indiana Univ. Math. J. 
\textbf{58} (2009), 2509--2515.


 \bibitem{RI} L.  Robert, {\em Classification of inductive limits of 1-dimensional NCCW complexes},  Adv. Math.  {\bf 231} (2012),  2802--2836.




\bibitem{Rz} L.   Robert, {\em Remarks on ${\mathcal Z}$-stable projectionless $C^*$-algebras},  Glasg. Math. J. {\bf 58} (2016), 273--277,


\bibitem{RorZRank} M. Rordam, \emph{The stable and real rank of
$\Z$-absorbing C*-algebras.} Internat. J. Math. \textbf{15} (2004),
1065--1084.  


\bibitem{SalinasQuasi}  N. Salinas, \emph{Relative quasidiagonality and
KK-theory.}  Houston J. Math. \textbf{18} (1992), \
97--116.   



\bibitem{Scha} C. Schafhauser, {\em Subalgebras of simple AF-algebras},  preprint,  arXiv: 1807.07381 v2.


\bibitem{SchochetQuasi} C. Schochet, \emph{The fine structure of the
Kasparov groups II:  Relative quasidiagonality.}
J. Operator Theory \textbf{53} (2005),  91--117.


\bibitem{Thomsen}
K.~Thomsen, \emph{Traces, unitary characters and crossed products by {${\mathbb
  Z}$}}, Publ. Res. Inst. Math. Sci. \textbf{31} (1995), 1011--1029.

\bibitem{TomsW07} {{A.~Toms and W.~Winter,
{\itshape Strongly self-absorbing $C^*$-algebras,}
Trans. Amer. Math. Soc. {\bf 359} (2007),  3999-4029.}}


\bibitem{aTz} A.  Tikuisis,  {\em  Nuclear dimension, ${\mathcal Z}$-stability, and algebraic simplicity for stably projectionless
\CA s}, {{Math. Ann.  {\bf 358} (2014),  729--778.}}


\bibitem{Tsang} K-W Tsang, \emph{On the positive tracial cones of simple stably
projectionless C*-algebras.} J. Funct. Anal. \textbf{227} (2005), 188--199.



\bibitem{Voiculescu} 
 D. Voiculescu,
\emph{A noncommutative Weyl--von Neumann theorem.} Rev. Roumaine Math.
Pures Appl. \textbf{21} (1976), 97--113.


\bibitem{VoiculescuQuasi1} D. Voiculescu,
\emph{A note on quasidiagonal C*-algebras and homotopy.}  Duke Math. J.
\textbf{62} (1991), 267--271.


\bibitem{VoiculescuQuasi2}  D. Voiculescu,
\emph{Around quasidiagonal operators.}  Integral Equat. Operator Theory
\textbf{17} (1993), 137--148.




\bibitem{WeggeOlsen} N. E. Wegge--Olsen, K-theory and C*-algebras. Oxford University
Press, Oxford, 1993.

\bibitem{ZhPIRR0} S. Zhang,
\emph{A property of purely infinite simple C*-algebras.} Proceedings of
the AMS \textbf{109} (1990),  717--720.

\bibitem{ZhangWVN} S. Zhang, 
\emph{$K_1$ groups, quasidiagonality, and interpolation by multiplier projections.}
Trans. Amer. Math. Soc. \textbf{325} (1991),  793--818.

\end{thebibliography}
\end{document}